\documentclass[a4paper]{amsart}
\usepackage[utf8]{inputenc}
\usepackage{amssymb}
\usepackage[all]{xy}
\usepackage{enumitem}
\usepackage{mathrsfs,dsfont,mathtools}
\usepackage{nicefrac}

\usepackage[pdftitle={Dualities in equivariant Kasparov theory},
  pdfauthor={Heath Emerson and Ralf Meyer},
  pdfsubject={Mathematics; MSC 19K35, 46L80}
]{hyperref}
\newcommand*{\MRref}[2]{ \href{http://www.ams.org/mathscinet-getitem?mr=#1}{MR \textbf{#1}}}
\newcommand*{\arxiv}[1]{\href{http://www.arxiv.org/abs/#1}{arXiv: #1}}
\newcommand*{\doi}[2]{\href{http://dx.doi.org/#1}{#1}}

\usepackage[lite]{amsrefs}
\usepackage{microtype}

\numberwithin{equation}{section}

\theoremstyle{plain}
\newtheorem{theorem}[equation]{Theorem}
\newtheorem{lemma}[equation]{Lemma}
\newtheorem{proposition}[equation]{Proposition}
\newtheorem{corollary}[equation]{Corollary}
\theoremstyle{definition}
\newtheorem{definition}[equation]{Definition}
\newtheorem{notation}[equation]{Notation}
\theoremstyle{remark}
\newtheorem{remark}[equation]{Remark}
\newtheorem{example}[equation]{Example}

\newcommand*{\KK}{\textup{KK}}
\newcommand*{\Cstarcat}{\mathfrak C^*}
\newcommand*{\KKcat}{\mathfrak{KK}}
\newcommand*{\RKKcat}{\mathfrak{RKK}}
\newcommand*{\RKK}{\textup{RKK}}
\newcommand*{\cRKK}{\mathscr R\textup{KK}}
\newcommand*{\K}{\textup K}%K-theory
\newcommand*{\KO}{\textup{KO}}%real K-theory
\newcommand*{\RK}{\textup{RK}}
\newcommand*{\PD}{\textup{PD}}
\newcommand*{\SPD}{\textup{PD}_2}

\newcommand*{\Ktop}{\textup K^\textup{top}}

\newcommand*{\forget}{\textup{forget}}%forgetful functor
\newcommand*{\ID}{\textup{id}}%identity map
\newcommand*{\lf}{\textup{lf}}%locally finite
\newcommand*{\flip}{\textup{flip}}%tensor flip automorphism

\newcommand*{\VB}{V}%vector bundle used in duality
\newcommand*{\Spinor}{S}%spinor bundle
\newcommand*{\Tubtot}{i}%tubular neighbourhood for diagonal embedding

\newcommand*{\Mult}{\mathcal M}%multiplier algebra
\newcommand*{\Dual}{\mathcal P}%C*-algebra of the dual
\newcommand*{\Spinc}{\textup{Spin}^\textup c}

\newcommand*{\C}{\mathbb C}
\newcommand*{\Z}{\mathbb Z}
\newcommand*{\N}{\mathbb N}
\newcommand*{\R}{\mathbb R}
\newcommand*{\Comp}{\mathbb K}

\newcommand*{\Hils}{\mathcal H}
\newcommand*{\CONT}{\textup C}
\newcommand*{\EG}{\mathcal E}

\newcommand*{\Tvert}{\textup T}% vertical tangent bundle
\newcommand*{\Grd}{\mathcal G}% groupoid
\newcommand*{\Base}{Z}% base space, object space of the groupoid \(\Grd\)
\newcommand*{\Tot}{X} % total space, a space over \(\Base\)
\newcommand*{\Other}{Y}% another space over \(\Base\)
\newcommand*{\grd}{g} % element of \(\Grd\)
\newcommand*{\base}{z}% element of \(\Base\)
\newcommand*{\tot}{x} % element of \(\Tot\)
\newcommand*{\other}{y}% element of \(\Other\)
 
\newcommand*{\UNIT}{\mathds 1}% unit object of a tensor category

\newcommand*{\nb}{\nobreakdash}
\newcommand*{\Cst}{\textup C^*}% C*-algebra
\newcommand*{\Cstr}{\textup C^*_\textup r}% reduced C*-algebra

\newcommand*{\abs}[1]{\lvert#1\rvert}
\newcommand*{\norm}[1]{\lVert#1\rVert}
\newcommand*{\conj}[1]{\overline{#1}}

\newcommand*{\funclass}{\textup V}
\newcommand*{\dualfunclass}{\Lambda}
\newcommand*{\defeq}{\mathrel{\vcentcolon=}}
\newcommand*{\eqdef}{\mathrel{=\vcentcolon}}
\newcommand*{\pt}{\star}
\newcommand*{\comul}{\nabla}
\newcommand*{\blank}{\text\textvisiblespace}
\newcommand*{\inOb}{\mathrel{\in\in}}

\newcommand{\ctau}{\mathcal C_\tau}%Clifford algebra bundle
 
\newcommand{\Fol}{\mathcal F}%foliation

\DeclareMathOperator{\supp}{supp}
\DeclareMathOperator{\Lef}{Lef}
\DeclareMathOperator{\Eul}{Eul}

\begin{document}
\title{Dualities in equivariant Kasparov theory}

\author{Heath Emerson}
\email{hemerson@math.uvic.ca}
\address{Department of Mathematics and Statistics\\
  University of Victoria\\
  PO BOX 3045 STN CSC\\
  Victoria, B.C.\\
  Canada V8W 3P4}

\author{Ralf Meyer}
\email{rameyer@uni-math.gwdg.de}
\address{Mathematisches Institut and Courant Research Centre ``Higher Order Structures''\\
  Georg-August Universit\"at G\"ottingen\\
  Bunsenstra{\ss}e 3--5\\
  37073 G\"ottingen\\
  Germany}

\begin{abstract}
  We study several duality isomorphisms between equivariant bivariant \(\K\)\nb-theory groups, generalising Kasparov's first and second Poincar\'e duality isomorphisms.

  We use the first duality to define an equivariant generalisation of Lefschetz invariants of generalised self-maps.  The second duality is related to the description of bivariant Kasparov theory for commutative \(\Cst\)\nb-algebras by families of elliptic pseudodifferential operators.  For many groupoids, both dualities apply to a universal proper \(\Grd\)\nb-space.  This is a basic requirement for the dual Dirac method and allows us to describe the Baum--Connes assembly map via localisation of categories.
\end{abstract}

\thanks{Heath Emerson was supported by a National Science and Research Council of Canada Discovery grant.  Ralf Meyer was supported by the German Research Foundation (Deutsche Forschungsgemeinschaft (DFG)) through the Institutional Strategy of the University of G\"ottingen.}
\maketitle

\tableofcontents

\section{Introduction}
\label{sec:overview}

The \(\K\)\nb-homology of a smooth compact manifold~\(M\) is naturally isomorphic to the (compactly supported) \(\K\)\nb-theory of its tangent bundle~\(\Tvert M\) via the map that assigns to a \(\K\)\nb-theory class on~\(\Tvert M\) an elliptic pseudodifferential operator with appropriate symbol.  Dually, the \(\K\)\nb-theory of~\(M\) is isomorphic to the (locally finite) \(\K\)\nb-homology of~\(\Tvert M\).  Both statements have bivariant generalisations, which identify Kasparov's group \(\KK_*\bigl(\CONT(M_1),\CONT(M_2)\bigr)\) for two smooth compact manifolds firstly with \(\K^*(\Tvert M_1\times M_2)\), secondly with \(\K^\lf_*(\Tvert M_2\times M_1) \defeq \KK_*(\CONT_0(\Tvert M_2\times M_1),\C)\).

In this article, we consider substantial generalisations of these two duality isomorphisms in which we replace smooth compact manifolds by more general spaces and work equivariantly with respect to actions of locally compact groups or groupoids.  Furthermore, we get duality theorems in twisted bivariant \(\K\)\nb-theory by allowing locally trivial \(\Cst\)\nb-algebra bundles.  Here we mainly develop the abstract theory behind the duality isomorphisms.  Some applications are explained in \cites{Emerson-Meyer:Euler, Emerson-Meyer:Equi_Lefschetz, Emerson-Meyer:Correspondences}.

Outside the world of manifolds, duality statements require a substitute for the tangent space.  Since there is no canonical choice, we follow an axiomatic approach, first finding necessary and sufficient conditions for the duality isomorphisms and then verifying them for the tangent space of a smooth manifold, and for the manifold itself if it is endowed with an equivariant \(\Spinc\)-structure.  We began this study in~\cite{Emerson-Meyer:Euler} with the notions of abstract duals and Kasparov duals, which are related to the first duality isomorphism in the untwisted case.  Here we also examine duality isomorphisms for non-trivial bundles of \(\Cst\)\nb-algebras (this contains computations in~\cite{Echterhoff-Emerson-Kim:Duality} related to duality in twisted bivariant \(\K\)\nb-theory) and consider the second duality isomorphism, both twisted and untwisted.  Moreover, we show that both duality isomorphisms become equivalent in the compact case.  The second duality isomorphism is used in an essential way in~\cite{Emerson-Meyer:Correspondences} to develop a topological model of equivariant Kasparov theory that refines the theory of correspondences due to Alain Connes and Georges Skandalis in~\cite{Connes-Skandalis:Longitudinal}, and to Paul Baum (see~\cite{Baum-Block:Bicycles}).  Both dualities together imply the equivalence of the construction of the Baum--Connes assembly map in~\cite{Baum-Connes-Higson:BC} with the localisation approach of~\cite{Meyer-Nest:BC}.

Our motivation for developing duality was to define and explore a new homotopy invariant of a space equipped with an action of a group or groupoid called the Lefschetz map.  Duality is essential to its definition, while its invariance properties rely on a careful analysis of the functoriality of duals.  In a forthcoming article, we will compute the Lefschetz map in a systematic way using the topological model of \(\KK\)-theory in~\cite{Emerson-Meyer:Correspondences}.

Jean-Louis Tu also formulates similar duality isomorphisms in~\cite{Tu:Novikov}, but there are some technical issues that he disregards, so that assumptions are missing in his theorems.  In particular, the two duality isomorphisms require two different variants of the local dual Dirac element.  In practice, these are often very closely related, as we shall explain, but this only means that both have the same source, not that one could be obtained from the other.

We now explain the contents of this article in more detail.  Let~\(\Grd\) be a locally compact Hausdorff groupoid with Haar system (see~\cite{Paterson:Groupoids}), let~\(\Base\) denote its object space.  Let~\(\Tot\) be a locally compact, proper \(\Grd\)\nb-space.  An \emph{abstract dual} for~\(\Tot\) of dimension \(n\in\Z\) consists of a \(\Grd\)\nb-\(\Cst\)-algebra~\(\Dual\) and a class
\[
\Theta\in\RKK^\Grd_n(\Tot;\CONT_0(\Base), \Dual)
\defeq \KK^{\Grd\ltimes\Tot}_n(\CONT_0(\Tot),\CONT_0(\Tot)\otimes_\Base\Dual)
\]
such that the Kasparov product with~\(\Theta\) induces an isomorphism
\begin{equation}
  \label{eq:intro_first_duality}
  \KK^\Grd_*(\Dual\otimes_\Base A,B) \xrightarrow{\cong} \RKK^\Grd_{*+n}(\Tot;A,B)
\end{equation}
for all \(\Grd\)\nb-\(\Cst\)\nb-algebras \(A\) and~\(B\).  This isomorphism is the \emph{first Poincar\'e duality isomorphism} and is already studied in~\cite{Emerson-Meyer:Euler}.  We may get rid of the dimension~\(n\) by suspending~\(\Dual\), but allowing \(n\neq0\) is useful for many examples.

Let~\(\Tot\) be a bundle of smooth manifolds over~\(\Base\) with a proper and fibrewise smooth action of~\(\Grd\), and let~\(\Tvert\Tot\) be its vertical tangent bundle.  Then \(\Dual \defeq \CONT_0(\Tvert\Tot)\) with a suitable~\(\Theta\) is an abstract dual for~\(\Tot\).  Except for the generalisation to bundles of smooth manifolds, this result is already due to Gennadi Kasparov~\cite{Kasparov:Novikov}*{\S4}.  More generally, if~\(\Tot\) is a bundle of smooth manifolds with boundary, then \(\Dual\defeq \CONT_0(\Tvert\Tot^\circ)\) is an abstract dual for~\(\Tot\), where~\(\Tot^\circ\) is obtained from~\(\Tot\) by attaching an open collar on the boundary.  A more complicated construction in~\cite{Emerson-Meyer:Euler} provides abstract duals for simplicial complexes (here~\(\Grd\) is a group acting simplicially on~\(\Tot\)).  With additional effort, it should be possible to enhance the duality isomorphism in~\cite{Debord-Lescure:K-duality_stratified} to an abstract dual for stratified pseudomanifolds.

To understand the meaning of~\eqref{eq:intro_first_duality}, we specialise to the case where~\(\Base\) is a point, so that~\(\Grd\) is a group, \(\Tot\) is a smooth manifold1 with boundary with a proper, smooth action of~\(\Grd\), and \(A=B=\C\).  In this case, we will establish a duality isomorphism with \(\Dual\defeq\CONT_0(\Tvert\Tot^\circ)\).  The left hand side in the first duality isomorphism~\eqref{eq:intro_first_duality} is the \(\Grd\)\nb-equivariant representable \(\K\)\nb-theory of~\(\Tot\),
\[
\RK^*_\Grd(\Tot) \defeq \RKK^\Grd_{*+n}(\Tot;\C,\C) \defeq
\KK^{\Grd\ltimes\Tot}_{*+n}\bigl(\CONT_0(\Tot),\CONT_0(\Tot)\bigr)
\]
(see~\cite{Emerson-Meyer:Equivariant_K}).  The right hand side of~\eqref{eq:intro_first_duality} is the locally finite \(\Grd\)\nb-equivariant \(\K\)\nb-homology of~\(\Tvert\Tot^\circ\),
\[
\K_*^{\Grd,\lf}(\Tvert\Tot^\circ) \defeq \KK^\Grd_*(\CONT_0(\Tvert\Tot^\circ),\C),
\]
where locally finite means finite on \(\Grd\)\nb-compact subsets in our equivariant setting.  Recall that the \(\Grd\)\nb-equivariant \(\K\)\nb-homology of~\(\Tvert\Tot^\circ\) is the inductive limit
\[
\K_*^\Grd(\Tvert\Tot^\circ) \defeq \varinjlim_\Other \KK^\Grd_*(\CONT_0(\Other),\C),
\]
where~\(\Other\) runs through the directed set of \(\Grd\)\nb-compact subsets of~\(\Tvert\Tot^\circ\).  We also establish a variant of~\eqref{eq:intro_first_duality} that specialises to an isomorphism
\[
\K_*^\Grd(\Tvert\Tot^\circ) \cong \K^*_\Grd(\Tot)
\]
between the \(\Grd\)\nb-equivariant \(\K\)\nb-homology of~\(\Tvert\Tot^\circ\) and the \(\Grd\)\nb-equivariant \(\K\)\nb-theory of~\(\Tot\) (homotopy theorists woud say ``\(\K\)\nb-theory with \(\Grd\)\nb-compact support'').

If the anchor map \(p_\Tot\colon \Tot\to\Base\) is proper, then~\eqref{eq:intro_first_duality} is equivalent to an isomorphism
\[
\KK^\Grd_*(\Dual\otimes_\Base A,B) \xrightarrow{\cong} \KK^\Grd_{*+n}(A,\CONT_0(\Tot)\otimes_\Base B),
\]
that is, to a duality between \(\CONT_0(\Tot)\) and~\(\Dual\) in the tensor category \(\KK^\Grd\) (see Section~\ref{sec:first_duality_compact}).  But in general, abstract duals cannot be defined purely inside \(\KK^\Grd\).

Abstract duals are unique up to \(\KK^\Grd\)\nb-equivalence and covariantly functorial for continuous, \emph{not-necessarily-proper} \(\Grd\)\nb-equivariant maps: if \(\Dual\) and~\(\Dual'\) are abstract duals for two \(\Grd\)\nb-spaces \(\Tot\) and~\(\Tot'\), then a continuous \(\Grd\)\nb-equivariant map \(f\colon \Tot\to\Tot'\) induces a class \(\alpha_f\in \KK^\Grd_*(\Dual,\Dual')\), and \(f\mapsto \alpha_f\) is functorial in a suitable sense.

For instance, if~\(\Tot\) is a universal proper \(\Grd\)\nb-space and \(\Tot'=\Base\), then the canonical projection \(\Tot\to\Base\) induces a class \(\alpha_f\in \KK^\Grd_0\bigl(\Dual,\CONT_0(\Base)\bigr)\).  This plays the role of the Dirac morphism of~\cite{Meyer-Nest:BC} --~in the group case, it \emph{is} the Dirac morphism~-- which is an important ingredient in the description of the Baum--Connes assembly map via localisation of categories.  This example shows that abstract duals allow us to translate constructions from homotopy theory to non-commutative topology.

In contrast, \(\CONT_0(\Tot)\) is contravariantly functorial, and only for proper maps.  Thus the map from the classifying space to a point does not induce anything on \(\CONT_0(\Tot)\).

An abstract dual for a \(\Grd\)\nb-space~\(\Tot\) gives rise to a certain grading preserving group homomorphism
\begin{equation}
  \label{eq:intro:lef_map}
  \Lef\colon
  \RKK^\Grd_*\bigl(\Tot;\CONT_0(\Tot),\CONT_0(\Base)\bigr)
  \to \KK^\Grd_*\bigl(\CONT_0(\Tot),\CONT_0(\Base)\bigr).
\end{equation}
This is the \emph{Lefschetz map} alluded to above.  The functoriality of duals implies that it only depends on the proper \(\Grd\)\nb-homotopy type of~\(\Tot\).  There is a natural map
\[
\KK_*^{\Grd}\bigl(\CONT_0(\Tot), \CONT_0(\Tot)\bigr) \to \RKK^\Grd_*\bigl(\Tot;\CONT_0(\Tot),\CONT_0(\Base)\bigr)
\]
which sends the class of a \(\Grd\)-equivariant proper map \(\varphi\colon \Tot\to\Tot\) to its \emph{graph}
\[
\tilde{\varphi}\colon \Tot \to \Tot\times_\Base \Tot, \qquad \tilde{\varphi}(\tot) \defeq (\tot, \varphi(\tot)).
\]
Since the latter map is proper even if the original map is not, we can think of the domain of \(\Lef\) as not-necessarily-proper \(\KK\)-self-maps of~\(\Tot\).  Combining both maps, we thus get a map
\[
\Lef\colon \KK^\Grd_*\bigl(\CONT_0(\Tot), \CONT_0(\Tot)\bigr) \to \KK^\Grd_*(\CONT_0(\Tot), \CONT(\Base)).
\]
The Euler characteristic of~\(\Tot\) already defined in~\cite{Emerson-Meyer:Euler} is the equivariant Lefschetz invariant of the identity map on~\(\Tot\).  With the specified domain~\eqref{eq:intro:lef_map}, the map \(\Lef\) is split surjective, so that Lefschetz invariants can be arbitrarily complicated.  Usually, the Lefschetz invariants of ordinary self-maps are quite special and can be represented by \(^*\)\nb-homomorphisms to the \(\Cst\)\nb-algebra of compact operators on some graded \(\Grd\)\nb-Hilbert space (see~\cite{Emerson-Meyer:Equi_Lefschetz}).

In many examples of abstract duals, the \(\Grd\)\nb-\(\Cst\)-algebra~\(\Dual\) has the additional structure of a \(\Grd\ltimes\Tot\)-\(\Cst\)-algebra.  In this case, we arrive at explicit conditions for~\eqref{eq:intro_first_duality} to be an isomorphism, which also involve explicit formulas for the inverse of the duality isomorphism~\eqref{eq:intro_first_duality} and for the Lefschetz map (see Theorem~\ref{the:first_duality} and Equation~\eqref{eq:compute_Lef}).  These involve the tensor product functor
\[
T_\Dual\colon \RKK^\Grd_*(\Tot;A,B) \to \KK^\Grd_*(\Dual\otimes_\Base A,\Dual\otimes_\Base B)
\]
and a class \(D\in\KK^\Grd_{-n}\bigl(\Dual,\CONT_0(\Base)\bigr)\) which is determined uniquely by~\(\Theta\).  Since the conditions for the duality isomorphism~\eqref{eq:intro_first_duality} are already formulated implicitly in~\cite{Kasparov:Novikov}*{\S4}, we call this situation \emph{Kasparov duality}.

The formula for the inverse isomorphism to~\eqref{eq:intro_first_duality} makes sense in greater generality: for any \(\Grd\ltimes\Tot\)-\(\Cst\)\nb-algebra~\(A\), we get a canonical map
\begin{equation}
  \label{eq:intro_first_duality_bundle}
  \KK^{\Grd\ltimes\Tot}_*(A,\CONT_0(\Tot)\otimes_\Base B) \to \KK^\Grd_*(A\otimes_\Tot \Dual,B).
\end{equation}
This is a more general situation because~\(A\) is allowed to be a non-trivial bundle over~\(\Tot\).  If~\(A\) is a trivial bundle \(\CONT_0(\Tot,A_0)\), then \(A\otimes_\Tot \Dual \cong A_0\otimes\Dual\) and the isomorphism in~\eqref{eq:intro_first_duality_bundle} is the inverse map to~\eqref{eq:intro_first_duality}.  It is shown in~\cite{Echterhoff-Emerson-Kim:Duality} that the map~\eqref{eq:intro_first_duality_bundle} is an isomorphism in some cases, but not always; this depends on whether or not the bundle~\(A\) is locally trivial in a sufficiently strong (equivariant) sense.  Theorem~\ref{the:KK_first_non-trivial} provides a necessary and sufficient condition for~\eqref{eq:intro_first_duality_bundle} to be an isomorphism.

We verify these conditions for the tangent duality if~\(\Tot\) is a bundle of smooth manifolds with boundary and the bundle~\(A\) is strongly locally trivial.  Let~\(A\) be a continuous trace algebra with spectrum~\(\Tot\) and a sufficiently nice \(\Grd\)\nb-action, so that duality applies, and let \(B=\CONT_0(\Base)\) in \eqref{eq:intro_first_duality_bundle}.  Let~\(A^*\) be the inverse of~\(A\) in the \(\Grd\)\nb-equivariant Brauer group of~\(\Tot\), that is, \(A\otimes_\Tot A^*\) is \(\Grd\)\nb-equivariantly Morita equivalent to \(\CONT_0(\Tot)\).  Now the left hand side in~\eqref{eq:intro_first_duality_bundle} may be interpreted as the \(\Grd\)\nb-equivariant twisted representable \(\K\)\nb-theory of~\(\Tot\) with twist~\(A^*\) because tensoring with~\(A^*\) provides an isomorphism
\[
\KK^{\Grd\ltimes\Tot}_*\bigl(A,\CONT_0(\Tot)\bigr) \cong \KK^{\Grd\ltimes\Tot}_*(\CONT_0(\Tot),A^*).
\]
On the other hand, the right hand side of \eqref{eq:intro_first_duality_bundle} is \(\KK^\Grd_*\bigl(A\otimes_\Tot \CONT_0(\Tvert\Tot^\circ),\CONT_0(\Base)\bigr)\), that is, the locally finite \(\Grd\)\nb-equivariant twisted \(\K\)\nb-homology of~\(\Tvert\Tot^\circ\), where the twist is given by the pull-back of~\(A\).

Thus~\eqref{eq:intro_first_duality_bundle} contains a twisted and groupoid-equivariant version of the familiar isomorphism \(\K^*(\Tot)\cong \K_*^\lf(\Tvert\Tot)\) for smooth compact manifolds (see also~\cite{Tu:Twisted_Poincare}).

In addition, we get a canonical map
\[
\varinjlim \KK^{\Grd\ltimes\Tot}_*(A|_\Other,\CONT_0(\Tot)\otimes_\Base B) \to
\varinjlim \KK^\Grd_*(A|_\Other\otimes_\Tot \Dual,B),
\]
where~\(\Other\) runs through the directed set of \(\Grd\)\nb-compact subsets of~\(\Tot\).  If~\(\Tot\) is a bundle of smooth manifolds with boundary, \(A\) is strongly locally trivial, and~\(\Dual\) is \(\CONT_0(\Tvert\Tot^\circ)\), then the latter map is an isomorphism as well.  This specialises to an isomorphism \(\K^\Grd_*(\Tvert\Tot^\circ) \cong \K_\Grd^*(\Tot)\) for \(A=\CONT_0(\Tot)\) and \(B=\CONT_0(\Base)\).

We now discuss the second duality isomorphism, which generalises the isomorphism \(\K_*(\Tot)\cong \K^*(\Tvert\Tot)\) for smooth compact manifolds.  Kasparov only formulates it for compact manifolds with boundary (see \cite{Kasparov:Novikov}*{Theorem 4.10}), and it is not obvious how best to remove the compactness assumption.  We propose to consider the canonical map
\begin{equation}
  \label{eq:intro_second_duality_bundle}
  \KK^{\Grd\ltimes\Tot}_{*+n}(A,B\otimes_\Base \Dual) \to
  \KK^\Grd_*(A,B),
\end{equation}
which first forgets the \(\Tot\)\nb-structure and then composes with \(D\in\KK^\Grd_{-n}\bigl(\Dual,\CONT_0(\Base)\bigr)\).  Here~\(A\) is a \(\Grd\ltimes\Tot\)-\(\Cst\)\nb-algebra and~\(B\) is a \(\Grd\)\nb-\(\Cst\)\nb-algebra.  For instance, if \(A= \CONT_0(\Tot)\) and \(B=\CONT_0(\Base)\), then this becomes a map
\[
\KK^{\Grd\ltimes\Tot}_{*+n}(\CONT_0(\Tot),\Dual) \to \KK^\Grd_*\bigl(\CONT_0(\Tot),\CONT_0(\Base)\bigr),
\]
that is, the right hand side is the \(\Grd\)\nb-equivariant locally finite \(\K\)\nb-homology of~\(\Tot\).  For the tangent duality \(\Dual = \CONT_0(\Tvert\Tot^\circ)\), the left hand side is, by definition, the \(\Grd\)\nb-equivariant \(\K\)\nb-theory of~\(\Tvert\Tot^\circ\) with \(\Tot\)\nb-compact support (see~\cite{Emerson-Meyer:Equivariant_K}).

Theorem~\ref{the:second_duality_general} provides a necessary and sufficient condition for~\eqref{eq:intro_second_duality_bundle} to be an isomorphism.  It is shown in Section~\ref{sec:tangent_dual} that these conditions hold for the tangent dual of a bundle~\(\Tot\) of smooth manifolds with boundary.  Hence~\eqref{eq:intro_second_duality_bundle} specialises to an isomorphism \(\K_*^{\Grd,\lf}(\Tot) \cong \RK^*_{\Grd,\Tot}(\Tvert\Tot^\circ)\) between the \(\Grd\)\nb-equivariant locally finite \(\K\)\nb-homology of~\(\Tot\) and the \(\Grd\)\nb-equivariant \(\K\)\nb-theory of~\(\Tvert\Tot^\circ\) with \(\Tot\)\nb-compact support.

As in the first duality isomorphism, we get a version of the second duality isomorphism in twisted equivariant \(\K\)\nb-theory if we allow~\(A\) to be a strongly locally trivial \(\Grd\)\nb-equivariant bundle of \(\Cst\)\nb-algebras over~\(\Tot\): the twisted \(\Grd\)\nb-equivariant locally finite \(\K\)\nb-homology of~\(\Tot\) with twist~\(A\) is isomorphic to the twisted \(\Grd\)\nb-equivariant \(\K\)\nb-theory of~\(\Tvert\Tot^\circ\) with \(\Tot\)\nb-compact support and with twist~\(A^*\).  And we get a version with different support conditions:
\begin{equation}
  \label{eq:intro_second_duality_bundle_local}
  \varinjlim \KK^{\Grd\ltimes\Tot}_{*+n}(A|_\Other,B\otimes_\Base \Dual) \to
  \varinjlim \KK^\Grd_*(A|_\Other,B),
\end{equation}
where~\(\Other\) runs through the directed set of \(\Grd\)\nb-compact subsets of~\(\Tot\).  For bundles of smooth manifolds, the latter specialises to an isomorphism
\[
\K^*_\Grd(\Tvert\Tot^\circ) \cong \K_*^\Grd(\Tot).
\]

The duality isomorphism~\eqref{eq:intro_second_duality_bundle} for the tangent space dual specialises to an isomorphism
\begin{equation}
  \label{eq:symbols_duality}
  \RK^\Grd_{*,\Tot}(\Other\times_\Base \Tvert\Tot^\circ)
  \cong \KK^\Grd_*\bigl(\CONT_0(\Tot),\CONT_0(\Other)\bigr).
\end{equation}
The first group is the home for symbols of families of elliptic pseudodifferential operators on~\(\Tot\) parametrised by~\(\Other\).  There are different formulas for this isomorphism, using a topological index map or the family of Dolbeault operators along the fibres of~\(\Tvert\Tot^\circ\to\Base\), or pseudodifferential calculus.  These are based on different formulas for the class~\(D\) involved in the duality isomorphisms.  Since~\(D\) is determined uniquely by~\(\Theta\), for which there is only one reasonable geometric formula, duality isomorphisms also contain equivariant index theorems such as Kasparov's Index Theorem~\cite{Kasparov:Invariants_elliptic}.

The duality isomorphism~\eqref{eq:symbols_duality} is the crucial step in the geometric description of Kasparov theory in~\cite{Emerson-Meyer:Correspondences}.  The question when this can be done remained unexamined since Paul Baum introduced his bicycles or correspondences some thirty years ago.  Even for the special case of \(\KK_*(\CONT_0(\Tot),\C)\) for a finite CW-complex~\(\Tot\), a detailed proof appeared only recently in~\cite{Baum-Higson-Schick:Equivalence}.  In~\cite{Emerson-Meyer:Correspondences}, we prove that the equivariant Kasparov groups \(\KK^\Grd_*\bigl(\CONT_0(\Tot), \CONT_0(\Other)\bigr)\) can be described in purely topological terms if~\(\Tot\) is a bundle over~\(\Base\) of smooth manifolds with boundary with a smooth action of a proper groupoid~\(\Grd\) and some conditions regarding equivariant vector bundles are met.

The basic idea of the argument is that an isomorphism similar to~\eqref{eq:symbols_duality} exists in the geometric theory.  Thus the problem of identifying geometric and analytic \emph{bivariant} \(\K\)\nb-theory reduces to the problem of identifying corresponding \emph{monovariant} \(\K\)\nb-theory groups with some support conditions.  This becomes trivial with an appropriate definition of the geometric cycles.  Besides finding this appropriate definition, the main work in~\cite{Emerson-Meyer:Correspondences} is necessary to equip the geometric version of \(\KK\) with all the extra structure that is needed to get duality isomorphisms.  Our analysis here already shows what is involved: composition and exterior products and certain pull-back and forgetful functors.

Let~\(\Tot\) be a proper \(\Grd\)\nb-space and let~\(B\) be a \(\Grd\ltimes\Tot\)-\(\Cst\)\nb-algebra.  Then \cite{Emerson-Meyer:Equivariant_K}*{Theorem 4.2} implies a natural isomorphism
\begin{equation}
  \label{eq:RKK_via_crossed}
  \varinjlim \KK^{\Grd\ltimes\Tot}_*(\CONT_0(\Other), B) \cong \K_*(\Grd\ltimes B),
\end{equation}
where the inductive limit runs over the directed set of \(\Grd\)\nb-compact subsets of~\(\Tot\) as in~\eqref{eq:intro_second_duality_bundle_local}.  Equations \eqref{eq:intro_second_duality_bundle} and~\eqref{eq:RKK_via_crossed} yield a duality isomorphism
\begin{equation}
  \label{eq:above_isomorphism}
  \varinjlim \KK^\Grd_*(\CONT_0(\Other),B) \cong
  \K_{*+n}\bigl(\Grd\ltimes(\Dual \otimes_\Base B)\bigr).
\end{equation}

If~\(\Tot\) is also a universal proper \(\Grd\)\nb-space, then our duality isomorphisms are closely related to the different approaches to the Baum--Connes assembly map for~\(\Grd\).  The dual Dirac method is the main ingredient in most proofs of the Baum--Connes and Novikov Conjectures; this goes back to Gennadi Kasparov \cites{Kasparov:Novikov, Kasparov-Skandalis:Bolic}, who used the first Poincar\'e duality isomorphism to prove the Novikov Conjecture for discrete subgroups of almost connected groups.  We will see that the first duality isomorphism for a groupoid~\(\Grd\) acting on its classifying space is essentially equivalent to constructing a Dirac morphism for the groupoid, which is the (easier) half of the dual Dirac method.  The remaining half, the dual Dirac morphism, is a lifting of \(\Theta\in \RKK^{\Grd\ltimes\Tot}_*\bigl(\CONT_0(\Tot),\CONT_0(\Tot)\bigr)\) to \(\KK^\Grd_*\bigl(\CONT_0(\Base),\CONT_0(\Base)\bigr)\).  Even if such a lifting does not exist, the Dirac morphism \(D\in\KK^\Grd_{-n}\bigl(\Dual,\CONT_0(\Base)\bigr)\) is exactly what is needed for the localisation approach to the Baum--Connes assembly map in~\cite{Meyer-Nest:BC}.

The isomorphism~\eqref{eq:above_isomorphism} relates two approaches to the Baum--Connes assembly map with coefficients.  The left hand side is the topological \(\K\)\nb-theory defined in~\cite{Baum-Connes-Higson:BC}, whereas the right hand side is the topological \(\K\)\nb-theory in the localisation approach of~\cite{Meyer-Nest:BC}.  This is exactly the \(\gamma\)\nb-part of \(\K_*(\Grd\ltimes B)\) provided~\(\Grd\) has a \(\gamma\)\nb-element.

\smallskip

Finally, we describe the contents of the following sections.  Sections \ref{sec:groupoid_preliminaries} and~\ref{sec:grd_KK} contain preparatory remarks on groupoids, their actions on spaces and \(\Cst\)\nb-algebras, and equivariant Kasparov theory for groupoids.  We pay special attention to tensor product functors because these will play a crucial role.

Section~\ref{sec:def_Kasparov_dual} deals with the first Poincar\'e duality isomorphism and related constructions.  We introduce abstract duals and Kasparov duals and construct equivariant Euler characteristics and Lefschetz maps from them.  We explain how the first duality is related to Dirac morphisms and thus to the Baum--Connes assembly map, and we provide a necessary and sufficient condition for the first duality isomorphism to extend to non-trivial bundles, formalising an example considered in~\cite{Echterhoff-Emerson-Kim:Duality}.

Section~\ref{sec:first_duality_compact} studies Kasparov duality for bundles of compact spaces.  In this case, the first and second kind of duality are both equivalent to a more familiar notion of duality studied already by Georges Skandalis in~\cite{Skandalis:KK_survey}.

Section~\ref{sec:second_duality} treats the second duality isomorphism.  We introduce symmetric Kasparov duals, which guarantee both duality isomorphisms for trivial bundles.

In Section~\ref{sec:tangent_dual}, we construct symmetric Kasparov duals for bundles of smooth manifolds with boundary.  For a single smooth manifold, this example is already considered in~\cite{Kasparov:Novikov}.  We also extend the duality isomorphisms to certain locally trivial \(\Cst\)\nb-algebra bundles as coefficients.

\subsection{Some standing assumptions}
\label{sec:standing_assumptions}

To avoid technical problems, we tacitly assume all \(\Cst\)\nb-algebras to be separable, and all topological spaces to be locally compact, Hausdorff, and second countable.  Groupoids are tacitly required to be locally compact, Hausdorff, and second countable and to have a Haar system.

Several constructions of Kasparov duals contain Clifford algebras and hence yield \(\Z/2\)\nb-graded \(\Cst\)\nb-algebras.  Therefore, we tacitly allow all \(\Cst\)\nb-algebra to carry a \(\Z/2\)-grading; that is, ``\(\Cst\)\nb-algebra'' stands for ``\(\Z/2\)\nb-graded \(\Cst\)\nb-algebra'' throughout.

The general theory in Sections \ref{sec:groupoid_preliminaries}--\ref{sec:second_duality} is literally the same for complex, real, and ``real'' \(\Cst\)\nb-algebras (although we usually presume in our notation that we are in the complex case).  The construction of duals for bundles of smooth manifolds in Section~\ref{sec:tangent_dual} also works for these three flavours, with some small modifications that are pointed out where relevant.  Most importantly, the results about tangent space duality claimed above only hold if the tangent bundle is equipped with a ``real'' structure or replaced by another vector bundle.

\section{Preliminaries on groupoid actions}
\label{sec:groupoid_preliminaries}

We recall some basic notions regarding groupoids and their actions on spaces and \(\Cst\)\nb-algebras to fix our notation.  We pay special attention to tensor product operations and their formal properties, which are expressed in the language of symmetric monoidal categories (see \cites{MacLane:Categories, Saavedra:Tannakiennes, Pareigis:C-categories}).  This framework is particularly suited to the first Poincar\'e duality isomorphism.

\subsection{Groupoids and their actions on spaces}
\label{sec:grd_act_Cstar}

Let~\(\Grd\) be a (locally compact) groupoid.  We write \(\Grd^{(0)}\) and \(\Grd^{(1)}\) for the spaces of objects and morphisms in~\(\Grd\) and \(r,s\colon \Grd^{(1)}\rightrightarrows \Grd^{(0)}\) for the range and source maps.

\begin{definition}
  \label{def:space_over}
  Let~\(\Base\) be a (locally compact, Hausdorff, second countable topological) space.  A \emph{space over~\(\Base\)} is a continuous map \(f\colon \Tot\to \Base\).  If~\(f\) is clear from the context, we also call~\(\Tot\) itself a space over~\(\Base\).
\end{definition}

\begin{definition}
  \label{def:times_over}
  Let \(f\colon \Tot\to \Base\) and \(g\colon \Other\to\Base\) be spaces over~\(\Base\).  Their \emph{fibred product} is
  \[
  \Tot \times_{f,g} \Other \defeq \{(\tot,\other)\in \Tot\times\Other\mid f(\tot)=g(\other)\}
  \]
  with the subspace topology and the continuous map \((\tot,\other)\mapsto f(\tot)=g(\other)\).  Thus \(\Tot \times_{f,g} \Other\) is again a space over~\(\Base\).  If \(f,g\) are clear from the context, we also write \(\Tot \times_\Base \Other\) instead of \(\Tot\times_{f,g} \Other\).
\end{definition}

\begin{definition}
  \label{def:groupoid_action}
  A \emph{\(\Grd\)\nb-space} is a space \((\Tot,\pi)\) over~\(\Grd^{(0)}\) with a homeomorphism
  \[
  \Grd^{(1)}\times_{s,\pi} \Tot \to \Grd^{(1)}\times_{r,\pi} \Tot, \qquad (\grd,\tot)\mapsto (\grd\cdot\tot,\tot),
  \]
  subject to the usual associativity and unitality conditions.
\end{definition}

\begin{example}
  \label{def:groups_as_groupoids}
  If~\(G\) is a group then \(\Grd \defeq G\) is a groupoid with \(\Base = \{\star\}\), and \(G\)\nb-spaces have the usual meaning.
\end{example}

\begin{example}
  \label{exa:space_action}
  View a space~\(\Base\) as a groupoid with only identity morphisms, that is, \(\Base^{(1)}=\Base^{(0)}=\Base\).  A \(\Base\)\nb-space is nothing but a space over~\(\Base\).
\end{example}

\begin{definition}
  \label{def:transformation_groupoid}
  If~\(\Base\) is a \(\Grd\)\nb-space, then the \emph{transformation groupoid} \(\Grd\ltimes \Base\) is the groupoid with \((\Grd\ltimes \Base)^{(0)}\defeq \Base\),
  \begin{gather*}
    (\Grd\ltimes \Base)^{(1)} \defeq \Grd^{(1)}\times_{s,\pi} \Base \cong \{(\base_1,\grd,\base_2)\in \Base\times_{\pi,r} \Grd^{(1)}\times_{s,\pi} \Base\mid
    \base_1=\grd\cdot\base_2\},\\
    r(\base_1,\grd,\base_2) \defeq \base_1, \quad s(\base_1,\grd,\base_2) \defeq \base_2, \qquad (\base_1,\grd,\base_2)\cdot (\base_2,h,\base_3)\defeq (\base_1,\grd\cdot h,\base_3).
  \end{gather*}
  This groupoid inherits a Haar system from~\(\Grd\).
\end{definition}

\begin{lemma}
  \label{lem:space_over_transformation_groupoid}
  A \(\Grd\ltimes \Base\)-space is the same as a \(\Grd\)-space~\(\Tot\) with a \(\Grd\)\nb-equivariant continuous map \(p\colon \Tot\to \Base\).
\end{lemma}

Hence we call \(\Grd\ltimes \Base\)-spaces \emph{\(\Grd\)-spaces over~\(\Base\)}.  We are going to study duality in bivariant \(\K\)\nb-theory for a \(\Grd\)\nb-space \(p\colon \Tot\to\Base\) over~\(\Base\) or, equivalently, for a \(\Grd\ltimes\Base\)-space.  Since we lose nothing by replacing~\(\Grd\) by \(\Grd\ltimes\Base\), we may assume from now on that \(\Base=\Grd^{(0)}\) to simplify our notation.  Thus, when we study duality for bundles of spaces over some base space~\(\Base\) then this bundle structure is hidden in the groupoid variable~\(\Grd\).

\subsection{\texorpdfstring{$\Cst$}{C*}-algebras over a space}
\label{sec:Cstar_space}

Let~\(\Base\) be a space.  There are several equivalent ways to define \(\Cst\)\nb-algebras over~\(\Base\).

\begin{definition}
  \label{def:Cstar_over}
  A \emph{\(\Cst\)\nb-algebra over~\(\Base\)} is a \(\Cst\)\nb-algebra~\(A\) together with an essential \(^*\)\nb-homomorphism~\(\varphi\) from \(\CONT_0(\Base)\) to the centre of the multiplier algebra of~\(A\); being \emph{essential} means that \(\varphi\bigl(\CONT_0(\Base)\bigr)\cdot A=A\); equivalently, \(\varphi\) extends to a strictly continuous unital \(^*\)\nb-homomorphism on \(\CONT_b(\Base)\).
\end{definition}

The map~\(\varphi\) is equivalent to a continuous map from the primitive ideal space of~\(A\) to~\(\Base\) by the Dauns--Hofmann Theorem (see~\cite{Nilsen:Bundles}).  Any \(\Cst\)\nb-algebra over~\(\Base\) is the \(\Cst\)\nb-algebra of \(\CONT_0\)\nb-sections of an upper semi-continuous \(\Cst\)\nb-bundle over~\(\Base\) by~\cite{Nilsen:Bundles}, and conversely such section algebras are \(\Cst\)\nb-algebras over~\(\Base\).  We may also describe a \(\Cst\)\nb-algebra over~\(\Base\) by the \(A\)\nb-linear essential \(^*\)\nb-homomorphism
\begin{equation}
  \label{eq:multiplication_map_over_Base}
  m\colon \CONT_0(\Base,A) \to A,
  \qquad f\otimes a\mapsto
  \varphi(f)\cdot a = a\cdot \varphi(f),
\end{equation}
called \emph{multiplication homomorphism}.  This \(^*\)\nb-homomorphism exists because \(\CONT_0(\Base,A)\) is the maximal \(\Cst\)\nb-tensor product of \(\CONT_0(\Base)\) and~\(A\), and it determines~\(\varphi\).

\begin{example}
  \label{exa:space_Cstar_over_base}
  If \(p\colon \Tot\to\Base\) is a space over~\(\Base\), then \(\CONT_0(\Tot)\) with \(p^*\colon \CONT_0(\Base)\to\CONT_b(\Tot)\) is a commutative \(\Cst\)\nb-algebras over~\(\Base\).  Any commutative \(\Cst\)\nb-algebra over~\(\Base\) is of this form.  The multiplication homomorphism
  \[
  m\colon \CONT_0\bigl(\Base,\CONT_0(\Tot)\bigr) \cong \CONT_0(\Base\times\Tot)\to\CONT_0(\Tot)
  \]
  is induced by the proper continuous map \(\Tot\to\Base\times\Tot\), \(\tot\mapsto \bigl(p(\tot),\tot\bigr)\).
\end{example}

\begin{definition}
  \label{def:Base_linear}
  Let \(A\) and~\(B\) be \(\Cst\)\nb-algebras over~\(\Base\) with multiplication homomorphisms \(m_A\colon \CONT_0(\Base,A)\to A\) and \(m_B\colon \CONT_0(\Base,B)\to B\).  A \(^*\)\nb-homomorphism \(f\colon A\to B\) is called \emph{\(\CONT_0(\Base)\)-linear} or \emph{\(\Base\)\nb-equivariant} if the following diagram commutes:
  \[
  \xymatrix@C+1.5em{ \CONT_0(\Base,A) \ar[r]^{\CONT_0(\Base,f)} \ar[d]_{m_A} &
    \CONT_0(\Base,B) \ar[d]^{m_B} \\
    A \ar[r]^{f} & B }
  \]
\end{definition}

\begin{definition}
  \label{def:Cstarcat_Base}
  We let~\(\Cstarcat_\Base\) be the category whose objects are the \(\Cst\)\nb-algebras over~\(\Base\) and whose morphisms are the \(\CONT_0(\Base)\)-linear \(^*\)\nb-homomorphisms.
\end{definition}

\begin{definition}
  \label{def:restrict_Cstar_over_Base}
  Let \(A\) be a \(\Cst\)\nb-algebras over~\(\Base\) and let \(S\subseteq \Base\) be a subset.  If~\(S\) is closed or open, then we define a \emph{restriction functor} \(\blank|_S\colon \Cstarcat_\Base\to\Cstarcat_S\):
  \begin{itemize}
  \item If~\(S\) is open, then \(A|_S\) is the closed \(^*\)\nb-ideal \(\CONT_0(S)\cdot A\) in~\(A\), equipped with the obvious structure of \(\Cst\)\nb-algebra over~\(S\).

  \item If~\(S\) is closed, then \(A|_S\) is the quotient of~\(A\) by the ideal \(A|_{\Base\setminus S}\), equipped with the induced structure of \(\Cst\)\nb-algebra over~\(S\).
  \end{itemize}
  We abbreviate \(A_\base \defeq A|_{\{\base\}}\) for \(\base\in\Base\).
\end{definition}

If \(S_1\subseteq S_2\subseteq\Base\) are both closed or both open in~\(\Base\), then we have a natural isomorphism \((A|_{S_2})|_{S_1} \cong A|_{S_1}\).

\begin{definition}
  \label{def:pullback_Cstar_over_Base}
  Let \(f\colon \Base'\to \Base\) be a continuous map.  Then we define a \emph{base change functor} \(f^*\colon \Cstarcat_\Base \to \Cstarcat_{\Base'}\).  Let~\(A\) be a \(\Cst\)\nb-algebra over~\(\Base\).  Then \(\CONT_0(\Base',A)\) is a \(\Cst\)\nb-algebra over \(\Base'\times \Base\).  The graph of~\(f\) is a closed subset \(\Gamma(f)\) of \(\Base'\times \Base\) and homeomorphic to~\(\Base'\) via \(\base\mapsto \bigl(\base,f(\base)\bigr)\).  We let \(f^*(A)\) be the restriction of \(\CONT_0(\Base',A)\) to \(\Gamma(f)\), viewed as a \(\Cst\)\nb-algebra over~\(\Base'\).  It is clear that this construction is natural, that is, defines a functor \(f^*\colon \Cstarcat_\Base \to \Cstarcat_{\Base'}\).
\end{definition}

\begin{lemma}
  \label{lem:pull-back_universal}
  Let \(f\colon \Base'\to \Base\) be a continuous map, let~\(A\) be a \(\Cst\)\nb-algebra over~\(\Base\) and let~\(B\) be a \(\Cst\)\nb-algebra.  Then essential \(^*\)\nb-homomorphisms \(f^*(A)\to\Mult(B)\) correspond bijectively to pairs of commuting essential \(^*\)\nb-homomorphisms \(\pi\colon A\to\Mult(B)\) and \(\varphi\colon \CONT_0(\Base')\to\Mult(B)\) that satisfy \(\varphi(h\circ f)\cdot \pi(a) = \pi(h\cdot a)\) for all \(h\in\CONT_0(\Base)\), \(a\in A\).
\end{lemma}

This universal property characterises the base change functor uniquely up to natural isomorphism and implies the following properties:

\begin{lemma}
  \label{lem:restrict_pullback}
  If \(f\colon S\to \Base\) is the embedding of an open or closed subset, then \(f^*(A)\) is naturally isomorphic to~\(A|_S\).

  We have \((g\circ f)^* = g^*\circ f^*\) for composable maps \(\Base''\xrightarrow{f} \Base'\xrightarrow{g} \Base\), and \(\ID_\Base^*\) is equivalent to the identity functor.  In particular, \(f^*(A)_\base \cong A_{f(\base)}\).
\end{lemma}

\begin{notation}
  \label{note:tensor_Cstar_over}
  Let \(A\) and~\(B\) be \(\Cst\)\nb-algebras over~\(\Base\).  Then \(A\otimes B\) is a \(\Cst\)\nb-algebra over \(\Base\times \Base\).  We let \(A\otimes_\Base B\) be its restriction to the diagonal in \(\Base\times \Base\).
\end{notation}

\begin{example}
  \label{exa:pull-back_space}
  Let \((\Tot,p)\) be a space over~\(\Base\).  If \(S\subseteq \Base\), then restriction yields \(\CONT_0(\Tot)|_S = \CONT_0\bigl(p^{-1}(S)\bigr)\) as a space over~\(S\).

  Now let \(f\colon \Base'\to \Base\) be a continuous map.  Then
  \[
  f^*\bigl( \CONT_0(\Tot)\bigr) \cong \CONT_0(\Tot\times_{p,f} \Base').
  \]
  In particular, \(f^*\bigl(\CONT_0(\Base)\bigr) \cong \CONT_0(\Base')\).

  We have \(\CONT_0(\Tot_1)\otimes_\Base \CONT_0(\Tot_2) \cong \CONT_0(\Tot_1\times_\Base \Tot_2)\) if \(\Tot_1\) and~\(\Tot_2\) are two spaces over~\(\Base\).
\end{example}

The properties of the tensor product~\(\otimes_\Base\) are summarised in Lemma~\ref{lem:Cstarcat_smc} below.  For the time being, we note that it is a bifunctor and that it is compatible with the functors~\(f^*\): if \(f\colon \Base'\to\Base\) is a continuous map, then there is a natural isomorphism
\[
f^*(A\otimes_\Base B) \cong f^*(A) \otimes_{\Base'} f^*(B)
\]
because both sides are naturally isomorphic to restrictions of \(\CONT_0(\Base'\times\Base')\otimes A\otimes B\) to the same copy of~\(\Base'\) in \(\Base'\times \Base'\times \Base\times\Base\).

\subsection{Groupoid actions on \texorpdfstring{$\Cst$}{C*}-algebras and tensor products}
\label{sec:grp_act_Cstar}

Let~\(\Grd\) be a groupoid with object space \(\Base\defeq \Grd^{(0)}\).

\begin{definition}
  \label{def:Grd-Cstar-algebra}
  Let~\(A\) be a \(\Cst\)\nb-algebra~\(A\) over~\(\Base\) together with an isomorphism \(\alpha\colon s^*(A) \xrightarrow{\cong} r^*(A)\) of \(\Cst\)\nb-algebras over~\(\Grd^{(1)}\).  Let~\(A_\base\) for \(\base\in\Base\) denote the fibres of~\(A\) and let \(\alpha_g\colon A_{s(g)}\xrightarrow{\cong} A_{r(g)}\) for \(g\in\Grd^{(1)}\) be the fibres of~\(\alpha\).  We call \((A,\alpha)\) a \emph{\(\Grd\)\nb-\(\Cst\)\nb-algebra} if \(\alpha_{g_1g_2} = \alpha_{g_1}\alpha_{g_2}\) for all \(g_1,g_2\in\Grd^{(1)}\).
\end{definition}

\begin{definition}
  \label{def:Grd-equivariant}
  A \(^*\)\nb-homomorphism \(\varphi\colon A\to B\) between two \(\Grd\)\nb-\(\Cst\)\nb-algebras is called \emph{\(\Grd\)\nb-equivariant} if it is \(\CONT_0(\Base)\)-linear and the diagram
  \[
  \xymatrix@C+1em{
    s^*(A) \ar[d]^{\cong}_\alpha \ar[r]^{s^*(\varphi)} &
    s^*(B) \ar[d]^{\cong}_\beta \\
    r^*(A) \ar[r]^{r^*(\varphi)} & r^*(B) }
  \]
  commutes.  We let \(\Cstarcat_\Grd\) be the category whose objects are the \(\Grd\)\nb-\(\Cst\)\nb-algebras and whose morphisms are the \(\Grd\)\nb-equivariant \(^*\)\nb-homomorphisms.
\end{definition}

This agrees with our previous definitions if~\(\Grd\) is a space viewed as a groupoid with only identity morphisms.

The tensor product over~\(\Base\) of two \(\Grd\)\nb-\(\Cst\)\nb-algebras carries a canonical action of~\(\Grd\) called \emph{diagonal action}.  Formally, this is the composite of the \(^*\)\nb-isomorphisms
\[
s^*(A\otimes_{\Grd^{(0)}} B) \xrightarrow{\cong} s^*(A)\otimes_{\Grd^{(1)}} s^*(B) \xrightarrow{\alpha\otimes_{\Grd^{(1)}} \beta} r^*(A)\otimes_{\Grd^{(1)}} r^*(B) \xrightarrow{\cong} r^*(A\otimes_{\Grd^{(0)}} B).
\]

\begin{notation}
  \label{note:abbreviate_tensor}
  The resulting tensor product operation on \(\Grd\)\nb-\(\Cst\)\nb-algebras is denoted by~\(\otimes_\Grd\).  We usually abbreviate~\(\otimes_\Grd\) to~\(\otimes\) and also write~\(\otimes_\Base\).
\end{notation}

\begin{lemma}
  \label{lem:Cstarcat_smc}
  The category \(\Cstarcat_\Grd\) with the tensor product~\(\otimes\) is a symmetric monoidal category with unit object \(\CONT_0(\Base)\).
  \label{pro:Cstarcat_tensor_category}
\end{lemma}

A \emph{symmetric monoidal category} is a category with a tensor product functor~\(\otimes\), a unit object~\(\UNIT\), and natural isomorphisms
\[
(A\otimes B)\otimes C \cong A\otimes (B\otimes C),\qquad A\otimes B\cong B\otimes A,\qquad \UNIT\otimes A \cong A \cong A\otimes \UNIT
\]
called associativity, commutativity, and unitality constraints; these are subject to various coherence laws, for which we refer to \cites{Saavedra:Tannakiennes}.  These conditions allow to define tensor products \(\bigotimes_{x\in F} A_x\) for any finite set of objects \((A_x)_{x\in F}\) with the expected properties such as natural isomorphisms \(\bigotimes_{x\in F_1} A_x \otimes \bigotimes_{x\in F_2} A_x \cong \bigotimes_{x\in F} A_x\) for any decomposition \(F=F_1\sqcup F_2\) into disjoint subsets.  The associativity, commutativity, and unitality constraints are obvious in our case, and the coherence laws are trivial to verify.  Therefore, we omit the details.

Let \(\Grd_1\) and~\(\Grd_2\) be groupoids and let \(f\colon \Grd_1\to\Grd_2\) be a continuous functor.  Let \(f^{(0)}\) and~\(f^{(1)}\) be its actions on objects and morphisms, respectively.  If~\(A\) is a \(\Grd_2\)\nb-\(\Cst\)\nb-algebra with action~\(\alpha\), then \((f^{(0)})^*(A)\) is a \(\Grd_1\)\nb-\(\Cst\)\nb-algebra for the action
\begin{multline*}
  s_1^*(f^{(0)})^*(A)
  \cong (f^{(0)}s_1)^*(A)
  = (s_2f^{(1)})^*(A)
  \cong (f^{(1)})^*s_2^*(A)\\
  \xrightarrow{(f^{(1)})^*(\alpha)} (f^{(1)})^*r_2^*(A)
  \cong (r_2f^{(1)})^*(A)
  = (f^{(0)}r_1)^*(A)
  \cong r_1^*(f^{(0)})^*(A).
\end{multline*}
This defines a functor
\[
f^*\colon \Cstarcat_{\Grd_2} \to \Cstarcat_{\Grd_1},
\]
which is \emph{symmetric monoidal}, that is, we have canonical isomorphisms
\begin{equation}
  \label{eq:fstar_smf}
  f^*(A) \otimes_{\Grd_1} f^*(B) \cong f^*(A\otimes_{\Grd_2} B)
\end{equation}
that are compatible in a suitable sense with the associativity, commutativity, and unitality constraints in \(\Cstarcat_{\Grd_2}\) and \(\Cstarcat_{\Grd_1}\) (we refer to~\cite{Saavedra:Tannakiennes} for the precise definition).  The natural transformation in~\eqref{eq:fstar_smf} is part of the data of a symmetric monoidal functor.  Again we omit the proof because it is trivial once it is clear what has to be checked.  As a consequence, \(f^*\) preserves tensor units, that is,
\[
f^*\bigl(\CONT_0(\Grd_2^{(0)})\bigr) \cong \CONT_0(\Grd_1^{(0)}).
\]

Let~\(\Tot\) be a \(\Grd\)\nb-space.  Then the category \(\Cstarcat_{\Grd\ltimes\Tot}\) carries its own tensor product, which we always denote by~\(\otimes_\Tot\), to distinguish it from the tensor product~\(\otimes\) in~\(\Cstarcat_\Grd\).  The projection map \(p_\Tot\colon \Grd\ltimes\Tot\to\Grd\) induces a functor
\[
p_\Tot^*\colon \Cstarcat_\Grd\to\Cstarcat_{\Grd\ltimes\Tot},
\]
which acts by \(A\mapsto \CONT_0(\Tot)\otimes_\Base A\) on objects.  We have seen above that such functors are symmetric monoidal, that is, if \(A\) and~\(B\) are \(\Grd\)\nb-\(\Cst\)\nb-algebras, then
\begin{equation}
  \label{eq:p_X_tensor_functor}
  p_\Tot^*(A) \otimes_\Tot p_\Tot^*(B)
  \cong p_\Tot^*(A\otimes B).
\end{equation}

If~\(A\) is a \(\Grd\ltimes\Tot\)-\(\Cst\)\nb-algebra and~\(B\) is merely a \(\Grd\)\nb-\(\Cst\)\nb-algebra, then \(A\otimes_\Base B\) is a \(\Grd\ltimes\Tot\)\nb-\(\Cst\)\nb-algebra.  This defines another tensor product operation
\[
\otimes = \otimes_\Base\colon \Cstarcat_{\Grd\ltimes\Tot} \times \Cstarcat_\Grd \to \Cstarcat_{\Grd\ltimes\Tot},
\]
which has obvious associativity and unitality constraints
\[
(A\otimes B)\otimes C \cong A\otimes (B\otimes C), \qquad A\otimes \UNIT \cong A,
\]
where~\(A\) belongs to \(\Cstarcat_{\Grd\ltimes\Tot}\), \(B\) and~\(C\) belong to \(\Cstarcat_\Grd\), and~\(\UNIT\) is the unit object, here \(\CONT_0(\Base)\).  These natural isomorphisms satisfy the relevant coherence laws formalised in~\cite{Pareigis:C-categories}.  In the notation of~\cite{Pareigis:C-categories}, \(\Cstarcat_{\Grd\ltimes\Tot}\) is a \emph{\(\Cstarcat_\Grd\)\nb-category}.

Our two tensor products are related by a canonical isomorphism
\begin{equation}
  \label{eq:otimes_via_pullback}
  A\otimes B\cong A\otimes_\Tot p_\Tot^*(B),
\end{equation}
or, more precisely,
\[
A\otimes_\Tot p_\Tot^*(B) \defeq A\otimes_\Tot (\CONT_0(\Tot)\otimes B) \cong \bigl(A\otimes_\Tot \CONT_0(\Tot)\bigr)\otimes B \cong A\otimes B.
\]
These isomorphisms are all natural and \(\Grd\ltimes\Tot\)-equivariant.

We also have a canonical \emph{forgetful functor}
\[
\forget_\Tot\colon \Cstarcat_{\Grd\ltimes\Tot} \to \Cstarcat_\Grd,
\]
which maps a \(\Grd\)\nb-\(\Cst\)\nb-algebra over~\(\Tot\) to the underlying \(\Grd\)\nb-\(\Cst\)\nb-algebra, forgetting the \(\Tot\)\nb-structure.  This is a \emph{\(\Cstarcat_\Grd\)\nb-functor} in the notation of~\cite{Pareigis:C-categories}, that is, there are natural isomorphisms
\[
\forget_\Tot(A\otimes B) \cong \forget_\Tot(A)\otimes B
\]
for~\(A\) in \(\Cstarcat_{\Grd\ltimes\Tot}\) and~\(B\) in \(\Cstarcat_\Grd\), and these isomorphisms are compatible with the associativity and unitality constraints.

\section{Equivariant Kasparov theory for groupoids}
\label{sec:grd_KK}

We use the equivariant Kasparov theory for \(\Cst\)\nb-algebras with groupoid actions by Pierre-Yves Le Gall~\cite{LeGall:KK_groupoid}.  Let~\(\Grd\) be a groupoid with object space~\(\Base\).  Le Gall defines \(\Z/2\)-graded Abelian groups \(\KK^\Grd_*(A,B)\) for (possibly \(\Z/2\)\nb-graded) \(\Grd\)-\(\Cst\)\nb-algebras \(A\) and~\(B\).  He also constructs a Kasparov cup-cap product
\begin{equation}
  \label{eq:cup-cap}
  \otimes_D\colon
  \KK^\Grd_*(A_1,B_1\otimes D) \times \KK^\Grd_*(D\otimes A_2, B_2)
  \to \KK^\Grd_*(A_1\otimes A_2, B_1\otimes B_2)
\end{equation}
in \(\KK^\Grd\) with the expected properties such as associativity in general and graded commutativity of the exterior product (see \cite{LeGall:KK_groupoid}*{\S6.3}).  Throughout this section, \(\otimes\) denotes the tensor product \emph{over~\(\Base\)}, so that it would be more precise to write~\(\otimes_{\Base,D}\) instead of~\(\otimes_D\).

\begin{notation}
  \label{note:KK_grading}
  When we write \(\KK^\Grd_*(A,B)\), we always mean the \(\Z/2\)-graded group.  We write \(\KK^\Grd_0(A,B)\) and \(\KK^\Grd_1(A,B)\) for the even and odd parts of \(\KK^\Grd_*(A,B)\).  We let \(\KKcat_\Grd\) be the category whose objects are the (separable, \(\Z/2\)\nb-graded) \(\Grd\)\nb-\(\Cst\)\nb-algebras and whose morphism spaces are \(\KK^\Grd_*(A,B)\), with composition given by the Kasparov composition product.
\end{notation}

\begin{example}
  \label{exa:Le_Gall_generalises_Kasparov}
  If~\(G\) is a locally compact group, viewed as a groupoid, then Le Gall's bivariant \(\K\)\nb-theory agrees with Kasparov's theory defined in~\cite{Kasparov:Novikov}.
\end{example}

\begin{example}
  \label{exa:Le_Gall_generalises_Kasparov_II}
  If \(\Grd = G\ltimes \Tot\) for a locally compact group~\(G\) and a locally compact \(G\)\nb-space~\(\Tot\), then \(\KK^\Grd_*(A,B)\) agrees with Kasparov's \(\cRKK^G_*(\Tot;A,B)\).  This also contains Kasparov's groups \(\RKK^G_*(\Tot;A,B)\) for two \(G\)\nb-\(\Cst\)\nb-algebras \(A\) and~\(B\) as a special case because
  \[
  \RKK^G_*(\Tot;A,B) \defeq \cRKK^G_*\bigl(\Tot;\CONT_0(\Tot,A),\CONT_0(\Tot,B)\bigr).
  \]
\end{example}

The cup-cap product~\eqref{eq:cup-cap} contains an exterior tensor product operation
\[
\otimes = \otimes_\Base\colon \KKcat_\Grd\times\KKcat_\Grd\to\KKcat_\Grd, \qquad (A,B)\mapsto A\otimes B,
\]
which extends the tensor product on~\(\Cstarcat_\Grd\) and turns \(\KKcat_\Grd\) into an additive symmetric monoidal category (see \cites{Saavedra:Tannakiennes, Meyer:KK-survey}).  That is, the associativity, commutativity, and unitality constraints that exist in~\(\Cstarcat_\Grd\) descend to natural transformations on \(\KKcat_\Grd\); this follows from the universal property of \(\KKcat_\Grd\) in the ungraded case or by direct inspection.  Fixing one variable, we get the exterior product functors
\[
\sigma_D\colon \KKcat_\Grd\to \KKcat_\Grd,\qquad A\mapsto A\otimes D
\]
for all \(\Grd\)\nb-\(\Cst\)\nb-algebras~\(D\).  These are additive \(\KKcat_\Grd\)\nb-functors, that is, there are natural isomorphisms \(\sigma_D(A\otimes B)\cong \sigma_D(A)\otimes B\) with good formal properties (see~\cite{Pareigis:C-categories}).

If \(\Grd_1\) and~\(\Grd_2\) are two groupoids and \(f\colon \Grd_1\to\Grd_2\) is a continuous functor, then the induced functor \(f^*\colon \Cstarcat_{\Grd_2}\to\Cstarcat_{\Grd_1}\) descends to an additive functor
\[
f^*\colon \KKcat_{\Grd_2} \to \KKcat_{\Grd_1},
\]
that is, there are canonical maps
\begin{equation}
  \label{eq:KK_functorial}
  f^*\colon \KK^{\Grd_2}_*(A,B)
  \to \KK^{\Grd_1}_*\bigl(f^*(A),f^*(B)\bigr)
\end{equation}
for all \(\Grd_2\)\nb-\(\Cst\)\nb-algebras \(A\) and~\(B\).  These maps are compatible with the cup-cap product in~\eqref{eq:cup-cap}, so that~\emph{\(f^*\) is a symmetric monoidal functor}.  More precisely, the natural isomorphisms \(f^*(A) \otimes_{\Grd_1} f^*(B) \cong f^*(A\otimes_{\Grd_2} B)\) in \(\Cstarcat_{\Grd_1}\) remain natural when we enlarge our morphism spaces from \(^*\)\nb-homomorphisms to \(\KK\).

Le Gall describes in~\cite{LeGall:KK_groupoid} how to extend this functoriality to Hilsum--Skandalis morphisms between groupoids.  As a consequence, \(\KKcat_{\Grd_1}\) and \(\KKcat_{\Grd_2}\) are equivalent as symmetric monoidal categories if the groupoids \(\Grd_1\) and~\(\Grd_2\) are equivalent.

We are mainly interested in the special case of~\eqref{eq:KK_functorial} where we consider the functor \(\Grd\ltimes\Tot\to \Grd\ltimes \Base=\Grd\) induced by the projection \(p_\Tot\colon \Tot\to\Base\) for a \(\Grd\)\nb-space~\(\Tot\).  This yields an additive, symmetric monoidal functor
\begin{equation}
  \label{eq:pTotstar}
  p_\Tot^*\colon \KKcat_\Grd \to \KKcat_{\Grd\ltimes\Tot},
\end{equation}
which acts on objects by \(A\mapsto \CONT_0(\Tot)\otimes A\).

The canonical tensor products in the categories \(\KKcat_{\Grd\ltimes\Tot}\) and \(\KKcat_\Grd\) are over \(\Base\) and~\(\Tot\), respectively.  Therefore, we denote the tensor product in \(\KKcat_{\Grd\ltimes\Tot}\) by~\(\otimes_\Tot\).

The tensor product operation
\[
\otimes = \otimes_\Base\colon \Cstarcat_{\Grd\ltimes\Tot} \times \Cstarcat_\Grd \to\Cstarcat_{\Grd\ltimes\Tot}
\]
also descends to the Kasparov categories, yielding a bifunctor
\begin{equation}\label{eq:Tot_Base_module}
  \otimes = \otimes_\Base\colon
  \KKcat_{\Grd\ltimes\Tot} \times \KKcat_\Grd \to
  \KKcat_{\Grd\ltimes\Tot}
\end{equation}
that is additive in each variable.  The easiest construction uses~\eqref{eq:otimes_via_pullback}.  The bifunctor so defined obviously satisfies the associativity and unitality conditions needed for a \(\KKcat_\Grd\)-category (see~\cite{Pareigis:C-categories}).

The forgetful functor descends to an additive functor
\[
\forget_\Tot\colon \KKcat_{\Grd\ltimes\Tot} \to \KKcat_\Grd
\]
between the Kasparov categories.  This is a \(\KKcat_\Grd\)-functor in the notation of~\cite{Pareigis:C-categories}.  The obvious \(\Cst\)\nb-algebra isomorphisms
\[
\forget_\Tot(A\otimes B) \cong \forget_\Tot(A)\otimes B
\]
for all \(\Grd\ltimes\Tot\)-\(\Cst\)\nb-algebras~\(A\) and all \(\Grd\)-\(\Cst\)\nb-algebras~\(B\) remain natural on the level of Kasparov categories.

Since many constructions do not work for arbitrary \(\Grd\ltimes\Tot\)-\(\Cst\)\nb-algebras, we often restrict to the following full subcategory of \(\KK^{\Grd\ltimes\Tot}\):

\begin{definition}
  \label{def:RKK}
  Let \(A\) and~\(B\) be \(\Grd\)-\(\Cst\)\nb-algebras.  We define
  \[
  \RKK^\Grd_*(\Tot;A,B) \defeq \KK^{\Grd\ltimes\Tot}_*\bigl(p_\Tot^*(A),p_\Tot^*(B)\bigr),
  \]
  and we let \(\RKKcat_\Grd(\Tot)\) be the category whose objects are the \(\Grd\)-\(\Cst\)\nb-algebras and whose morphism spaces are \(\RKK^\Grd_0(\Tot;A,B)\).  By definition, \(\RKKcat_\Grd\) is the (co)image of the functor \(p_\Tot^*\) in~\eqref{eq:pTotstar}.  We often view \(\RKKcat_\Grd(\Tot)\) as a full subcategory of \(\KKcat_{\Grd\ltimes\Tot}\).
\end{definition}

\begin{example}
  \label{exa:RKK_classical}
  Let~\(\Grd\) be a group~\(G\), so that \(\Base=\pt\).  Then~\(\Tot\) is just a \(G\)\nb-space, and \(\KKcat_\Grd=\KKcat_G\) is the usual group-equivariant Kasparov category.  We have \(p_\Tot^*(A) = \CONT_0(\Tot,A)\) in this case, so that \(p_\Tot^*\colon \KKcat_G\to\RKKcat_G(\Tot)\) is the same functor as in \cite{Emerson-Meyer:Euler}*{Equation (7)}.  The functor \(\forget_\Tot\colon \RKKcat_G(\Tot)\to\KKcat_G\) agrees with the forgetful functor in \cite{Emerson-Meyer:Euler}*{Equation (6)}.
\end{example}

The subcategory \(\RKKcat_\Grd(\Tot)\subseteq \KKcat_{\Grd\ltimes\Tot}\) is closed under the tensor product operations \(\otimes_\Tot\) and~\(\otimes_\Base\).  Hence it is a symmetric monoidal category and a \(\KKcat_\Grd\)-category in its own right.

A \(\Grd\)\nb-equivariant map \(f\colon \Tot_1\to \Tot_2\) induces a functor \(f^*\colon \KKcat_{\Grd\ltimes\Tot_2} \to \KKcat_{\Grd\ltimes\Tot_1}\), which restricts to a functor
\[
f^*\colon \RKKcat_\Grd(\Tot_2) \to \RKKcat_\Grd(\Tot_1).
\]
This functoriality contains grading preserving group homomorphisms
\[
f^*\colon \RKK^\Grd_*(\Tot_2;A,B) \to \RKK^\Grd_*(\Tot_1;A,B),
\]
which are compatible with cup-cap products in both variables \(A\) and~\(B\).  These maps also turn \(\Tot\mapsto \RKK^\Grd_*(\Tot;A,B)\) into a functor from the category of locally compact \(\Grd\)\nb-spaces with \(\Grd\)\nb-equivariant continuous maps to the category of \(\Z/2\)-graded Abelian groups.  This is a \emph{homotopy functor}, that is, two \(\Grd\)\nb-equivariantly homotopic maps induce the same map on \(\RKK^\Grd_*\) (see Example~\ref{exa:RKK_space_homotopy} below for a proof).

\begin{notation}
  \label{note:sigma_EP}
  Let~\(\Dual\) be a \(\Grd\ltimes\Tot\)-\(\Cst\)\nb-algebra.  Then there is an associated functor
  \[
  \sigma_\Dual\colon \KKcat_{\Grd\ltimes\Tot} \to \KKcat_{\Grd\ltimes\Tot}, \qquad A\mapsto \Dual\otimes_\Tot A.
  \]
  We denote the composite functor
  \[
  \RKKcat_\Grd(\Tot) \xrightarrow{\subseteq} \KKcat_{\Grd\ltimes \Tot} \xrightarrow{\sigma_\Dual} \KKcat_{\Grd\ltimes \Tot} \xrightarrow{\forget_\Tot} \KKcat_\Grd
  \]
  by~\(T_\Dual\).  We have \(T_\Dual(A) = \Dual\otimes A\) for a \(\Grd\)\nb-\(\Cst\)\nb-algebra~\(A\), viewed as an object of \(\RKKcat_\Grd(\Tot)\), because of the natural isomorphisms \(\Dual \otimes_\Tot p_\Tot^*(A) \cong \Dual\otimes A\).  Thus~\(T_\Dual\) determines maps
  \[
  T_\Dual\colon \RKK^\Grd_*(\Tot;A,B) \defeq \KK^{\Grd\ltimes\Tot}_*\bigl(p_\Tot^*(A),p_\Tot^*(B)\bigr) \to \KK^\Grd_*(\Dual\otimes A, \Dual\otimes B).
  \]
\end{notation}

The functor~\(T_\Dual\) is the analogue for groupoids of the functor that is called \(\sigma_{\Tot,\Dual}\) in~\cite{Emerson-Meyer:Euler}.  If \(f\in \KK^\Grd_*(A,B)\), then
\begin{equation}
  \label{eq:T_sour_P}
  T_\Dual\bigl(p_\Tot^*(f)\bigr)
  = \sigma_\Dual(f)
  = \ID_\Dual\otimes f\qquad
  \text{in \(\KK^\Grd_*(\Dual\otimes A,\Dual\otimes B)\).}
\end{equation}
This generalises \cite{Emerson-Meyer:Euler}*{Equation (26)}.

\section{The first duality}
\label{sec:def_Kasparov_dual}

Let~\(\Grd\) be a locally compact groupoid and let \(\Base\defeq\Grd^{(0)}\) with the canonical (left) \(\Grd\)\nb-action, so that \(\Grd\ltimes\Base \cong \Grd\).  Let~\(\Tot\) be a \(\Grd\)\nb-space.  The notion of an equivariant Kasparov dual for group actions in~\cite{Emerson-Meyer:Euler} can be copied literally to our more general setup.  To clarify the relationship, we write
\[
\UNIT \defeq \CONT_0(\Base), \qquad \UNIT_\Tot \defeq \CONT_0(\Tot).
\]
These are the tensor units in \(\KKcat_\Grd\) and \(\KKcat_{\Grd\ltimes\Tot}\), respectively.  Wherever~\(\C\) appears in~\cite{Emerson-Meyer:Euler}, it is replaced by~\(\UNIT\).  Furthermore, we write~\(T_\Dual\) instead of~\(\sigma_{\Tot,\Dual}\) and~\(\UNIT_\Tot\) instead of \(\CONT_0(\Tot)\) here.

\begin{definition}
  \label{def:Kasparov_dual}
  Let \(n\in\Z\).  An \emph{\(n\)\nb-dimensional \(\Grd\)\nb-equivariant Kasparov dual} for the \(\Grd\)\nb-space~\(\Tot\) is a triple \((\Dual,D,\Theta)\), where
  \begin{itemize}
  \item \(\Dual\) is a (possibly \(\Z/2\)-graded) \(\Grd\ltimes\Tot\)-\(\Cst\)\nb-algebra,

  \item \(D\in\KK^\Grd_{-n}(\Dual,\UNIT)\), and

  \item \(\Theta\in\RKK^\Grd_n(\Tot;\UNIT,\Dual)\),
  \end{itemize}
  subject to the following three conditions:
  \begin{alignat}{2}
    \label{eq:Kasparov_dual_1}
    \Theta \otimes_\Dual D &= \ID_\UNIT
    &\qquad&\text{in \(\RKK^\Grd_0(\Tot;\UNIT,\UNIT)\);}\\
    \label{eq:Kasparov_dual_2}
    \Theta \otimes_\Tot f &= \Theta \otimes_\Dual T_\Dual(f)
    &\qquad&\text{in \(\RKK^\Grd_{*+n}(\Tot;A, \Dual\otimes B)\)}\\
    \intertext{for all \(\Grd\)\nb-\(\Cst\)\nb-algebras \(A\) and~\(B\) and all \(f\in\RKK^\Grd_*(\Tot;A,B)\);}
    \label{eq:Kasparov_dual_3}
    T_\Dual(\Theta)\otimes_{\Dual\otimes\Dual} \flip_\Dual &= (-1)^n T_\Dual(\Theta)
    &\qquad&\text{in \(\KK^\Grd_n(\Dual,\Dual\otimes\Dual)\),}
  \end{alignat}
  where~\(\flip_\Dual\) is the flip automorphism on \(\Dual\otimes\Dual\) as in~\cite{Emerson-Meyer:Euler}.
\end{definition}

This differs slightly from the definition of a Kasparov dual in \cite{Emerson-Meyer:Euler}*{Definition 18} because~\eqref{eq:Kasparov_dual_2} contains no auxiliary space~\(\Other\) as in~\cite{Emerson-Meyer:Euler}.  As a result, \((\Dual,D,\Theta)\) is a Kasparov dual in the sense of~\cite{Emerson-Meyer:Euler} if and only if its pull-back to~\(\Base'\) is a Kasparov dual for \(\Base'\times_\Base \Tot\), viewed as a \(\Grd\ltimes \Base'\)-space, for any \(\Grd\)\nb-space~\(\Base'\).  The space~\(\Base'\) plays no significant role and is only added in~\cite{Emerson-Meyer:Euler} because this general setting is considered in~\cite{Kasparov:Novikov}.  We expect that geometric sufficient conditions for Kasparov duals are invariant under this base change operation.

The notion of dual in Definition~\ref{def:Kasparov_dual} is relative to the base space~\(\Base\).  In a sense, a \(\Grd\)\nb-equivariant Kasparov dual is a \(\Grd\)\nb-equivariant family of Kasparov duals for the fibres of the map \(p_\Tot\colon \Tot\to\Base\).

We remark without proof that~\eqref{eq:Kasparov_dual_3} is equivalent, in the presence of the other two conditions, to
\begin{equation}
  \label{eq:other_third_condition}
  T_\Dual(\Theta)\otimes_{\Dual\otimes\Dual} (D\otimes\ID_\Dual)
  = (-1)^n \ID_\Dual
  \qquad \text{in \(\KK^\Grd_0(\Dual,\Dual)\)}
\end{equation}
(the easier implication \eqref{eq:Kasparov_dual_3}$\Longrightarrow$\eqref{eq:other_third_condition} is contained in Lemma~\ref{lem:duality_coalgebra}).  Both formulations \eqref{eq:Kasparov_dual_3} and~\eqref{eq:other_third_condition} tend to be equally hard to check.

Condition~\eqref{eq:Kasparov_dual_2} is not so difficult to check in practice, but it lacks good functoriality properties and is easily overlooked (such as in \cite{Tu:Novikov}*{Th\'eor\`eme 5.6}).  This condition turns out to be automatic for universal proper \(\Grd\)\nb-spaces (Lemma~\ref{lem:dual_2_automatic}).  We will comment further on its role in Section~\ref{sec:first_duality_compact}, where we discuss the special case where the map \(p_\Tot\colon \Tot\to\Base\) is proper.

\begin{theorem}
  \label{the:first_duality}
  Let~\(\Dual\) be a \(\Grd\ltimes\Tot\)-\(\Cst\)\nb-algebra, \(n\in\Z\), \(D\in\KK^\Grd_{-n}(\Dual,\UNIT)\), and \(\Theta\in\RKK^\Grd_n(\Tot;\UNIT,\Dual)\).  The natural transformations
  \begin{alignat*}{2}
    \PD&\colon \KK^\Grd_{i-n}(\Dual\otimes A,B) \to \RKK^\Grd_i(\Tot;A,B), &\qquad
    f&\mapsto \Theta \otimes_\Dual f,\\
    \PD^*&\colon \RKK^\Grd_i(\Tot;A,B) \to \KK^\Grd_{i-n}(\Dual\otimes A,B), &\qquad g&\mapsto (-1)^{in} T_\Dual(g) \otimes_\Dual D
  \end{alignat*}
  are inverse to each other for all \(\Grd\)\nb-\(\Cst\)-algebras \(A\) and~\(B\) if and only if \((\Dual,D,\Theta)\) is a \(\Grd\)\nb-equivariant Kasparov dual for~\(\Tot\).  In this case, the functor \(T_\Dual\colon \RKKcat_\Grd(\Tot)\to\KKcat_\Grd\) is left adjoint to the functor \(p_\Tot^*\colon \KKcat_\Grd\to \RKKcat_\Grd(\Tot)\).
\end{theorem}

We call the isomorphism in Theorem~\ref{the:first_duality} \emph{Kasparov's first duality isomorphism} because it goes back to Gennadi Kasparov's proof of his First Poincar\'e Duality Theorem \cite{Kasparov:Novikov}*{Theorem 4.9}.  We postpone the proof of Theorem~\ref{the:first_duality} to Section~\ref{sec:comultiplication} in order to utilise some more notation that we need also for other purposes.

Kasparov duals need not exist in general, even if the groupoid~\(\Grd\) is trivial and \(\Base=\pt\).  The Cantor set is a counterexample (see Proposition~\ref{pro:duals_UCT}).

To construct a Kasparov dual for a space, we need some geometric information on the space in question.  For instance, for a smooth manifold we can either use Clifford algebras or the tangent bundle to construct a Kasparov dual.  We may also triangulate the manifold and use this to construct a more combinatorial dual.

We will use Kasparov duals to construct Lefschetz invariants and Euler characteristics.  This leads to the question how unique Kasparov duals are and whether other notions derived from them may depend on choices.  The following counterexample shows that Kasparov duals are not unique.

\begin{example}
  \label{exa:counter_sour_unique}
  Let~\(\Grd\) be the trivial groupoid, so that \(\Base\defeq\pt\) is the one-point space, and let \(\Tot\defeq[0,1]\).  The homotopy invariance of \(\RKK\) in the space-variable implies
  \[
  \RKK_*(\Tot;A,B) \cong \RKK_*(\pt;A,B) = \KK_*(A,B)
  \]
  for all \(\Cst\)\nb-algebras \(A\) and~\(B\).

  Let \(\Dual\defeq\CONT([0,1])\), let~\(D\) be the class of an evaluation homomorphism, and let~\(\Theta\) be the class of the map \(\CONT([0,1])\to \CONT([0,1]\times[0,1])\), \(f\mapsto f\otimes 1\).  Inspection shows that this is a Kasparov dual for~\(\Tot\).  So is \(\Dual'\defeq\C\), viewed as a \(\Cst\)\nb-algebra over \([0,1]\) by evaluation at any point, with \(D'\) and~\(\Theta'\) being the identity maps.  While \(\Dual\) and~\(\Dual'\) are homotopy equivalent and hence isomorphic in \(\KKcat\), they are not isomorphic in \(\RKKcat([0,1])\) because their fibres are not \(\KK\)-equivalent everywhere.
\end{example}

Abstract duals formalise what is unique about Kasparov duals.  This is important because constructions that can be expressed in terms of abstract duals yield equivalent results for \emph{all} Kasparov duals.  The equivalence between the smooth and combinatorial duals for smooth manifolds is used in \cites{Emerson-Meyer:Euler, Emerson-Meyer:Equi_Lefschetz} to reprove an index theorem for the equivariant Euler operator and the Equivariant Lefschetz Fixed Point Theorem of Wolfgang L\"uck and Jonathan Rosenberg (see \cites{Lueck-Rosenberg:Lefschetz,Lueck-Rosenberg:Euler}).

\begin{definition}
  \label{def:abstract_dual}
  Let~\(\Dual\) be a \(\Grd\)\nb-\(\Cst\)\nb-algebra and let \(\Theta\in\RKK^\Grd_n(\Tot;\UNIT,\Dual)\).  We call the pair \((\Dual,\Theta)\) an \(n\)\nb-dimensional \emph{abstract dual} for~\(\Tot\) if the map \(\PD\) defined as in Theorem~\ref{the:first_duality} is an isomorphism for all \(\Grd\)-\(\Cst\)\nb-algebras \(A\) and~\(B\).  We call \((\Dual,\Theta)\) an \(n\)\nb-dimensional \emph{weak abstract dual} for~\(\Tot\) if \(\PD\) is an isomorphism for \(A=\UNIT_\Tot\) and all \(\Grd\)-\(\Cst\)\nb-algebras~\(B\).
\end{definition}

We can always adjust the dimension to be~\(0\) by passing to a suspension of~\(\Dual\).

Theorem~\ref{the:first_duality} shows that \((\Dual,\Theta)\) is an abstract dual if \((\Dual,D,\Theta)\) is a Kasparov dual.  We will see below that we can recover~\(D\) and the functor~\(T_\Dual\) from the abstract dual.  The main difference between Kasparov duals and abstract duals is that for the latter, \(\Dual\) is not necessarily a \(\Cst\)\nb-algebra over~\(\Tot\).  This is to be expected because of Example~\ref{exa:counter_sour_unique}.  We need weak abstract duals in connection with~\cite{Emerson-Meyer:Correspondences}, for technical reasons, because computations in the geometric version of \(\KK\) may provide weak abstract duals, but not the existence of a duality isomorphism for \emph{all} \(\Grd\)\nb-\(\Cst\)-algebras~\(A\).

\begin{proposition}
  \label{pro:dual_unique}
  A weak abstract dual for a space~\(\Tot\) is unique up to a canonical \(\KK^\Grd\)-equivalence if it exists, and even covariantly functorial in the following sense.

  Let \(\Tot\) and~\(\Other\) be two \(\Grd\)\nb-spaces and let \(f\colon \Tot\to \Other\) be a \(\Grd\)\nb-equivariant continuous map.  Let \((\Dual_\Tot,\Theta_\Tot)\) and \((\Dual_\Other,\Theta_\Other)\) be weak abstract duals for \(\Tot\) and~\(\Other\) of dimensions \(n_\Tot\) and~\(n_\Other\), respectively.  Then there is a unique \(\Dual_f\in\KK^\Grd_{n_\Other-n_\Tot}(\Dual_\Tot,\Dual_\Other)\) with \(\Theta_\Tot \otimes_{\Dual_\Tot} \Dual_f = f^*(\Theta_\Other)\).  Given two composable maps between three spaces with duals, we have \(\Dual_{f\circ g} = \Dual_f \circ \Dual_g\).  If \(\Tot=\Other\), \(f=\ID_\Tot\), and \((\Dual_\Tot,\Theta_\Tot)=(\Dual_\Other,\Theta_\Other)\), then \(\Dual_f=\ID_{\Dual_\Tot}\).  If only \(\Tot=\Other\), \(f=\ID_\Tot\), then~\(\Dual_f\) is a \(\KK^\Grd\)-equivalence between the two duals of~\(\Tot\).
\end{proposition}

\begin{proof}
  The condition characterising~\(\Dual_f\) is equivalent to \(\PD_\Tot(\Dual_f) = f^*(\Theta_\Other)\), which uniquely determines~\(\Dual_f\).  The uniqueness of~\(\Dual_f\) implies identities such as \(\Dual_{f\circ g} = \Dual_f \circ \Dual_g\) for composable morphisms \(f\) and~\(g\) and \(\Dual_{\ID_\Tot}=\ID_{\Dual_\Tot}\) when we use the same dual of~\(\Tot\) twice.  These functoriality properties imply that~\(\Dual_f\) is invertible if~\(f\) is a \(\Grd\)\nb-homotopy equivalence.  In particular, the dual is unique up to a canonical isomorphism.
\end{proof}

\subsection{Basic constructions with abstract duals}
\label{sec:basic_duality_constructions}

Most of the following constructions are immediate generalisations of corresponding ones in~\cite{Emerson-Meyer:Euler}.  They only use an abstract dual or a weak abstract dual and, therefore, up to canonical isomorphisms between different duals, do not depend on the choice of Kasparov dual.

Let \((\Dual,\Theta)\) be an \(n\)\nb-dimensional weak abstract dual for a \(\Grd\)\nb-space~\(\Tot\).  Another weak abstract dual \((\Dual',\Theta')\) of dimension~\(n'\) is related to \((\Dual,\Theta)\) by an invertible element~\(\psi\) in \(\KK^\Grd_{n'-n}(\Dual,\Dual')\) such that \(\Theta\otimes_{\Dual} \psi = \Theta'\).  We will express in these terms what happens when we change the dual.

\subsubsection{Counit}
\label{sec:counit}

Define \(D\in\KK^\Grd_{-n}(\Dual,\UNIT)\) by
\begin{equation}
  \label{eq:D_and_Theta}
  \PD(D)\defeq \Theta\otimes_\Dual D = 1
  \qquad \text{in \(\RKK^\Grd_0(\Tot;\UNIT,\UNIT)\).}
\end{equation}
Comparison with~\eqref{eq:Kasparov_dual_1} shows that this is the class named~\(D\) in a Kasparov dual, which is uniquely determined once \(\Dual\) and~\(\Theta\) are fixed.  A change of dual replaces~\(D\) by \(\psi^{-1}\otimes_\Dual D\).

The example of the self-duality of spin manifolds motivates calling \(D\) and~\(\Theta\) \emph{Dirac} and \emph{local dual Dirac}.  We may also call~\(D\) the \emph{counit} of the duality because it plays the algebraic role of a counit by Lemma~\ref{lem:duality_coalgebra} below.

\subsubsection{Comultiplication}
\label{sec:comultiplication}

Define \(\comul\in\KK^\Grd_n(\Dual,\Dual\otimes\Dual)\) by
\[
\PD(\comul) \defeq \Theta\otimes_\Dual \comul = \Theta \otimes_\Tot \Theta \qquad \text{in \(\RKK^\Grd_{2n}(\Tot;\UNIT,\Dual\otimes\Dual)\).}
\]
We call~\(\comul\) the \emph{comultiplication} of the duality.  A change of dual replaces~\(\comul\) by
\[
(-1)^{n(n'-n)}\psi^{-1}\otimes_\Dual \comul \otimes_{\Dual\otimes\Dual} (\psi\otimes\psi) \in \KK^\Grd_{n'}(\Dual',\Dual'\otimes\Dual')
\]
because \((\Theta\otimes_\Dual\psi)\otimes_\Tot (\Theta\otimes_\Dual\psi) = (-1)^{n(n'-n)} (\Theta\otimes_\Tot\Theta) \otimes_{\Dual\otimes\Dual} (\psi\otimes\psi)\) by the Koszul sign rule.

\begin{lemma}
  \label{lem:duality_coalgebra}
  The object~\(\Dual\) of \(\KKcat_\Grd\) with counit~\(D\) and comultiplication~\(\comul\) is a cocommutative, counital coalgebra object in the tensor category \(\KKcat_\Grd\) if \(n=0\).  For general~\(n\), the coassociativity, cocommutativity, and counitality conditions hold up to signs:
  \begin{gather}
    \label{eq:coassociative}
    (-1)^n \comul\otimes_{\Dual\otimes \Dual}(\comul\otimes 1_\Dual) = \comul \otimes_{\Dual\otimes\Dual}(1_\Dual\otimes\comul),
    \\
    \label{eq:cocommutative}
    \comul\otimes_{\Dual\otimes\Dual}\flip_\Dual = (-1)^n\comul,
    \\
    \label{eq:counit}
    (-1)^n\comul \otimes_{\Dual\otimes\Dual} (D\otimes 1_\Dual) = 1_\Dual = \comul \otimes_{\Dual\otimes\Dual} (1_\Dual\otimes D).
  \end{gather}
  Equation~\eqref{eq:coassociative} holds in \(\KK^\Grd_{2n}(\Dual,\Dual^{\otimes 3})\), \eqref{eq:cocommutative} holds in \(\KK^\Grd_n(\Dual,\Dual\otimes\Dual)\), and~\eqref{eq:counit} holds in \(\KK^\Grd_0(\Dual,\Dual)\).
\end{lemma}

Recall that~\(\flip_\Dual\) denotes the flip operator on \(\Dual\otimes\Dual\).

\begin{proof}
  The proof is identical to that of \cite{Emerson-Meyer:Euler}*{Lemma 17}.
\end{proof}

\begin{proof}[Proof of Theorem~\textup{\ref{the:first_duality}}]
  The proof that \(\PD\) and \(\PD^*\) are inverse to each other if \((\Dual,D,\Theta)\) is a Kasparov dual can be copied from the proof of \cite{Emerson-Meyer:Euler}*{Proposition 19}; see also the proof of Theorem~\ref{the:KK_first_non-trivial} below.  The existence of such isomorphisms means that the functor~\(T_\Dual\) is left adjoint to the functor~\(p_\Tot^*\) (with range category \(\RKKcat_\Grd(\Tot)\)).

  Conversely, assume that \(\PD\) and \(\PD^*\) are inverse to each other, so that \((\Dual,\Theta)\) is an abstract dual for~\(\Tot\).  We check that \((\Dual,D,\Theta)\) is a Kasparov dual.  The first condition~\eqref{eq:Kasparov_dual_1} follows because it is equivalent to \(\PD\circ\PD^*(\ID_\UNIT)=\ID_\UNIT\) in \(\RKK^\Grd_0(\Tot;\UNIT,\UNIT)\).  Thus~\(D\) is the counit of the duality.  Furthermore, we get \(\PD^*(\Theta)=\ID_\Dual\) because \(\PD(\ID_\Dual)=\Theta\).  That is,
  \begin{equation}
    \label{eq:T_sour_Theta_D}
    (-1)^n T_\Dual(\Theta) \otimes_{\Dual\otimes\Dual}
    (D\otimes\ID_\Dual) = \ID_\Dual.
  \end{equation}
  Equation~\eqref{eq:Kasparov_dual_2} is equivalent to
  \begin{equation}
    \label{eq:PDstar_Theta}
    \PD^*(\Theta \otimes_\Tot f) = T_\Dual(f)
  \end{equation}
  because \(\Theta\otimes_\Dual T_\Dual(f) = \PD\bigl(T_\Dual(f)\bigr)\).  We use graded commutativity of~\(\otimes_\Tot\) and functoriality of~\(T_\Dual\) to rewrite
  \begin{align*}
    \PD^*(\Theta\otimes_\Tot f) &= (-1)^{in} \PD^*(f\otimes_\Tot \Theta) \\
    &= (-1)^{in+(i+n)n} T_\Dual(f\otimes_{\CONT_0(\Tot)} \Theta) \otimes_\Dual D \\
    &= (-1)^{n} T_\Dual(f) \otimes_{T_\Dual\bigl(\CONT_0(\Tot)\bigr)} T_\Dual(\Theta) \otimes_\Dual D \\
    &= (-1)^{n} T_\Dual(f) \otimes_\Dual T_\Dual(\Theta) \otimes_\Dual D.
  \end{align*}
  Thus~\eqref{eq:PDstar_Theta} follows from~\eqref{eq:T_sour_Theta_D}.

  As a special case, \eqref{eq:PDstar_Theta} contains \(\PD^*(\Theta\otimes_\Tot\Theta) = T_\Dual(\Theta)\), so that
  \begin{equation}
    \label{eq:comul_from_Kasparov_dual}
    \comul = T_\Dual(\Theta).
  \end{equation}
  Hence~\eqref{eq:Kasparov_dual_3} is equivalent to~\eqref{eq:cocommutative}, which holds for any abstract dual.  This finishes the proof of Theorem~\ref{the:first_duality}.
\end{proof}

Equation~\eqref{eq:comul_from_Kasparov_dual} shows how to compute~\(\comul\) using a Kasparov dual.

\subsubsection{The tensor functor}
\label{sec:T_sour_abstract}

Now let \((\Dual,\Theta)\) be an abstract dual, that is, \(\PD\) is invertible for all \(\Grd\)\nb-\(\Cst\)\nb-algebras \(A\) and~\(B\).  We define
\[
T'_\Dual\colon \RKK^\Grd_*(\Tot;A,B) \to \KK^\Grd_*(\Dual\otimes A,\Dual\otimes B), \qquad f\mapsto \comul \otimes_\Dual \PD^{-1}(f),
\]
where~\(\comul\) is the comultiplication of the duality and~\(\otimes_\Dual\) operates on the \emph{second} copy of~\(\Dual\) in the target \(\Dual\otimes\Dual\) of~\(\comul\).  This map is denoted~\(\sigma'_\Dual\) in~\cite{Emerson-Meyer:Euler}.  A computation as in \cite{Emerson-Meyer:Euler}*{Equation (23)} yields
\begin{equation}
  \label{eq:PD_sigma_prime}
  \PD\bigl(T'_\Dual(f)\bigr) = \Theta \otimes_\Tot f
  \qquad \text{in \(\RKK^\Grd_{i+n}(\Tot;A,\Dual\otimes B)\)}
\end{equation}
for all \(f\in\RKK^\Grd_i(\Tot;A,B)\).  Thus~\eqref{eq:PDstar_Theta} implies
\[
T'_\Dual(f) = T_\Dual(f)
\]
if \((\Dual,\Theta)\) is part of a Kasparov dual.  Thus~\(T_\Dual\) does not depend on the dual, and we may write~\(T_\Dual\) instead of~\(T'_\Dual\) from now on.

A change of dual replaces~\(T_\Dual\) by the map
\[
\RKK^\Grd_i(\Tot;A,B) \ni f\mapsto (-1)^{i(n-n')}\psi^{-1}\otimes_\Dual T_\Dual(f)\otimes_\Dual \psi \in \KK^\Grd_i(\Dual'\otimes A, \Dual'\otimes B).
\]

\begin{lemma}
  \label{lem:sigma_sour_functor}
  The maps~\(T_\Dual\) above define a functor
  \[
  T_\Dual\colon \RKKcat_\Grd(\Tot)\to\KKcat_\Grd.
  \]
  This is a \(\KKcat_\Grd\)-functor, that is, it is compatible with the tensor products~\(\otimes\), and it is left adjoint to the functor \(p_\Tot^*\colon \KKcat_\Grd\to\RKKcat_\Grd\).
\end{lemma}

\begin{proof}
  It is clear that the natural transformation~\(\PD\) is compatible with~\(\otimes\) in~\eqref{eq:Tot_Base_module} in the sense that \(\PD(f_1\otimes f_2) = \PD(f_1)\otimes f_2\) if \(f_1\) and~\(f_2\) are morphisms in \(\RKKcat_\Grd(\Tot)\) and \(\KKcat_\Grd\), respectively.  Hence so are its inverse \(\PD^{-1}\) and~\(T_\Dual\).  The existence of a duality isomorphism as in Theorem~\ref{the:first_duality} implies that \(p_\Tot^*\colon \KKcat_\Grd\to\RKKcat_\Grd\) has a left adjoint \emph{functor} \(T\colon \RKKcat_\Grd\to\KKcat_\Grd\) that acts on objects by \(A\mapsto A\otimes\Dual\) like~\(T_\Dual\).  This is a functor for general nonsense reasons.  We claim that \(T_\Dual=T\), proving functoriality of~\(T_\Dual\).  The functor~\(T\) is constructed as follows.  A morphism \(\alpha\in\RKK^\Grd_j(\Tot;A_1,A_2)\) induces a natural transformation
  \[
  \alpha^*\colon \RKK^\Grd_i(\Tot;A_2,B) \to \RKK^\Grd_{i+j}(\Tot;A_1,B),
  \]
  which corresponds by the duality isomorphisms to a natural transformation
  \[
  \alpha^*\colon \KK^\Grd_{i-n}(\Dual\otimes A_2,B) \to \KK^\Grd_{i+j-n}(\Dual\otimes A_1,B).
  \]
  By definition, \(T(\alpha)\) is the image of \(\ID_{\Dual\otimes A_2}\) under this map.  Thus \(T(\alpha)\) is determined by the condition
  \[
  \PD\bigl(T(\alpha)\bigr) = \alpha^*\bigl(\PD(\ID_{\Dual\otimes A_2})\bigr) = \alpha^*(\Theta \otimes \ID_{A_2}) = \Theta \otimes_\Tot \alpha.
  \]
  The same condition uniquely characterises \(T_\Dual(\alpha)\) by~\eqref{eq:PD_sigma_prime}.
\end{proof}

Now we can describe the inverse duality map as in \cite{Emerson-Meyer:Euler}*{Equation (24)}:
\begin{equation}
  \label{eq:inverse_PD_abstract}
  \PD^{-1}(f) = (-1)^{in} T_\Dual(f) \otimes_\Dual D
  \qquad
  \text{in \(\KK^\Grd_{i-n}(\Dual\otimes A,B)\)}
\end{equation}
for \(f\in \RKK^\Grd_i(\Tot;A,B)\), generalising the \emph{definition} of~\(\PD^*\) for a Kasparov dual in Theorem~\ref{the:first_duality} to abstract duals.

\subsubsection{The diagonal restriction class}
\label{sec:diagonal_restriction}

The diagonal embedding \(\Tot\to\Tot\times_\Base\Tot\) is a proper map and hence induces a \(^*\)\nb-homomorphism
\[
\UNIT_\Tot\otimes \UNIT_\Tot \cong \CONT_0(\Tot\times_\Base\Tot)\to\CONT_0(\Tot) = \UNIT_\Tot.
\]
This map is \(\Grd\ltimes\Tot\)-equivariant and hence yields
\[
\Delta_\Tot\in\RKK_0^\Grd(\Tot;\UNIT_\Tot,\UNIT) \cong \KK_0^{\Grd\ltimes\Tot}\bigl(\CONT_0(\Tot\times_\Base\Tot),\CONT_0(\Tot)\bigr).
\]
This is the \emph{diagonal restriction class}, which is an ingredient in equivariant Euler characteristics (see Definition~\ref{def:Lefschetz_map}).  It yields a canonical map
\begin{equation}
  \label{eq:diagonal_gives_this}
  \blank\otimes_{\UNIT_\Tot}\Delta_\Tot\colon
  \KK^\Grd_*(\UNIT_\Tot\otimes A,\UNIT_\Tot\otimes B) \to
  \RKK^\Grd_*(\Tot;\UNIT_\Tot\otimes A,B).
\end{equation}
In particular, this contains a map \(\KK^\Grd_*(\UNIT_\Tot,\UNIT_\Tot) \to \RKK_*^\Grd(\Tot;\UNIT_\Tot,\UNIT)\), which will be used to construct Lefschetz invariants.

\begin{example}
  \label{exa:diagonal_restrict_proper_map}
  If \(f\colon \Tot\to\Tot\) is a proper, continuous, \(\Grd\)\nb-equivariant map, then
  \[
  [f]\otimes_{\UNIT_\Tot} \Delta_\Tot \in \RKK^\Grd_0(\Tot;\UNIT_\Tot,\UNIT)
  \]
  is the class of the \(^*\)\nb-homomorphism induced by \((\ID_\Tot,f)\colon \Tot \to \Tot \times_\Base\Tot\).

  Now drop the assumption that~\(f\) be proper.  Then \((\ID_\Tot,f)\) is still a proper, continuous, \(\Grd\)\nb-equivariant map.  The class of the \(^*\)\nb-homomorphism it induces is equal to \(f^*(\Delta_\Tot)\), where we use the maps
  \[
  f^*\colon \RKK^\Grd_*(\Tot;A,B) \to \RKK^\Grd_*(\Tot;A,B)
  \]
  for \(A=\UNIT_\Tot\), \(B=\UNIT\) induced by \(f\colon \Tot\to\Tot\).
\end{example}

\subsubsection{The multiplication class}
\label{sec:multiplication_class}

Let~\(T_\Dual\) be the tensor functor and~\(\Delta_\Tot\) the diagonal restriction class of an abstract dual \((\Dual,\Theta)\).  The \emph{multiplication class} of~\(\Dual\) is
\begin{equation}
  \label{eq:def_multiplication_class}
  [m] \defeq T_\Dual(\Delta_\Tot)
  \in \KK^\Grd_0(\Dual\otimes\UNIT_\Tot,\Dual).
\end{equation}
A change of dual replaces \([m]\) by \(\psi^{-1}\otimes_\Dual [m]\otimes_\Dual \psi\).

\begin{lemma}
  \label{lem:multiplication_class}
  Let \((\Dual,D,\Theta)\) be a Kasparov dual.  Then~\([m]\) is the class in \(\KK^\Grd\) of the multiplication homomorphism \(\CONT_0(\Tot)\otimes \Dual\to\Dual\) \textup(see~\textup{\eqref{eq:multiplication_map_over_Base})} that describes the \(\Tot\)\nb-structure on~\(\Dual\) \textup{(}up to commuting the tensor factors\textup{)}.
\end{lemma}

Recall that~\(\otimes\) denotes the tensor product over~\(\Base\).  Since a \(\Grd\)\nb-\(\Cst\)\nb-algebra is already a \(\Cst\)\nb-algebra over~\(\Base\), we can describe an additional structure of \(\Cst\)\nb-algebra over~\(\Tot\) by a multiplication homomorphism \(\CONT_0(\Tot)\otimes_\Base \Dual\to\Dual\).

\begin{proof}
  Whenever we have a Kasparov dual, the homomorphism \(T_\Dual(\Delta_\Tot)\) is the class of the multiplication homomorphism for~\(\Dual\) because~\(\Delta_\Tot\) is the multiplication homomorphism for \(\CONT_0(\Tot)\).
\end{proof}

\subsubsection{Abstract duality as an adjointness of functors}
\label{sec:duality_adjointness}

\begin{proposition}
  \label{pro:abstract_dual_adjoint}
  A \(\Grd\)\nb-space~\(\Tot\) has an abstract dual if and only if the functor
  \[
  p_\Tot^*\colon \KKcat_\Grd\to\RKKcat_\Grd(\Tot)
  \]
  has a left adjoint functor \(T\colon \RKKcat_\Grd(\Tot)\to\KKcat_\Grd\) such that~\(T\) is a \(\KKcat_\Grd\)-functor and the natural isomorphism
  \[
  \PD\colon \KK_0^\Grd(\Dual\otimes A,B) \to \RKK_0^\Grd(\Tot;A,B)
  \]
  is a \(\KKcat_\Grd\)-morphism in the notation of~\cite{Pareigis:C-categories}; this means that both \(T\) and~\(\PD\) are compatible with the tensor product~\(\otimes\).
\end{proposition}

\begin{proof}
  Given an abstract dual, the functor \(T\defeq T_\Dual\) is a \(\KKcat_\Grd\)-functor and left adjoint to~\(p_\Tot^*\) by Lemma~\ref{lem:sigma_sour_functor}.  The natural transformation \(\PD\) is compatible with~\(\otimes\) by definition.

  Suppose, conversely, that~\(p_\Tot^*\) has a left adjoint functor~\(T\) with the required properties.  Compatibility with~\(\otimes\) implies \(T(A) \cong T(\UNIT\otimes A) \cong \Dual\otimes A\) for the \(\Grd\)\nb-\(\Cst\)\nb-algebra \(\Dual\defeq T(\UNIT)\).  Let \(\Theta\defeq \PD(\ID_\Dual) \in \RKK^\Grd_0(\Tot;\UNIT,\Dual)\).  Compatibility with~\(\otimes\) yields \(\PD(\ID_{A\otimes\Dual}) = \PD(\ID_\Dual)\otimes \ID_A = \Theta\otimes\ID_A\).  Finally, naturality forces \(\PD\) to be of the form \(f = f\circ (\ID_{A\otimes\Dual})\mapsto \PD(\ID_{A\otimes\Dual}) \otimes_{A\otimes\Dual} f = \Theta \otimes_\Dual f\) for all \(f\in\KK^\Grd_0(A\otimes\Dual,B)\).  Hence \((\Dual,\Theta)\) is an abstract dual for~\(\Tot\).
\end{proof}

It may seem more natural to require an adjoint functor for~\(p_\Tot^*\) on \(\KKcat^{\Grd\ltimes\Tot}\), not just on the subcategory \(\RKKcat^\Grd(\Tot)\).  But such an extension is not possible in general (see Example~\ref{exa:first_Kasparov_fails} below).

\subsection{Equivariant Euler characteristic and Lefschetz invariants}
\label{sec:def_Euler_Lefschetz}

Now we use an abstract dual to define a \emph{Lefschetz map}
\[
\Lef\colon \RKK^\Grd_*\bigl(\Tot;\CONT_0(\Tot),\CONT_0(\Base)\bigr) \to \KK^\Grd_*\bigl(\CONT_0(\Tot),\CONT_0(\Base)\bigr).
\]
This generalises the familiar construction of Lefschetz numbers for self-maps of spaces in three ways: first, we consider self-maps in Kasparov theory; secondly, our invariant is an equivariant \(\K\)\nb-homology class, not a number; thirdly, self-maps are not required to be proper, so that the domain of our map is \(\RKK^\Grd_*\bigl(\Tot;\CONT_0(\Tot),\CONT_0(\Base)\bigr)\) and not \(\KK^\Grd_*\bigl(\CONT_0(\Tot),\CONT_0(\Tot)\bigr)\).

We let~\(\Tot\) be a \(\Grd\)\nb-space and \((\Dual,\Theta)\) an \(n\)\nb-dimensional abstract dual for~\(\Tot\) throughout.  Occasionally, we assume that this is part of a Kasparov dual \((\Dual,D,\Theta)\), but the definitions and main results do not require this.  Let \(\PD\) and \(\PD^{-1}\) be the duality isomorphisms.  As before, we write
\[
\UNIT \defeq \CONT_0(\Base), \qquad \UNIT_\Tot \defeq \CONT_0(\Tot).
\]
We let \(\Delta_\Tot\in \RKK_0^\Grd(\Tot;\UNIT_\Tot,\UNIT) = \KK_0^{\Grd\ltimes\Tot}(\UNIT_\Tot\otimes\UNIT_\Tot,\UNIT_\Tot)\) be the diagonal restriction class and
\[
\bar{\Theta} \defeq \forget_\Tot(\Theta) \in \KK^\Grd_n(\UNIT_\Tot,\Dual\otimes\UNIT_\Tot).
\]

\begin{definition}
  \label{def:Lefschetz_map}
  The equivariant \emph{Lefschetz map}
  \[
  \Lef\colon \RKK^\Grd_*(\Tot;\UNIT_\Tot,\UNIT) \to \KK^\Grd_*(\UNIT_\Tot,\UNIT)
  \]
  for a \(\Grd\)\nb-space~\(\Tot\) is defined as the composite map
  \[
  \RKK^\Grd_i(\Tot;\UNIT_\Tot,\UNIT) \xrightarrow{\PD^{-1}} \KK^\Grd_{i-n}(\Dual\otimes\UNIT_\Tot,\UNIT) \xrightarrow{\bar{\Theta}\otimes_{\Dual\otimes\UNIT_\Tot}\blank} \KK^\Grd_i(\UNIT_\Tot,\UNIT).
  \]
  The equivariant \emph{Euler characteristic} of~\(\Tot\) is
  \[
  \Eul_\Tot \defeq \Lef(\Delta_\Tot) \in \KK^\Grd_0(\UNIT_\Tot,\UNIT) = \KK^\Grd_0\bigl(\CONT_0(\Tot),\CONT_0(\Base)\bigr).
  \]
\end{definition}

Our definition of the equivariant Euler characteristic is literally the same as \cite{Emerson-Meyer:Euler}*{Definition 12} in the group case.

Let \(f\in \RKK^\Grd_i(\Tot;\UNIT_\Tot,\UNIT)\).  Equations \eqref{eq:inverse_PD_abstract} and~\eqref{eq:def_multiplication_class} yield
\begin{align}
  \label{eq:compute_Lef}
  \Lef(f) &= (-1)^{in} \bar{\Theta}
  \otimes_{\Dual\otimes\UNIT_\Tot} T_\Dual(f) \otimes_\Dual D,\\
  \label{eq:compute_Eul}
  \Eul_\Tot &= (-1)^{in} \bar{\Theta} \otimes_{\Dual\otimes\UNIT_\Tot} [m] \otimes_\Dual D.
\end{align}
If \((\Dual,\Theta)\) is part of a Kasparov dual, then \(T_\Dual=T_\Dual\) and \([m]\) is the \(\KK\)-class of the multiplication \(^*\)\nb-homomorphism \(\CONT_0(\Tot,\Dual)\to\Dual\), so that~\eqref{eq:compute_Lef} yields explicit formulas for \(\Lef(f)\) and \(\Eul_\Tot\).  These are applied in \cites{Emerson-Meyer:Euler, Emerson-Meyer:Equi_Lefschetz}.

In the group case, \cite{Emerson-Meyer:Euler}*{Proposition 13} asserts that the equivariant Euler characteristic does not depend on the abstract dual and is a proper homotopy invariant of~\(\Tot\).  This immediately extends to the groupoid case, and also to the Lefschetz map.  The most general statement requires some preparation.

Let \(\Tot\) and~\(\Tot'\) be \(\Grd\)\nb-spaces, and let \(f\colon \Tot\to\Tot'\) be a \(\Grd\)\nb-homotopy equivalence.  Then~\(f\) induces an equivalence of categories \(\RKKcat_\Grd(\Tot') \cong \RKKcat_\Grd(\Tot)\), that is, we get invertible maps
\[
f^*\colon \RKK^\Grd_*(\Tot';A,B) \to \RKK^\Grd_*(\Tot;A,B)
\]
for all \(\Grd\)\nb-\(\Cst\)\nb-algebras \(A\) and~\(B\).  Now assume, in addition, that~\(f\) is proper; we do not need the inverse map or the homotopies to be proper.  Then~\(f\) induces a \(^*\)\nb-homomorphism \(f^!\colon \CONT_0(\Tot') \to \CONT_0(\Tot)\), which yields \([f^!]\in \KK_0^\Grd\bigl(\CONT_0(\Tot'),\CONT_0(\Tot)\bigr)\).  We write \([f^!]\) instead of \([f^*]\) to better distinguish this from the map~\(f^*\) above.  Unless~\(f\) is a \emph{proper} \(\Grd\)\nb-homotopy equivalence, \([f^!]\) need not be invertible.

\begin{proposition}
  \label{pro:Lef_well-defined}
  Let \(\Tot\) and~\(\Tot'\) be \(\Grd\)\nb-spaces with abstract duals.  Let \(f\colon \Tot\to\Tot'\) be both a proper map and a \(\Grd\)\nb-homotopy equivalence.  Then
  \[
  [f^!]\otimes_{\CONT_0(\Tot)} \Eul_\Tot = \Eul_{\Tot'} \qquad \text{in \(\KK^\Grd_0(\CONT_0(\Tot'),\UNIT)\)}
  \]
  and the Lefschetz maps for \(\Tot\) and~\(\Tot'\) are related by a commuting diagram
  \[
  \xymatrix{ \RKK^\Grd_*(\Tot;\CONT_0(\Tot),\UNIT) \ar[d]^{\Lef_\Tot} & \ar[l]_{f^*}^{\cong} \RKK^\Grd_*(\Tot';\CONT_0(\Tot),\UNIT) \ar[r]^{[f^!]^*} &
    \RKK^\Grd_*(\Tot';\CONT_0(\Tot'),\UNIT) \ar[d]^{\Lef_{\Tot'}} \\
    \KK^\Grd_*(\CONT_0(\Tot),\UNIT) \ar[rr]^{[f^!]^*} && \KK^\Grd_*(\CONT_0(\Tot'),\UNIT), }
  \]
  where \([f^!]^*\) denotes composition with~\([f^!]\).

  In particular, \(\Eul_\Tot\) and the map \(\Lef_\Tot\) do not depend on the chosen dual.
\end{proposition}

\begin{proof}
  The assertion about Euler characteristics is a special case of the one about Lefschetz invariants because the proof of \cite{Emerson-Meyer:Euler}*{Proposition 13} shows that the diagonal restriction classes \(\Delta_\Tot\) and~\(\Delta_{\Tot'}\) are related by
  \[
  \Delta_{\Tot'} = [f^!] \otimes_{\CONT_0(\Tot)} (f^*)^{-1}(\Delta_\Tot).
  \]
  When we replace~\(\Delta_\Tot\) in the proof of \cite{Emerson-Meyer:Euler}*{Proposition 13} by a general element \(\alpha\in\RKK^\Grd_*(\Tot;\CONT_0(\Tot),\UNIT)\), then the same computations yield our assertion about the Lefschetz maps.
\end{proof}

Proposition~\ref{pro:Lef_well-defined} implies that the Lefschetz maps for properly \(\Grd\)\nb-homotopy equivalent spaces are equivalent because then \([f^!]\) is invertible, so that all horizontal maps in the diagram in Proposition~\ref{pro:Lef_well-defined} are invertible.  In this sense, the Lefschetz map and the Euler class are invariants of the proper \(\Grd\)\nb-homotopy type of~\(\Tot\).

The construction in Example~\ref{exa:diagonal_restrict_proper_map} associates a class \([\Delta_f] \in \RKK^\Grd_0(\Tot;\CONT_0(\Tot),\UNIT)\) to any continuous, \(\Grd\)\nb-equivariant map \(f\colon \Tot\to\Tot\); it does not matter whether~\(f\) is proper.  We abbreviate
\[
\Lef(f) \defeq \Lef([\Delta_f])
\]
and call this the Lefschetz invariant of~\(f\).  Of course, equivariantly homotopic self-maps induce the same class in \(\RKK^\Grd_0(\Tot;\CONT_0(\Tot),\UNIT)\) and therefore have the same Lefschetz invariant.  We have \(\Lef(\ID_\Tot)=\Eul_\Tot\).

Furthermore, the Kasparov product with~\(\Delta_\Tot\) provides a natural map
\[
\blank\otimes_{\UNIT_\Tot}\Delta_\Tot\colon \KK^\Grd_*\bigl(\CONT_0(\Tot),\CONT_0(\Tot)\bigr) \to \RKK^\Grd_*(\Tot;\CONT_0(\Tot),\UNIT),
\]
which we compose with the Lefschetz map to get a map
\[
\KK^\Grd_*\bigl(\CONT_0(\Tot),\CONT_0(\Tot)\bigr) \to \KK^\Grd_*(\CONT_0(\Tot),\UNIT)
\]
While elements of \(\KK^\Grd_*\bigl(\CONT_0(\Tot),\CONT_0(\Tot)\bigr)\) are the self-maps of \(\CONT_0(\Tot)\) in the category \(\KKcat_\Grd\), elements of \(\RKK^\Grd_*(\Tot;\CONT_0(\Tot),\UNIT)\) may be thought of as non-proper self-maps.

If~\(\Grd\) is a discrete group and \(f\colon \Tot\to\Tot\) is a \(\Grd\)\nb-equivariant continuous map, then its Lefschetz invariant is usually a combination of point evaluation classes, that is, \(\Lef(f)\) can be represented by an equivariant \(^*\)\nb-homomorphism \(\CONT_0(\Tot)\to \Comp(\Hils)\) for some \(\Z/2\)-graded \(\Grd\)\nb-Hilbert space~\(\Hils\).  \cite{Emerson-Meyer:Equi_Lefschetz}*{Theorems 1 and~2} assert this if~\(f\) a simplicial map on a simplicial complex or if~\(f\) is a smooth map on a smooth manifold whose graph is transverse to the diagonal.  Dropping the transversality condition, \(\Lef(f)\) is an equivariant Euler characteristic of its fixed point subset, twisted by a certain orientation line bundle.

In contrast, Lefschetz invariants for general elements of \(\RKK^\Grd_*(\Tot;\CONT_0(\Tot),\UNIT)\) may be arbitrarily complicated:

\begin{proposition}
  \label{pro:Lef_split_surjective}
  The composition
  \[
  \KK^\Grd_*(\CONT_0(\Tot),\UNIT) \xrightarrow{p_\Tot^*} \RKK^\Grd_*(\Tot;\CONT_0(\Tot),\UNIT) \xrightarrow{\Lef} \KK^\Grd_*(\CONT_0(\Tot),\UNIT)
  \]
  is the identity map.
\end{proposition}

\begin{proof}
  Let \(\alpha\in \KK^\Grd_*(\CONT_0(\Tot),\UNIT)\).  We check \(\Lef\bigl(p_\Tot^*(\alpha)\bigr) = \alpha\).  Let \(D\in\KK^\Grd_{-n}(\Dual,\UNIT)\) be the counit of the duality.  Then \(\PD(D\otimes\alpha) = \Theta \otimes_\Dual D\otimes\alpha = p_\Tot^*(\alpha)\).  Therefore,
  \[
  \Lef\bigl(p_\Tot^*(\alpha)\bigr) = \bar{\Theta} \otimes_{\Dual\otimes\CONT_0(\Tot)} \PD^{-1}\bigl(p_\Tot^*(\alpha)\bigr) = \bar{\Theta} \otimes_{\Dual\otimes\CONT_0(\Tot)} D\otimes\alpha = \alpha
  \]
  because \(\bar{\Theta}\otimes_\Dual D = \overline{\Theta\otimes_\Dual D} = \overline{\ID_{\CONT_0(\Tot)}} = \ID_{\CONT_0(\Tot)}\).
\end{proof}

\subsubsection{Mapping to topological K-theory}
\label{sec:Lef_to_Ktop}

We briefly explain an approach to extract numerical invariants out of Lefschetz invariants and Euler characteristics.

The topological \(\K\)\nb-theory of~\(\Grd\) may be defined as the inductive limit
\[
\Ktop_*(\Grd) = \varinjlim_\Tot \KK^\Grd_*(\CONT_0(\Tot),\UNIT),
\]
where~\(\Tot\) runs through the category of proper \(\Grd\)\nb-compact \(\Grd\)\nb-spaces with homotopy classes of \(\Grd\)\nb-equivariant continuous maps as morphisms.  If \(\EG\Grd\) is a universal proper \(\Grd\)\nb-space, we may replace this category by the directed set of \(\Grd\)\nb-compact \(\Grd\)\nb-invariant subsets of \(\EG\Grd\), which is cofinal in the above category.

Therefore, if~\(\Tot\) is proper and \(\Grd\)\nb-compact and has an abstract dual, we may map \(\Lef(\alpha)\) for \(\alpha\in\RKK^\Grd_*(\Tot;\CONT_0(\Tot),\UNIT)\) to an element of \(\Ktop_*(\Grd)\).  A transverse measure on~\(\Grd\) induces a trace map \(\tau\colon \Ktop_0(\Grd)\to\R\).  It is justified to call \(\tau\bigl(\Lef(\alpha)\bigr)\) the \emph{\(L^2\)\nobreakdash-Lefschetz number of~\(\alpha\)} and \(\tau(\Eul_\Tot)\) the \emph{\(L^2\)\nobreakdash-Euler characteristic of~\(\Tot\)}; equation~\eqref{eq:euler:foliation_deRham_ltwo_version} shows that the resulting \(L^2\)\nobreakdash-Euler characteristic is the alternating sum of the \(L^2\)\nb-Betti numbers and hence agrees with the \(L^2\)\nobreakdash-Euler characteristic studied by Alain Connes in~\cite{Connes:Survey_foliations}.

\subsection{Duality for universal proper actions}
\label{sec:dual_EG}

Now we consider the special case where~\(\Tot\) is a universal proper \(\Grd\)\nb-space.  More precisely, we assume that the two coordinate projections \(\Tot\times_\Base\Tot\rightrightarrows\Tot\) are \(\Grd\)\nb-equivariantly homotopic; equivalently, any two \(\Grd\)\nb-maps to~\(\Tot\) are \(\Grd\)\nb-equivariantly homotopic, which is a necessary condition for a proper action to be universal.  Here a simplification occurs because~\eqref{eq:Kasparov_dual_2} is automatic:

\begin{lemma}
  \label{lem:dual_2_automatic}
  Assume that the two coordinate projections \(\Tot\times_\Base\Tot\rightrightarrows\Tot\) are \(\Grd\)\nb-equivariantly homotopic.  Let \(A\) and~\(B\) be \(\Grd\)\nb-\(\Cst\)-algebras, let~\(\Dual\) be a \(\Grd\ltimes\Tot\)-\(\Cst\)-algebra, and let \(f\in\RKK^\Grd_*(\Tot;A,B)\).  Then
  \[
  f\otimes\ID_\Dual = p_\Tot^*\bigl(T_\Dual(f)\bigr)
  \qquad\text{in \(\RKK^\Grd_*(\Tot;\Dual\otimes A,\Dual\otimes B)\).}
  \]
  As a consequence, \(\Theta \otimes_\Tot f = \Theta \otimes_\Dual T_\Dual(f)\) for any \(\Theta\in \RKK^\Grd_*(\Tot;\UNIT,\Dual)\).
\end{lemma}

\begin{proof}
  Since the coordinate projections \(\pi_1,\pi_2\colon \Tot\times_\Base\Tot\rightrightarrows \Tot\) are \(\Grd\)\nb-equivariantly homotopic, we have \(\pi_1^*(f) = \pi_2^*(f)\) in \(\RKK^\Grd_*(\Tot\times_\Base\Tot;A,B)\).  Now tensor this over the second space~\(\Tot\) with~\(\Dual\) and forget the equivariance in this direction.  The resulting classes \(T_\Dual\bigl(\pi_1^*(f)\bigr)\) and \(T_\Dual\bigl(\pi_2^*(f)\bigr)\) in \(\RKK^\Grd_*(\Tot;\Dual\otimes A,\Dual\otimes B)\) are still equal.  The first one is \(f\otimes\ID_\Dual\), the second one is \(p_\Tot^*\bigl(T_\Dual(f)\bigr)\).  Finally, we observe that \(\Theta\otimes_\Tot f = \Theta\otimes_{\Tot,\Dual} (f\otimes\ID_\Dual)\) and \(\Theta\otimes_\Dual T_\Dual(f) = \Theta\otimes_{\Tot,\Dual} p_\Tot^*\bigl(T_\Dual(f)\bigr)\).
\end{proof}

Thus, we only need the two conditions \eqref{eq:Kasparov_dual_1} and~\eqref{eq:Kasparov_dual_3} for \((\Dual,D,\Theta)\) to be a Kasparov dual for~\(\Tot\).

\begin{theorem}
  \label{the:dual_EG}
  Let \(\EG\Grd\) be a universal proper \(\Grd\)\nb-space and let \((\Dual,D,\Theta)\) be an \(n\)\nb-dimensional Kasparov dual for \(\EG\Grd\).  Then
  \[
  \Theta \in \RKK^\Grd_n(\EG\Grd;\UNIT,\Dual) \qquad\text{and}\qquad p_{\EG\Grd}^*(D) \in \RKK^\Grd_{-n}(\EG\Grd;\Dual,\UNIT)
  \]
  are inverse to each other, and so are
  \[
  \comul\in \KK^\Grd_n(\Dual,\Dual\otimes\Dual) \qquad\text{and}\qquad \ID_\Dual\otimes D=(-1)^n D\otimes \ID_\Dual \in \KK^\Grd_{-n}(\Dual\otimes\Dual,\Dual).
  \]
  Thus the functor \(A\mapsto \Dual\otimes A\) is idempotent up to a natural isomorphism in \(\KKcat_\Grd\) and the class in \(\KK_0(\Dual\otimes\Dual,\Dual\otimes\Dual)\) of the flip automorphism on \(\Dual\otimes\Dual\) is \((-1)^n\).
\end{theorem}

\begin{proof}
  The pull-back \((p_{\EG\Grd}^*\Dual,p_{\EG\Grd}^*\Theta)\) is a Kasparov dual for \(p_{\EG\Grd}^*\EG\Grd \cong \EG\Grd\times_\Base\EG\Grd\) over~\(\EG\Grd\), where \(\EG\Grd\times_\Base\EG\Grd\) is a space over \(\EG\Grd\) via a coordinate projection \(\pi\colon\EG\Grd\times_\Base \EG\Grd \to \EG\Grd\) and
  \[
  p_{\EG\Grd}^*\Theta\in \RKK^\Grd_n(\EG\Grd\times_\Base\EG\Grd; \UNIT,\Dual) \cong \RKK^{\Grd\ltimes\EG\Grd}_n(\EG\Grd; p_{\EG\Grd}^*\UNIT,p_{\EG\Grd}^*\Dual).
  \]
  The reason for this is that \eqref{eq:Kasparov_dual_1} and~\eqref{eq:Kasparov_dual_3} are preserved by the base change~\(p_{\EG\Grd}^*\), while~\eqref{eq:Kasparov_dual_2} is automatic by Lemma~\ref{lem:dual_2_automatic}.

  As a result, \(p_{\EG\Grd}^*\Theta\) induces isomorphisms
  \[
  \RKK^{\Grd\ltimes \EG\Grd}_*(\EG\Grd;A,B) \cong \KK^{\Grd\ltimes \EG\Grd}_*(\Dual\otimes A, B)
  \]
  for all \(\Grd\ltimes\EG\Grd\)-\(\Cst\)-algebras \(A\) and~\(B\); recall that \(p_{\EG\Grd}^*\Dual\otimes_{\EG\Grd} A \cong \Dual\otimes A\).

  The universal property of~\(\EG\Grd\) implies that the projection~\(\pi\) is a \(\Grd\)\nb-homotopy equivalence.  Hence it induces isomorphisms
  \begin{equation}
    \label{eq:double_EG_iso}
    \KK^{\Grd\ltimes\EG\Grd}_*(A,B) \cong
    \RKK^{\Grd\ltimes\EG\Grd}_*(\EG\Grd;A,B).
  \end{equation}
  Both isomorphisms together show that \(\Dual\cong\UNIT\) in \(\RKKcat_\Grd(\EG\Grd)\).  More precisely, inspection shows that the invertible elements in \(\RKK^\Grd_{-n}(\EG\Grd;\Dual,\UNIT)\) and \(\RKK^\Grd_n(\EG\Grd;\UNIT,\Dual)\) that we get are \(p_{\EG\Grd}^*D\) and~\(\Theta\).

  To get the remaining assertions, we apply the functor \(T_\Dual\).  This shows that \(\comul=T_\Dual(\Theta)\) and \(\ID\otimes D = T_\Dual\bigl(p_{\EG\Grd}^*(D)\bigr)\) are inverse to each other.  Thus \(\Dual\otimes\Dual\cong\Dual\) in \(\KKcat_\Grd\).  We get \(\ID_\Dual\otimes D = (-1)^n D\otimes\ID_\Dual\) because both sides are inverses for~\(\comul\) by~\eqref{eq:counit}, and \(\flip=(-1)^n\) follows from~\eqref{eq:Kasparov_dual_3}.
\end{proof}

\begin{theorem}
  \label{the:dual_EG2}
  Let \(\EG\Grd\) be a universal proper \(\Grd\)\nb-space and let \((\Dual,D,\Theta)\) be a \(0\)\nb-dimensional Kasparov dual for \(\EG\Grd\).  Let~\(A\) be a \(\Grd\)\nb-\(\Cst\)\nb-algebra.  The following assertions are equivalent:
  \begin{enumerate}[label=\textup{(\arabic{*})}]
  \item \(D\otimes\ID_A\) is invertible in \(\KK^\Grd_0(\Dual\otimes A,A)\);
  \item \(A\cong \Dual\otimes A\) in \(\KKcat_\Grd\);
  \item \(A\) is \(\KK^\Grd\)-equivalent to a proper \(\Grd\)\nb-\(\Cst\)\nb-algebra, that is, \(A\cong \forget_{\EG\Grd}(\hat{A})\) for some \(\Grd\ltimes\EG\Grd\)-\(\Cst\)\nb-algebra~\(\hat{A}\);
  \item the map
    \[
    p_{\EG\Grd}^*\colon \KK^\Grd_*(A,B) \to \RKK^\Grd_*(\EG\Grd;A,B)
    \]
    is invertible for all \(\Grd\)\nb-\(\Cst\)\nb-algebras~\(B\).
  \end{enumerate}
\end{theorem}

\begin{proof}
  The implications (1)\(\Longrightarrow\)(2)\(\Longrightarrow\)(3) are trivial because \(\Dual\otimes A\) is a proper \(\Grd\)\nb-\(\Cst\)\nb-algebra.

  We prove (3)\(\Longrightarrow\)(1).  By definition, a proper \(\Grd\)\nb-\(\Cst\)\nb-algebra is a \(\Grd\ltimes\Tot\)-\(\Cst\)\nb-algebra for some proper \(\Grd\)\nb-space~\(\Tot\).  Since there is a \(\Grd\)\nb-map \(\Tot\to\EG\Grd\), we may view any \(\Grd\ltimes\Tot\)\nb-\(\Cst\)\nb-algebra as a \(\Grd\ltimes\EG\Grd\)-\(\Cst\)\nb-algebra and thus assume \(A=\forget_{\EG\Grd}(\hat{A})\).  Then
  \[
  p_{\EG\Grd}^*(D) \otimes_{\EG\Grd} \ID_{\hat{A}} \in \KK_0^{\Grd\ltimes \EG\Grd} (p_{\EG\Grd}^*(\Dual)\otimes_{\EG\Grd}\hat{A}, \hat{A})
  \]
  is invertible because \(p_{\EG\Grd}^*(D)\) is.  Now identify \(p_{\EG\Grd}^*(\Dual)\otimes_{\EG\Grd}\hat{A} \cong \Dual\otimes\hat{A}\) and forget the \(\EG\Grd\)-structure to see that \(D\otimes\ID_A\) in \(\KK_0^\Grd(\Dual\otimes A,A)\) is invertible.

  Finally, we prove (1)\(\iff\)(4).  For all \(\Grd\)\nb-\(\Cst\)\nb-algebras \(A\) and~\(B\), the diagram
  \begin{equation}
    \label{eq:localise_sour}
    \begin{gathered}
      \xymatrix@R+1em@C+2em{ \KK^\Grd_*(A,B) \ar[r]^{D\otimes\blank} \ar[d]^{p_{\EG\Grd}^*} & \KK^\Grd_*(\Dual \otimes A,B)
        \ar[dl]_{\PD}^{\cong} \\
        \RKK^\Grd_*(\EG\Grd;A,B) }
    \end{gathered}
  \end{equation}
  commutes because \(\Theta\otimes_\Dual D=1\).  By the Yoneda Lemma, \(D\otimes\ID_A\) is invertible if and only if the horizontal arrow is invertible for all~\(B\).  Since the diagonal arrow is invertible, this is equivalent to the vertical arrow being invertible for all~\(B\), that is, to~(4).
\end{proof}

Theorems \ref{the:dual_EG} and~\ref{the:dual_EG2} are important in connection with the localisation approach to the Baum--Connes assembly map developed in~\cite{Meyer-Nest:BC}, as we now explain.

\begin{definition}
  \label{def:CC}
  Let~\(\EG\Grd\) be a universal proper \(\Grd\)\nb-space.  We define two subcategories of \(\KKcat_\Grd\):
  \begin{align*}
    \mathcal{CC} &\defeq \{A\inOb\KKcat^\Grd\mid \text{\(p_{\EG\Grd}^*(A) \cong 0\) in \(\RKKcat^\Grd(\EG\Grd)\)}\},
    \\
    \mathcal{CP} &\defeq \{A\inOb\KKcat^\Grd\mid \text{\(A\) is \(\KK^\Grd\)-equivalent to a proper \(\Grd\)-\(\Cst\)\nb-algebra}\}.
  \end{align*}
\end{definition}

\begin{corollary}[compare \cite{Meyer-Nest:BC}*{Theorem 7.1}]
  \label{cor:complementary_CC_CI}
  Let \(\EG\Grd\) be a universal proper \(\Grd\)\nb-space and suppose that \(\EG\Grd\) has a \(0\)\nb-dimensional Kasparov dual \((\Dual,D,\Theta)\).  Then the pair of subcategories \((\mathcal{CP},\mathcal{CC})\) is complementary.  The localisation functor \(\KKcat_\Grd\to\mathcal{CP}\) is \(A\mapsto \Dual\otimes A\), and the natural transformation from this functor to the identity functor is induced by~\(D\).  The localisation of \(\KKcat_\Grd\) at \(\mathcal{CC}\) is isomorphic to \(\RKKcat_\Grd(\EG\Grd)\) with the functor \(p_{\EG\Grd}^*\colon \KKcat_\Grd\to\RKKcat_\Grd(\EG\Grd)\).
\end{corollary}

\begin{proof}
  Let~\(L\) belong to \(\mathcal{CP}\) and~\(C\) belong to \(\mathcal{CC}\).  Then \(L\cong \Dual\otimes L\) by Theorem~\ref{the:dual_EG2}.  Hence
  \[
  \KK^\Grd_0(L,C) \cong \KK^\Grd_0(\Dual\otimes L,C) \cong \RKK^\Grd_0(\EG\Grd;L,C) = \KK^{\Grd\ltimes\EG\Grd}_0 \bigl(p_{\EG\Grd}^*(L),p_{\EG\Grd}^*(C)\bigr) =0.
  \]
  Thus \(\mathcal{CP}\) is left orthogonal to \(\mathcal{CC}\).

  Let~\(A\) be a \(\Grd\)\nb-\(\Cst\)\nb-algebra.  The cone of \(D\otimes\ID_A\colon \Dual\otimes A\to A\) (mapping cone in the sense of triangulated categories) belongs to \(\mathcal{CC}\) because \(p_{\EG\Grd}^*(D\otimes\ID_A)\) is invertible in \(\RKKcat_\Grd(\EG\Grd)\) by Theorem~\ref{the:dual_EG}.  Hence any object~\(A\) belongs to an exact triangle \(L\to A\to C\to L[1]\) with \(L\in\mathcal{CP}\), \(C\in\mathcal{CC}\), where we take \(L = \Dual\otimes A\) and the map \(L\to A\) induced by~\(D\).  Thus \((\mathcal{CP},\mathcal{CC})\) is a complementary pair of subcategories.

  In the localisation of \(\KKcat_\Grd\) at \(\mathcal{CC}\), the morphism groups are \(\KK^\Grd_0(A\otimes\Dual,B)\).  These are identified with \(\RKK^\Grd_0(\EG\Grd;A,B)\) by the first Poincar\'e duality isomorphism.  Hence the localisation is equivalent to \(\RKKcat_\Grd(\EG\Grd)\).  The commuting diagram~\eqref{eq:localise_sour} shows that the localisation functor becomes \(p_{\EG\Grd}^*\).
\end{proof}

Let~\(G\) be a group.  In~\cite{Meyer-Nest:BC}, the analogues of the categories \(\mathcal{CP}\) and \(\mathcal{CC}\) are defined slightly differently: for \(\mathcal{CC}\), it is only required that \(p_{G/H}^*(A)\cong0\) for all compact subgroups \(H\subseteq G\), and \(\mathcal{CP}\) is replaced by the triangulated subcategory generated by objects of the form \(\forget_{G/H}(\hat{A})\).  We have not yet tried to construct proper actions of groupoids out of simpler building blocks in a similar way.

Next we relate Kasparov duality to the Dirac dual Dirac method for groupoids.

\begin{definition}
  \label{def:dual_Dirac}
  An \(n\)\nb-dimensional \emph{Dirac-dual-Dirac triple} for the groupoid~\(\Grd\) is a triple \((\Dual, D, \eta)\) where~\(\Dual\) is a \(\Grd\ltimes \EG\Grd\)\nb-$\Cst$-algebra, \(D\in \KK^\Grd_{-n} (\Dual, \UNIT)\), and \(\eta\in\KK^\Grd_n(\UNIT,\Dual)\), such that \(p_{\EG\Grd}^*(\eta\otimes_\Dual D)=1_{\EG \Grd}\) in \(\RKK^\Grd_0(\EG\Grd;\UNIT,\UNIT)\) and \(D\otimes\eta = \ID_\Dual\) in \(\KK^\Grd_0(\Dual,\Dual)\).
\end{definition}

The two conditions \(p_{\EG\Grd}^*(\eta\otimes_\Dual D)=1_{\EG \Grd}\) and \(D\otimes\eta = \ID_\Dual\) are independent: if both hold, then we may violate the second one without violating the first one by adding some proper \(\Grd\)\nb-\(\Cst\)-algebra to~\(\Dual\) and taking \((D,0)\) and \((\eta,0)\); and the second one always holds if we take \(\Dual=0\).

\begin{theorem}
  \label{the:existence_of_Kasparov_dual_for_EG}
  Let \((\Dual, D, \eta)\) be an \(n\)\nb-dimensional Dirac-dual-Dirac triple.  Let \(\Theta\defeq p_{\EG\Grd}^*(\eta)\) and let \(\gamma \defeq \eta \otimes_\Dual D\).  Then \(\bigl(\Dual, D, \Theta\bigr)\) is an \(n\)\nb-dimensional Kasparov dual for~\(\EG\Grd\).  Furthermore, \(\gamma\) is an idempotent element of the ring \(\KK^\Grd_0(\UNIT,\UNIT)\).  This ring acts naturally on all \(\KK^\Grd\)-groups by exterior product.  The map
 \[
 p_{\EG\Grd}^*\colon \KK^\Grd_*(A,B) \to \RKK^\Grd_*(\EG\Grd;A,B)
 \]
 vanishes on the kernel of~\(\gamma\) and restricts to an isomorphism on the range of~\(\gamma\).  Its inverse is the map
 \[
 \RKK^\Grd_i(\EG\Grd;A,B) \to \gamma\cdot \KK^\Grd_i(A,B),\qquad
 \alpha\mapsto (-1)^{in}\eta \otimes_\Dual T_\Dual(\alpha)\otimes_\Dual D.
 \]
\end{theorem}

\begin{proof}
  Condition~\eqref{eq:Kasparov_dual_1} amounts to our assumption \(p_{\EG \Grd}^*(\gamma) = 1_{\EG \Grd}\).  Clearly, the exterior product \(\eta\otimes\eta\) is invariant under~\(\flip_\Dual\) up to the sign~\((-1)^n\).  Since \(\eta\otimes\eta = \eta\otimes_\Dual T_\Dual(\eta)\) and \(D\otimes\eta=\ID_\Dual\), this implies~\eqref{eq:Kasparov_dual_3}.  Condition~\eqref{eq:Kasparov_dual_2} is automatic by Lemma~\ref{lem:dual_2_automatic}.  Hence \((\Dual,D,\Theta)\) is a Kasparov dual for~\(\EG\Grd\).  The assumption \(D\otimes\eta=\ID_\Dual\) implies that~\(\gamma\) is idempotent.

  If \(f\in\KK^\Grd_i(A,B)\), then
  \begin{multline*}
    (-1)^{in}\eta\otimes_\Dual T_\Dual\bigl(p_{\EG\Grd}^*(f)\bigr) \otimes_\Dual D
    \\= (-1)^{in}\eta\otimes_\Dual (\ID_\Dual \otimes f) \otimes_\Dual D
    = \eta\otimes_\Dual D \otimes f = \gamma\cdot f.
  \end{multline*}
  If \(f\in\RKK^\Grd_i(\EG\Grd;A,B)\), then Lemma~\ref{lem:dual_2_automatic} implies
  \begin{multline*}
    (-1)^{in} p_{\EG\Grd}^*(\eta\otimes_\Dual T_\Dual(f) \otimes_\Dual D)
    = (-1)^{in} \Theta\otimes_\Dual T_\Dual(f) \otimes_\Dual D
    \\= (-1)^{in} \Theta\otimes_{\EG\Grd} f \otimes_\Dual D
    = f \otimes_{\EG\Grd} \Theta\otimes_\Dual D = f.
  \end{multline*}
  The remaining assertions follow.
\end{proof}

\subsection{Extension to non-trivial bundles}
\label{sec:KK_first_non-trivial}

Let \((\Dual,D,\Theta)\) be an \(n\)\nb-dimensional Kasparov dual for~\(\Tot\).  The functor~\(T_\Dual\) extends to a functor
\[
T_\Dual\colon \KK^{\Grd\ltimes\Tot}_*(A,B) \to \KK^\Grd_*(\Dual\otimes_\Tot A,\Dual \otimes_\Tot B)
\]
for all \(\Grd\ltimes\Tot\)-\(\Cst\)\nb-algebras \(A\) and~\(B\), combining the tensor product over~\(\Tot\) with~\(\Dual\) and \(\forget_\Tot\).  If \(B=p_\Tot^*(B_0) = \CONT_0(\Tot)\otimes B_0\), then we can simplify \(\Dual\otimes_\Tot B \cong \Dual\otimes B_0\).  Extending the definition in Theorem~\ref{the:first_duality}, we get a natural transformation
\begin{multline}
  \label{eq:first_Kasparov_bundle}
  \PD^*\colon \KK^{\Grd\ltimes\Tot}_i(A,p_\Tot^*B) \xrightarrow{(-1)^{in}T_\Dual} \KK^\Grd_i(\Dual\otimes_\Tot A, \Dual\otimes B) \\
  \xrightarrow{\blank\otimes_\Dual D} \KK^\Grd_{i-n}(\Dual\otimes_\Tot A, B)
\end{multline}
if~\(A\) is a \(\Grd\ltimes\Tot\)-\(\Cst\)\nb-algebra and~\(B\) is a \(\Grd\)\nb-\(\Cst\)\nb-algebra.  This map fails to be an isomorphism in the following simple counterexample:

\begin{example}
  \label{exa:first_Kasparov_fails}
  Let~\(\Grd\) be trivial, take a \(\Cst\)\nb-algebra~\(A\), and view it as a \(\Cst\)\nb-algebra over~\(\Tot\) concentrated in some \(\tot\in\Tot\).  Unless~\(\tot\) is isolated, the only \(\Tot\)\nb-linear Kasparov cycle for \(A\) and \(\CONT_0(\Tot,B)\) is the zero cycle, so that \(\KK^\Tot_*\bigl(A,\CONT_0(\Tot, B)\bigr)=0\).  But there is no reason for \(\KK_*(A\otimes_\Tot \Dual, B)\) to vanish because \(A\otimes_\Tot \Dual= A\otimes \Dual_\tot\).
\end{example}

The following theorem gives necessary and sufficient conditions for~\eqref{eq:first_Kasparov_bundle} to be an isomorphism.  The first results of this kind appeared in \cite{Echterhoff-Emerson-Kim:Duality} and~\cite{Tu:Twisted_Poincare}.

\begin{theorem}
  \label{the:KK_first_non-trivial}
  Let \(\Dual\) and~\(A\) be \(\Grd\ltimes\Tot\)\nb-\(\Cst\)-algebras.  Let \(\Theta\in\RKK^\Grd_n(\Tot;\UNIT,\Dual)\) and \(D\in\KK^\Grd_{-n}(\Dual,\UNIT)\) satisfy \eqref{eq:Kasparov_dual_1} and~\eqref{eq:other_third_condition} \textup(both are necessary conditions for Kasparov duals\textup).  The map \(\PD^*\) in~\eqref{eq:first_Kasparov_bundle} is invertible for all \(\Grd\)\nb-\(\Cst\)\nb-algebras~\(B\) if and only if there is
  \[
  \Theta_A\in \KK^{\Grd\ltimes\Tot}_n \bigl(A, p_\Tot^*(\Dual \otimes_\Tot A)\bigr)
  \]
  such that the diagram
  \begin{equation}
    \label{eq:KK_first_non-trivial_3}
    \begin{gathered}
      \xymatrix@C+2em{ \Dual\otimes_\Tot A \ar[r]^{T_\Dual(\Theta_A)} \ar[d]_{T_\Dual(\Theta\otimes_\Tot \ID_A)} & \Dual\otimes (\Dual\otimes_\Tot A)
        \ar[dl]^{(-1)^n\flip}\\
        (\Dual\otimes_\Tot A) \otimes \Dual }
    \end{gathered}
  \end{equation}
  in \(\KKcat_\Grd\) commutes and, for all \(\alpha\in \KK^{\Grd\ltimes\Tot}_i \bigl(A, p_\Tot^*(B)\bigr)\),
  \begin{equation}
    \label{eq:KK_first_non-trivial_4}
    \Theta_A \otimes_{\Dual\otimes_\Tot A} T_\Dual(\alpha) = \Theta \otimes_\Tot\alpha
    \qquad\text{in \(\KK^{\Grd\ltimes\Tot}_{i+n} \bigl(A, p_\Tot^*(\Dual\otimes B)\bigr)\).}
  \end{equation}
  There is at most one element~\(\Theta_A\) with these properties, and if it exists, then the inverse isomorphism to~\eqref{eq:first_Kasparov_bundle} is the map
  \[
  \PD\colon \KK^\Grd_*(\Dual\otimes_\Tot A, B) \to \KK^{\Grd\ltimes\Tot}_{*+n}\bigl(A, p_\Tot^*(B)\bigr), \qquad \alpha \mapsto \Theta_A\otimes_{\Dual\otimes_\Tot A}\alpha.
  \]
\end{theorem}

\begin{proof}
  If there is~\(\Theta_A\) with the required properties, then the following routine computations show that the maps \(\PD^*\) and \(\PD\) defined above are inverse to each other.  Starting with \(\alpha\in \KK^{\Grd\ltimes\Tot}_i \bigl(A, p_\Tot^*(B)\bigr)\), we compute
  \begin{multline*}
    \PD\circ\PD^*(\alpha) \defeq (-1)^{in} \Theta_A \otimes_{\Dual\otimes_\Tot A} T_\Dual(\alpha)\otimes_\Dual D
    \\= (-1)^{in} \Theta \otimes_\Tot \alpha\otimes_\Dual D
    = \alpha \otimes_\Tot \Theta\otimes_\Dual D = \alpha,
  \end{multline*}
  using \eqref{eq:KK_first_non-trivial_4}, graded commutativity of exterior products, and~\eqref{eq:Kasparov_dual_1}.  Starting with \(\beta\in \KK^\Grd_{i-n}(\Dual\otimes_\Tot A,B)\), we compute
  \begin{align*}
    \PD^*\circ\PD(\beta) &\defeq (-1)^{in} T_\Dual(\Theta_A \otimes_{\Dual\otimes_\Tot A}\beta) \otimes_\Dual D \\
    &= (-1)^{in} T_\Dual(\Theta_A) \otimes_{\Dual\otimes_\Tot A}\beta \otimes_\Dual D \\
    &= (-1)^{in+n} T_\Dual(\Theta\otimes_\Tot \ID_A) \otimes_{(\Dual\otimes_\Tot A)\otimes\Dual} (\beta\otimes D) \\
    &= T_\Dual(\Theta\otimes_\Tot \ID_A) \otimes_{\Dual\otimes(\Dual\otimes_\Tot A)} (D\otimes\beta) \\
    &= T_\Dual\bigl((\Theta\otimes_\Dual D)\otimes_\Tot \ID_A\bigr) \otimes_{\Dual\otimes_\Tot A}\beta = \beta,
  \end{align*}
  using \eqref{eq:KK_first_non-trivial_3}, graded commutativity of exterior products and~\eqref{eq:Kasparov_dual_1}.  Hence our two maps are inverse to each other.

  Now suppose, conversely, that \(\PD^*\) is an isomorphism for all~\(B\).  Equations \eqref{eq:KK_first_non-trivial_3} and~\eqref{eq:Kasparov_dual_1} imply
  \begin{equation}
    \label{eq:define_Theta_A}
    \PD^*(\Theta_A)= (-1)^n T_\Dual(\Theta_A) \otimes_\Dual D
    = T_\Dual(\Theta \otimes_\Tot \ID_A) \otimes_\Dual D
    = \ID_{\Dual\otimes_\Tot A}.
  \end{equation}
  Hence there is at most once choice for~\(\Theta_A\), namely, the unique pre-image of the identity map on \(\Dual\otimes_\Tot A\).  The inverse map \(\PD\) of \(\PD^*\) must have the asserted form by naturality.  We claim that the above choice of~\(\Theta_A\) satisfies \eqref{eq:KK_first_non-trivial_3} and~\eqref{eq:KK_first_non-trivial_4}.

  For~\eqref{eq:KK_first_non-trivial_3}, we compute the image of
  \[
  \Theta_A \otimes_\Tot \Theta = (-1)^n \Theta \otimes_\Tot \Theta_A \in \KK^{\Grd\ltimes\Tot}_{2n}(A, p_\Tot^*(\Dual\otimes_\Tot A)\otimes\Dual)
  \]
  under~\(\PD^*\) in two different ways.  On the one hand,
  \begin{multline*}
    \PD^*(\Theta_A\otimes_\Tot \Theta) \defeq T_\Dual(\Theta_A\otimes_\Tot\Theta)\otimes_\Dual D \\
    = T_\Dual(\Theta_A)\otimes_\Dual T_\Dual(\Theta) \otimes_\Dual D = (-1)^n T_\Dual(\Theta_A),
  \end{multline*}
  using that~\(T_\Dual\) is functorial and~\eqref{eq:other_third_condition}.  On the other hand,
  \begin{multline*}
    \PD^*(\Theta\otimes_\Tot \Theta_A) \defeq T_\Dual(\Theta\otimes_\Tot\Theta_A)\otimes_\Dual D \\
    = T_\Dual(\Theta\otimes_\Tot\ID_A)\otimes_{\Dual\otimes_\Tot A} T_\Dual(\Theta_A) \otimes_\Dual D = (-1)^n T_\Dual(\Theta\otimes_\Tot\ID_A)
  \end{multline*}
  by~\eqref{eq:define_Theta_A}.  Hence \(T_\Dual(\Theta\otimes_\Tot\ID_A)\) and \(T_\Dual(\Theta_A)\) agree up to the sign \((-1)^n\) and the flip of the tensor factors, which we have ignored in the above computation.

  Now we check~\eqref{eq:KK_first_non-trivial_4}.  Let \(\alpha\in \KK^{\Grd\ltimes\Tot}_i\bigl(A, p_\Tot^*(B)\bigr)\).  Then
  \begin{multline*}
    \PD^*(\Theta\otimes_\Tot\alpha)
    = \PD^*\bigl((-1)^{in}\alpha\otimes_\Tot\Theta\bigr)
    = T_\Dual(\alpha\otimes_\Tot\Theta) \otimes_\Dual D
    \\= T_\Dual(\alpha)\otimes_\Dual T_\Dual(\Theta) \otimes_\Dual D
    = T_\Dual(\alpha)\otimes_\Dual \comul \otimes_\Dual D
    = (-1)^n T_\Dual(\alpha).
  \end{multline*}
  The graded commutativity of exterior products yields
  \begin{multline*}
    \PD^*\bigl(\Theta_A \otimes_{\Dual\otimes_\Tot A} T_\Dual(\alpha) \bigr) = (-1)^{in} T_\Dual(\Theta_A) \otimes_{\Dual\otimes_\Tot A} T_\Dual(\alpha) \otimes_\Dual D \\
    = T_\Dual(\Theta_A) \otimes_\Dual D \otimes_{\Dual\otimes_\Tot A} T_\Dual(\alpha) = (-1)^n T_\Dual(\alpha),
  \end{multline*}
  again using~\eqref{eq:define_Theta_A}.  These computations imply~\eqref{eq:KK_first_non-trivial_4} because \(\PD^*\) is injective.
\end{proof}

\begin{remark}
  \label{rem:KK_first_non-trivial}
  The conditions \eqref{eq:KK_first_non-trivial_3} and~\eqref{eq:KK_first_non-trivial_4} are related: we claim that~\eqref{eq:KK_first_non-trivial_3} already implies
  \begin{equation}
    \label{eq:KK_first_1b}
    T_\Dual\bigl(\Theta_A \otimes_{\Dual\otimes_\Tot A}
    T_\Dual(\alpha)\bigr)
    = (-1)^{in} T_\Dual(\alpha \otimes_\Tot \Theta)
  \end{equation}
  in \(\KK^\Grd_{i+n}(\Dual\otimes_\Tot A, \Dual\otimes\Dual\otimes B)\) for all \(\alpha\in \KK^{\Grd\ltimes\Tot}_i \bigl(A, p_\Tot^*(B)\bigr)\).  If the first Poincar\'e duality map in~\eqref{eq:first_Kasparov_bundle} is an isomorphism, then~\(T_\Dual\) must be injective, so that~\eqref{eq:KK_first_1b} yields~\eqref{eq:KK_first_non-trivial_4}.  Hence~\eqref{eq:KK_first_non-trivial_4} is equivalent to injectivity of~\(T_\Dual\) on suitable groups.  This also applies to the second condition~\eqref{eq:Kasparov_dual_2} in the definition of a Kasparov dual: this is just the special case of Theorem~\ref{the:KK_first_non-trivial} where~\(A\) is a trivial bundle over~\(\Tot\).

  We check~\eqref{eq:KK_first_1b}:
  \begin{align*}
    T_\Dual\bigl(\Theta_A \otimes_{\Dual\otimes_\Tot A} T_\Dual(\alpha)\bigr) &= T_\Dual(\Theta_A) \otimes_{\Dual\otimes_\Tot A} T_\Dual(\alpha) \\
    &= (-1)^n T_\Dual(\Theta\otimes_\Tot\ID_A) \otimes_{\Dual\otimes_\Tot A\otimes \Dual} \flip \otimes_{\Dual\otimes_\Tot A} T_\Dual(\alpha) \\
    &=(-1)^n T_\Dual(\Theta\otimes_\Tot\alpha) \otimes_{\Dual\otimes \Dual} \flip_\Dual \\
    &=(-1)^{n+in} T_\Dual(\alpha) \otimes_\Dual T_\Dual(\Theta) \otimes_{\Dual\otimes \Dual} \flip_\Dual \\
    &=(-1)^{in} T_\Dual(\alpha) \otimes_\Dual T_\Dual(\Theta) =(-1)^{in} T_\Dual(\alpha \otimes_\Dual \Theta).
  \end{align*}
  This computation uses~\eqref{eq:KK_first_non-trivial_3}, graded commutativity of exterior products, and~\eqref{eq:cocommutative}.
\end{remark}

\subsection{Modifying the support conditions}
\label{sec:modify_support_first}

The first duality isomorphism for a proper \(\Grd\)\nb-space~\(\Tot\) specialises to an isomorphism \(\KK_*^\Grd(\Dual,\UNIT) \cong \RK^*_\Grd(\Tot)\) for trivial \(A\) and~\(B\).  Now we construct a similar duality isomorphism in which \(\RK^*_\Grd(\Tot)\) is replaced by the equivariant \(\K\)\nb-theory (with \(\Grd\)\nb-compact support)
\[
\K^*_\Grd(\Tot) \defeq \K_*\bigl(\Grd\ltimes\CONT_0(\Tot)\bigr).
\]
This is based on a result of Jean-Louis Tu (\cite{Tu:Novikov}*{Proposition 6.25}).  Let~\(I_\Tot\) denote the directed set of \(\Grd\)\nb-compact subsets of~\(\Tot\) (these are closed and \(\Grd\)\nb-invariant by convention).

\begin{theorem}
  \label{the:RKK_via_crossed}
  Let~\(\Tot\) be a proper \(\Grd\)\nb-space and let~\(B\) be a \(\Grd\ltimes \Tot\)-\(\Cst\)\nb-algebra.  Then there is a natural isomorphism
  \[
  \varinjlim_{\Other\in I_\Tot} \KK^{\Grd\ltimes\Tot}_*(\CONT_0(\Other), B) \cong \K_*(\Grd\ltimes B).
  \]
\end{theorem}

\begin{proof}
  If~\(\Tot\) is \(\Grd\)\nb-compact, then the left hand side is simply \(\KK^{\Grd\ltimes\Tot}_*(\CONT_0(\Tot), B)\), and the statement follows easily from \cite{Tu:Novikov}*{Proposition 6.25}, see also \cite{Emerson-Meyer:Equivariant_K}*{Theorem 4.2}.  In general, we let
  \[
  B_\Other \defeq \{b\in B\mid \text{\(b_\tot=0\) for all \(\tot\in\Tot \setminus \Other\)}\}
  \qquad\text{for \(\Other\in I_X\).}
  \]
  This is a \(\Grd\)\nb-invariant ideal in~\(B\), and \(\K_*(\Grd\ltimes B_\Other)\) is the inductive limit of \(\K_*(\Grd\ltimes B_\Other)\) because \(\K\)\nb-theory commutes with inductive limits.  To finish the proof, we reduce the general to the \(\Grd\)\nb-compact case by constructing natural isomorphisms
  \[
  \KK^{\Grd\ltimes\Tot}_*(\CONT_0(\Other), B) \cong
  \KK^{\Grd\ltimes\Other}_*(\CONT_0(\Other), B_\Other).
  \]
  We observe that an \(\Tot\)\nb-equivariant correspondence from~\(\Other\) to~\(B\) must involve a Hilbert module over~\(B\) whose fibres vanish outside~\(\Other\).  Hence inner products of vectors in this Hilbert module must belong to the ideal \(B_\Other\) in~\(B\) and provide a factorisation of any Kasparov cycle for \(\KK^{\Grd\ltimes\Tot}_*(\CONT_0(\Other), B)\) through a unique Kasparov cycle for \(\KK^{\Grd\ltimes\Other}_*(\CONT_0(\Other), B_\Other)\).
\end{proof}

Now we modify the first duality isomorphism as follows.  Let~\(A\) be a \(\Grd\ltimes\Tot\)\nb-\(\Cst\)-algebra and let~\(B\) be a \(\Grd\)\nb-\(\Cst\)-algebra.  For \(\Other\in I_\Tot\), we let \(A|_\Other\defeq A\bigm/ \CONT_0(\Tot\setminus\Other)\cdot A\) be the restriction of~\(A\) to~\(\Other\).  Then we consider the map
\begin{multline}
  \label{eq:first_Kasparov_bundle_local}
  \PD^*\colon \varinjlim_{\Other\in I_\Tot} \KK^{\Grd\ltimes\Tot}_i(A|_\Other, p_\Tot^*B) \xrightarrow{(-1)^{in}T_\Dual} \varinjlim_{\Other\in I_\Tot} \KK^\Grd_i(\Dual\otimes_\Tot A|_\Other, \Dual\otimes B)\\
  \xrightarrow{\blank\otimes_\Dual D}
  \varinjlim_{\Other\in I_\Tot} \KK^\Grd_{i-n}(\Dual\otimes_\Tot A|_\Other, B)
\end{multline}
This is the inductive limit of the maps in~\eqref{eq:first_Kasparov_bundle_local} for \(A|_\Other\) for \(\Other\in I_\Tot\).  In examples, it often happens that this map is invertible between the inductive limits although the maps for a single \(\Other\in I_\Tot\) are not invertible.  The following variant of Theorem~\ref{the:KK_first_non-trivial} provides a necessary and sufficient condition for this:

\begin{theorem}
  \label{the:KK_first_non-trivial_local}
  Let \(\Dual\) and~\(A\) be \(\Grd\ltimes\Tot\)\nb-\(\Cst\)-algebras.  Let \(\Theta\in\RKK^\Grd_n(\Tot;\UNIT,\Dual)\) and \(D\in\KK^\Grd_{-n}(\Dual,\UNIT)\) satisfy \eqref{eq:Kasparov_dual_1} and~\eqref{eq:other_third_condition}.  The map \(\PD^*\) in~\eqref{eq:first_Kasparov_bundle_local} is invertible for all \(\Grd\)\nb-\(\Cst\)\nb-algebras~\(B\) if and only if for each \(\Other\in I_\Tot\) there are \(\Other'\in I_\Tot\) with \(\Other'\supseteq\Other\) and
  \[
  \Theta_A|_\Other\in \KK^{\Grd\ltimes\Tot}_n \bigl(A|_{\Other'}, p_\Tot^*(\Dual \otimes_\Tot A|_\Other)\bigr)
  \]
  such that the diagram
  \[
  \begin{gathered}
    \xymatrix@C+2em{ \Dual\otimes_\Tot A|_{\Other'} \ar[r]^{T_\Dual(\Theta_A|_\Other)} \ar[d]_{T_\Dual(\Theta\otimes_\Tot \ID_A|_{\Other'})} & \Dual\otimes (\Dual\otimes_\Tot A|_\Other)
      \ar[d]^{(-1)^n\flip}\\
      (\Dual\otimes_\Tot A|_{\Other'}) \otimes \Dual \ar[r]_{R_{\Other'}^\Other}&
      (\Dual\otimes_\Tot A|_\Other) \otimes \Dual
    }
  \end{gathered}
  \]
  in \(\KKcat_\Grd\) commutes and, for all \(\alpha\in \KK^{\Grd\ltimes\Tot}_i \bigl(A|_\Other, p_\Tot^*(B)\bigr)\),
  \[
  \Theta_A|_\Other \otimes_{\Dual\otimes_\Tot A|_\Other} T_\Dual(\alpha)
  = (R_{\Other'}^\Other)^*(\Theta \otimes_\Tot\alpha)
  \qquad\text{in \(\KK^{\Grd\ltimes\Tot}_{i+n} \bigl(A|_{\Other'}, p_\Tot^*(\Dual\otimes B)\bigr)\).}
  \]
  Here \(R_{\Other'}^\Other\colon A|_{\Other'}\to A|_\Other\) denotes the restriction map.
\end{theorem}

The proof of Theorem~\ref{the:KK_first_non-trivial_local} is almost literally the same as for Theorem~\ref{the:KK_first_non-trivial}.

In particular, Theorem~\ref{the:KK_first_non-trivial_local} identifies
\[
\K_i^\Grd(\Tot) \cong \varinjlim_{\Other\in I_\Tot} \KK^\Grd_{i-n}(\Dual|_\Other,\UNIT).
\]
The right hand side is a variant of the \(\Grd\)\nb-equivariant \(\K\)\nb-homology of~\(\Dual\) with some built-in finiteness properties.

What does the factorisation of~\(\Theta\) in Theorem~\ref{the:KK_first_non-trivial_local} mean if \(A=\CONT_0(\Tot)\)?  The following discussion explains why we call this a locality condition and why it does not come for free in an axiomatic approach but is cheap to get in concrete examples.  Since this, eventually, does not help to prove theorems, we are rather brief.  Roughly speaking, we introduce a support of~\(\Theta\) in \(\Tot\times_\Base\Tot\).  Merely being a cycle for \(\RKK^\Grd(\Tot;\UNIT,\Dual)\) forces the restriction of the \emph{first} coordinate projection to the support of~\(\Theta\) to be proper.  The factorisations require the \emph{second} coordinate projection to be proper as well.  In practice, the support of~\(\Theta\) is a small closed neighbourhood of the diagonal in \(\Tot\times_\Base\Tot\), so that both coordinate projections are indeed proper on the support.  But this is a feature of concrete constructions, which cannot be taken for granted in a general axiomatic theory.

A cycle~\(\Theta\) for \(\RKK^\Grd_n(\Tot;\UNIT;\Dual)\) consists of a (possibly \(\Z/2\)\nb-graded) \(\Grd\)\nb-equivariant Hilbert module~\(\mathcal{E}\) over \(\CONT_0(\Tot)\otimes \Dual\) and \(\Grd\)\nb-equivariant self-adjoint operator~\(F\) such that \(\varphi\cdot (1-F^2)\) is compact for all \(\varphi\in \CONT_0(\Tot)\); here \(\CONT_0(\Tot)\) acts by right multiplication because of \(\CONT_0(\Tot)\)-linearity.

Since~\(\Dual\) is a \(\Cst\)\nb-algebra over~\(\Tot\) as well, we may multiply on the left by functions in \(\CONT_0(\Tot\times_\Base\Tot)\) and view~\(F\) as a \(\Grd\)\nb-equivariant family of operators~\(F_{(\tot_1,\tot_2)}\) for \((\tot_1,\tot_2)\in\Tot\times_\Base\Tot\), with~\(F_{\tot_1,\tot_2}\) acting on a Hilbert module~\(\mathcal{E}_{(\tot_1,\tot_2)}\) over~\(\Dual_{\tot_2}\).  Compactness of \(\varphi'\cdot (1-F^2)\) for all \(\varphi'\in \CONT_0(\Tot\times_\Base\Tot)\) means that \(1-F^2_{(\tot_1,\tot_2)}\) is compact for all \((\tot_1,\tot_2)\) and depends norm-continuously on \((\tot_1,\tot_2)\) in a suitable sense.  Assuming this, the operators \(\varphi\cdot (1-F^2)\) are compact for all \(\varphi\in \CONT_0(\Tot)\) if and only if \(\norm{1-F^2_{(\tot_1,\tot_2)}}\to0\) for \(\tot_2\to\infty\), uniformly for~\(\tot_1\) in compact subsets in~\(\Tot\).  Letting~\(S_\varepsilon\) be the closed set of all \((\tot_1,\tot_2)\in\Tot\times_\Base\Tot\) with \(\norm{F_{(\tot_1,\tot_2)}^2-1}\ge\varepsilon\), this means that the first coordinate projection \(\pi_1\colon S_\varepsilon\to\Tot\) is proper for all \(\varepsilon>0\).  Actually, it suffices to assume this for \(\varepsilon=1\) because~\(F^2\) is positive where \(\norm{F^2-1}<1\), so that we may always homotope~\(F\) to another operator with \(F^2=1\) outside a given neighbourhood of~\(S_1\).

When does~\(\Theta|_\Other\) factor through \(\KK_n^{\Grd\ltimes\Tot}(\CONT_0(\Other'),\CONT_0(\Tot)\otimes\Dual|_\Other)\)?  Cycles for the latter group may be viewed as cycles for \(\KK_n^{\Grd\ltimes\Tot}(\CONT_0(\Tot),\CONT_0(\Tot)\otimes\Dual)\) where \(\mathcal{E}_{(\tot_1,\tot_2)}=0\) for \(\tot_2\notin \Other\) or \(\tot_1\notin\Other'\).  The restriction of~\(\Theta\) to~\(\Theta|_\Other\) only ensures \(\mathcal{E}_{(\tot_1,\tot_2)}=0\) for \(\tot_2\notin \Other\).  If only \(\norm{F^2_{(\tot_1,\tot_2)}-1}<1\) for \(\tot_2\notin \Other\) or \(\tot_1\notin\Other'\), then standard homotopies provide another representative first with \(F^2_{(\tot_1,\tot_2)}=1\) for \(\tot_2\notin \Other\) or \(\tot_1\notin\Other'\) and then one with \(\mathcal{E}_{(\tot_1,\tot_2)}=0\) for \(\tot_2\notin \Other\) or \(\tot_1\notin\Other'\).  Hence the factorisations required in Theorem~\ref{the:KK_first_non-trivial_local} exist if and only if for each \(\Grd\)\nb-compact subset \(\Other\subseteq\Tot\) there is a \(\Grd\)\nb-compact subset \(\Other'\subseteq\Tot\) such that \(S_1\subseteq \Other'\times\Other\).  This condition seems unrelated to the properness of the \emph{first} coordinate projection on~\(S_1\).  It follows if the \emph{second} coordinate projection is proper.

In practice, if~\(\Theta\) does not come from a dual Dirac element, its construction usually ensures \(\mathcal{E}_{(\tot_1,\tot_2)}=0\) outside a small neighbourhood of the diagonal, so that both coordinate projections are proper on~\(S_1\) (see also Section~\ref{sec:tangent_dual}).  Hence, the necessary factorisations exist by construction.  This feature of the construction is also used to get~\(\Theta_A\) for strongly locally trivial~\(A\), compare Section~\ref{sec:tangent_locally_trivial}.

\section{Bundles of compact spaces}
\label{sec:first_duality_compact}

Throughout this section, we consider the simpler case of a \emph{proper space over~\(\Base\)}, that is, the map \(p_\Tot\colon \Tot\to\Base= \Grd^{(0)}\) is proper.  We may then view~\(\Tot\) as a bundle of compact spaces over~\(\Base\) (but these bundles need not be locally trivial).  If~\(\Tot\) is proper over~\(\Base\), then there is an equivariant \(^*\)\nb-homomorphism
\[
p_\Tot^!\colon \CONT_0(\Base)\to\CONT_0(\Tot), \qquad \varphi\mapsto \varphi\circ p_\Tot.
\]
If~\(\Base\) is a point, then~\(\Tot\) is compact and~\(p_\Tot^!\) is the unit map \(\C\to\CONT(\Tot)\).

\begin{proposition}
  \label{pro:diagonal_restriction_compact}
  Let~\(\Tot\) be a proper \(\Grd\)\nb-space over~\(\Base\).  Let~\(A\) be a \(\Grd\)\nb-\(\Cst\)\nb-algebra and let~\(B\) be a \(\Grd\ltimes\Tot\)-\(\Cst\)\nb-algebra.  Then the map
  \begin{equation}
    \label{eq:p_Tot_shriek}
    \KK^{\Grd\ltimes\Tot}_*(p_\Tot^*(A),B) \to \KK^\Grd_*(A,B),
    \qquad
    \alpha \mapsto [p_\Tot^!]\otimes_{\CONT_0(\Tot)}
    \forget_\Tot(\alpha)
  \end{equation}
  is a natural isomorphism.  Let \([m_B]\in \KK^{\Grd\ltimes\Tot}_0(p_\Tot^*(B), B)\) be the class of the multiplication homomorphism \(p_\Tot^*(B) = \CONT_0(\Tot) \otimes_\Base B \to B\).  The inverse of the isomorphism in~\eqref{eq:p_Tot_shriek} is the map
  \[
  \KK^\Grd_*(A,B) \to \KK^{\Grd\ltimes\Tot}_*(p_\Tot^*(A),B), \qquad \alpha\mapsto p_\Tot^*(\alpha)\otimes_{p_\Tot^*(B)} [m_B].
  \]
\end{proposition}

\begin{proof}
  The action of \(\CONT_0(\Tot)\otimes A\) in a cycle for \(\KK^{\Grd\ltimes\Tot}_*(\CONT_0(\Tot)\otimes A,B)\) is already determined by its restriction to~\(A\) and \(\CONT_0(\Tot)\)-linearity.  We may describe the restriction to~\(A\) as the composition with \(p_\Tot^!\otimes\ID_A\colon A\cong \CONT_0(\Base)\otimes_\Base A\to \CONT_0(\Tot)\otimes_\Base A\).  Since~\(p_\Tot\) is proper, the compactness conditions for a Kasparov cycle are the same for \(\KK^{\Grd\ltimes\Tot}_*(p_\Tot^*(A),B)\) and \(\KK^\Grd_*(A,B)\).  Thus~\(p_\Tot^!\) induces an isomorphism as claimed.  The formula for the inverse follows because
  \[
  [p_\Tot^!] \otimes_{\CONT_0(\Tot)} p_\Tot^*(\alpha) \otimes_{p_\Tot^*(B)} [m_B] = \alpha \otimes [p_\Tot^!] \otimes_{p_\Tot^*(B)} [m_B] = \alpha
  \]
  for all \(\alpha\in\KK^\Grd_*(A,B)\) --~this formalises the naturality of~\([p_\Tot^!]\).
\end{proof}

In the non-equivariant case and in Kasparov's notation, Proposition~\ref{pro:diagonal_restriction_compact} asserts
\[
\cRKK_*(\Tot;\CONT_0(\Tot)\otimes_\Base A, B) \cong \cRKK_*(\Base;A, B)
\]
where~\(A\) is a \(\Base\)\nb-\(\Cst\)\nb-algebra and~\(B\) is an \(\Tot\)\nb-\(\Cst\)\nb-algebra, provided~\(\Tot\) is proper over~\(\Base\).  For \(\Base=\pt\) and compact~\(\Tot\), we get
\[
\cRKK_*(\Tot;\CONT(\Tot, A), B) \cong \KK_*(A, B).
\]

When we specialise Proposition~\ref{pro:diagonal_restriction_compact} to the case where both algebras are pulled back from~\(\Base\), we get
\begin{multline*}
  \RKK^\Grd_*(\Tot;A,B) \defeq \KK^{\Grd\ltimes\Tot}_*\bigl(p_\Tot^*(A),p_\Tot^*(B)\bigr) \\
  \cong \KK^\Grd_*\bigl(A,p_\Tot^*(B)\bigr) \cong \KK^\Grd_*(A, \CONT_0(\Tot) \otimes B).
\end{multline*}

\begin{example}
  \label{exa:RKK_space_homotopy}
  The coordinate projection \(\Tot\times[0,1]\to\Tot\) is proper.  Proposition~\ref{pro:diagonal_restriction_compact} applied to this map and \(A\) and~\(B\) pulled back from~\(\Base\) yields
  \[
  \RKK^\Grd_*(\Tot\times[0,1];A,B) \cong \RKK^\Grd_*(\Tot;A,\CONT([0,1])\otimes B) \cong \RKK^\Grd_*(\Tot;A, B)
  \]
  because Kasparov theory is homotopy invariant in the second variable.  This shows that \(\Tot\mapsto \RKK^\Grd_*(\Tot;A,B)\) is a homotopy functor.
\end{example}

Plugging Proposition~\ref{pro:diagonal_restriction_compact} into the definition of an abstract dual, we get:

\begin{corollary}
  \label{cor:abstract_dual_compact}
  Assume that~\(\Tot\) is proper over~\(\Base\), let~\(\Dual\) be a \(\Grd\)\nb-\(\Cst\)\nb-algebra and let \(\Theta\in\RKK^\Grd_n(\Tot;\UNIT,\Dual)\).  Let \(\dualfunclass \defeq [p_\Tot^!]\otimes_{\CONT_0(\Tot)} \forget_\Tot(\Theta) \in \KK^\Grd_n(\UNIT, \CONT_0(\Tot)\otimes\Dual)\).  The pair \((\Dual,\Theta)\) is an \(n\)\nb-dimensional \(\Grd\)\nb-equivariant abstract dual for~\(\Tot\) if and only if the map
  \[
  \widetilde{\PD}\colon \KK^\Grd_*(\Dual\otimes A,B) \to \KK^\Grd_{*+n}(A,\CONT_0(\Tot)\otimes B), \qquad f\mapsto \dualfunclass \otimes_\Dual f
  \]
  is an isomorphism for all \(\Grd\)\nb-\(\Cst\)\nb-algebras \(A\) and~\(B\).
\end{corollary}

Such an isomorphism means that \(\CONT_0(\Tot)\) and~\(\Dual\) are Poincar\'e dual objects of \(\KK^\Grd\) (see \cites{Skandalis:KK_survey, Emerson:Duality_hyperbolic, Brodzki-Mathai-Rosenberg-Szabo:D-Branes}).  Recall how such duals arise:

\begin{theorem}
  \label{the:nc_duality}
  Let \(A\) and~\(\hat{A}\) be objects of\/ \(\KK^\Grd\), let \(n\in\Z\), and let
  \[
  \funclass\in \KK^\Grd_{-n}(\hat{A}\otimes A,\UNIT), \qquad \dualfunclass\in \KK^\Grd_n(\UNIT, A\otimes \hat{A}).
  \]
  The maps
  \begin{alignat*}{2}
    \widetilde{\PD}&\colon \KK^\Grd_{i-n}(\hat{A}\otimes C,D) \to \KK^\Grd_i(C, A\otimes D),&\qquad
    f&\mapsto \dualfunclass\otimes_{\hat{A}} f,\\
    \widetilde{\PD}{}^*&\colon \KK^\Grd_i(C, A\otimes D) \to \KK^\Grd_{i-n}(\hat{A}\otimes C,D),&\qquad f&\mapsto (-1)^{in} f\otimes_A \funclass,
  \end{alignat*}
  are inverse to each other if and only if \(\funclass\) and~\(\dualfunclass\) satisfy the \emph{zigzag equations}
  \begin{alignat*}{2}
    \dualfunclass\otimes_{\hat{A}} \funclass &= \ID_A
    &\qquad\text{in \(\KK^\Grd_0(A,A)\),}\\
    \dualfunclass\otimes_A \funclass &= (-1)^n\,\ID_{\hat{A}} &\qquad\text{in \(\KK^\Grd_0(\hat{A},\hat{A})\).}
  \end{alignat*}
\end{theorem}

\begin{definition}
  \label{def:dual_in_KK}
  If this is the case, then we call \(A\) and~\(\hat{A}\) \emph{Poincar\'e dual}, and we call \(\funclass\) and \(\dualfunclass\) the \emph{fundamental class} and the \emph{dual fundamental class} of the duality.
\end{definition}

The zigzag equations are equivalent to
\[
\widetilde{\PD}(\funclass) = \ID_A, \qquad \widetilde{\PD}{}^*(\dualfunclass) = \ID_{\hat{A}}.
\]
Therefore,
\[
\funclass = \widetilde{\PD}{}^*(\ID_A), \qquad \dualfunclass = \widetilde{\PD}(\ID_{\hat{A}})
\]
if we have a Poincar\'e duality.  In the situation of Corollary~\ref{cor:abstract_dual_compact}, we can compute the fundamental class in terms of the constructions in Section~\ref{sec:basic_duality_constructions}:
\[
\funclass = \widetilde{\PD}{}^*(\ID_{\CONT_0(\Tot)}) = \PD^*(\Delta_\Tot) = T_\Dual(\Delta_\Tot) \otimes_\Dual D = [m]\otimes_\Dual D.
\]
Here we use that the isomorphism~\eqref{eq:p_Tot_shriek} maps the diagonal restriction class~\(\Delta_\Tot\) to~\(\ID_{\CONT_0(\Tot)}\), Equation~\eqref{eq:inverse_PD_abstract}, and the definition of the multiplication class \([m]\) in~\eqref{eq:def_multiplication_class}.

\begin{theorem}
  \label{the:Kasparov_dual_compact}
  Let~\(\Dual\) be a \(\Grd\)\nb-\(\Cst\)\nb-algebra, let \(n\in\Z\), let \(\Theta\in\RKK^\Grd_n(\Tot;\UNIT,\Dual)\) and \(D\in\RKK^\Grd_{-n}(\Tot;\UNIT,\Dual)\).  Then \((\Dual,D,\Theta)\) is an \(n\)\nb-dimensional Kasparov dual for~\(\Tot\) if and only if \(\CONT_0(\Tot)\) and~\(\Dual\) are Poincar\'e dual objects of \(\KK^\Grd\) with fundamental class
  \[
  \funclass \defeq m_\Dual\otimes_\Dual D \qquad \text{in \(\KK^\Grd_{-n}(\Dual\otimes\CONT_0(\Tot),\UNIT)\)}
  \]
  and dual fundamental class
  \[
  \dualfunclass \defeq p_\Tot^!\otimes_{\CONT_0(\Tot)} \Theta \qquad \text{in \(\KK^\Grd_n(\UNIT, \CONT_0(\Tot)\otimes\Dual)\),}
  \]
  where \(m_\Dual\in\KK^\Grd_0(\Dual\otimes\CONT_0(\Tot),\Dual)\) is the class of the multiplication homomorphism and \(p_\Tot^!\colon \CONT_0(\Base)\to\CONT_0(\Tot)\) is induced by~\(p_\Tot\).

  Furthermore, if we identify \(\RKK^\Grd_*(\Tot;A,B) \cong \KK^\Grd_*(A,\CONT_0(\Tot)\otimes B)\) as in Proposition~\textup{\ref{pro:diagonal_restriction_compact}}, then the duality isomorphisms \(\PD\) and \(\PD^*\) in Theorem~\textup{\ref{the:first_duality}} agree with the duality isomorphisms \(\widetilde{\PD}\) and \(\widetilde{\PD}{}^*\) in Theorem~\textup{\ref{the:nc_duality}}.
\end{theorem}

\begin{proof}
  We claim that
  \begin{alignat*}{2}
    \widetilde{\PD}(f) &= p_\Tot^!\otimes_{\CONT_0(\Tot)} \Theta \otimes_\Dual f &\qquad&\text{in \(\KK^\Grd_{i+n}(A,\CONT_0(\Tot)\otimes B)\),}
    \\
    \widetilde{\PD}{}^*(p_\Tot^!\otimes_{\CONT_0(\Tot)} g) &= (-1)^{in} T_\Dual(g)\otimes_\Dual D &\qquad&\text{in \(\KK^\Grd_{i-n}(\Dual \otimes A, B)\)}
  \end{alignat*}
  for all \(f\in\KK^\Grd_i(\Dual\otimes A,B)\), \(g\in\RKK^\Grd_i(\Tot;A,B)\).  The formula for \(\widetilde{\PD}(f)\) follows immediately from the definitions.  To prove the formula for \(\widetilde{\PD}{}^*\), let \(\hat{g} \defeq p_\Tot^!\otimes_{\CONT_0(\Tot)} g\) in \(\KK^\Grd_i(A,\CONT_0(\Tot)\otimes B)\).  Proposition~\ref{pro:diagonal_restriction_compact} implies
  \[
  g = p_\Tot^*(\hat{g}) \otimes_{p_\Tot^*(\CONT_0(\Tot)\otimes B)} m_{\CONT_0(\Tot)\otimes B} = p_\Tot^*(\hat{g}) \otimes_{\CONT_0(\Tot)\otimes\CONT_0(\Tot)} \Delta_\Tot.
  \]
  Using \(T_\Dual\bigl(p_\Tot^*(\hat{g})\bigr) = \ID_\Dual\otimes\hat{g}\) and \(T_\Dual(\Delta_\Tot) = m_\Dual\), we compute
  \[
  T_\Dual(g)\otimes_\Dual D = T_\Dual(p_\Tot^*\hat{g})\otimes_{\Dual\otimes\CONT_0(\Tot)} T_\Dual(\Delta_\Tot) \otimes_\Dual D = \hat{g} \otimes_{\CONT_0(\Tot)} m_\Dual \otimes_\Dual D = \hat{g} \otimes_{\CONT_0(\Tot)} \funclass.
  \]
  This yields the formula for \(\widetilde{\PD}{}^*\) and establishes the claim.

  As a consequence of the claim, the duality maps in Theorem~\ref{the:first_duality} agree with those in Theorem~\ref{the:nc_duality} up to the isomorphism in Proposition~\ref{pro:diagonal_restriction_compact}.  We know that \(\CONT_0(\Tot)\) and~\(\Dual\) are Poincar\'e dual with respect to \(\funclass\) and~\(\dualfunclass\) if and only if the maps \(\widetilde{\PD}\) and \(\widetilde{\PD}{}^*\) are inverse to each other for all \(\Grd\)\nb-\(\Cst\)-algebras \(A\) and~\(B\).  By the claim, this is the case if and only if the maps \(\PD\) and \(\PD^*\) in Theorem~\ref{the:first_duality} are inverse to each other for all \(\Grd\)\nb-\(\Cst\)-algebras \(A\) and~\(B\), which is equivalent to \((\Dual,D,\Theta)\) being a Kasparov dual for~\(\Tot\).
\end{proof}

In the situation of Theorem~\ref{the:Kasparov_dual_compact}, it is easy to reformulate the zigzag equations in terms of \(\Theta\) and~\(D\).  Defining \(\comul = T_\Dual(\Theta) \in \KK^\Grd_n(\Dual,\Dual\otimes\Dual)\) as usual, the second zigzag equation is equivalent to \(\comul \otimes_{\Dual,1} D = (-1)^n\), where \(\otimes_{\Dual,1}\) means that~\(D\) acts on the first copy of~\(\Dual\) in the target of~\(\comul\); this condition also appears in~\eqref{eq:counit}.  The first zigzag equation is equivalent to \(\PD(\funclass) = \Delta_\Tot\) because the isomorphism in~\eqref{eq:p_Tot_shriek} maps~\(\Delta_\Tot\) to \(\ID_{\CONT_0(\Tot)}\).  As a consequence:

\begin{corollary}
  \label{cor:Kasparov_dual_compact}
  In the situation of Theorem~\textup{\ref{the:Kasparov_dual_compact}}, \((\Dual,D,\Theta)\) is a Kasparov dual for~\(\Tot\) if and only if \(\comul \otimes_{\Dual,1} D = (-1)^n\) in \(\KK^\Grd_0(\Dual,\UNIT)\) and \(\Theta \otimes_\Dual [m_\Dual] \otimes_\Dual D = \Delta_\Tot\) in \(\RKK^\Grd_0(\Tot; \CONT_0(\Tot),\UNIT)\).
\end{corollary}

We can simplify Definition~\ref{def:Kasparov_dual} here because any element of \(\RKK^\Grd_*(\Tot;A,B)\) is of the form \(p_\Tot^*(f)\otimes_{\CONT_0(\Tot)} \Delta_\Tot\) for some \(f\in\KK^\Grd_*(A,\CONT_0(\Tot)\otimes B)\).  In~\eqref{eq:Kasparov_dual_2}, we can easily get rid of the factor \(p_\Tot^*(f)\) because
\[
\Theta\otimes_\Tot p_\Tot^*(f) = \Theta\otimes_\Dual T_\Dual(f) = \Theta \otimes f = (-1)^{in} f\otimes \Theta.
\]
Hence~\eqref{eq:Kasparov_dual_2} is equivalent to \(\Theta \otimes_\Tot \Delta_\Tot = \Theta \otimes_\Dual m_\Dual\) in \(\RKK^\Grd_*(\Tot;\CONT_0(\Tot),\Dual)\) because \(m_\Dual = T_\Dual(\Delta_\Tot)\).  But this simplification depends on~\(p_\Tot\) being proper.

\begin{proposition}
  \label{pro:duals_UCT}
  Let~\(A\) be a separable \(\Cst\)\nb-algebra in the UCT class.  Then~\(A\) has a Poincar\'e dual in \(\KKcat\) if and only if \(\K_*(A)\) is finitely generated.
\end{proposition}

\begin{proof}
  Any \(\Cst\)\nb-algebra in the UCT class is \(\KK\)-equivalent to \(\CONT_0(\Tot)\) for a locally finite, two-dimensional, countable simplicial complex~\(\Tot\) (see~\cite{Blackadar:K-theory}) because any countable \(\Z/2\)\nb-graded Abelian group arises as \(\K^*(\Tot)\) for some such~\(\Tot\).

  Assume first that \(\K_*(A)\) is finitely generated.  Then this simplicial complex may be taken finite, so that~\(A\) is \(\KK\)\nb-equivalent to \(\CONT(\Tot)\) for a finite simplicial complex~\(\Tot\).  Such spaces admit a Kasparov dual (see~\cite{Emerson-Meyer:Euler}), so that \(\CONT(\Tot)\) admits a Poincar\'e dual in \(\KKcat\).  So does~\(A\) because \(A\cong\CONT(\Tot)\) in \(\KKcat\).

  Assume conversely that~\(A\) has a Poincar\'e dual~\(\hat{A}\).  Then the functor \(B\mapsto \KK_*(A,B) \cong \K_*(B\otimes\hat{A})\) commutes with inductive limits of \(\Cst\)\nb-algebras.  Choose~\(\Tot\) with \(\K^*(\Tot)=\K_*(A)\) as above and write~\(\Tot\) as an increasing union of finite subcomplexes~\(\Tot_n\).  Let~\(\Tot_n^\circ\) be the interior of~\(\Tot_n\) in~\(\Tot\), then \(\CONT_0(\Tot) = \varinjlim \CONT_0(\Tot_n^\circ)\).  Since \(\KK_*(A,\blank)\) commutes with inductive limits, the \(\KK\)\nb-equivalence \(A\to \CONT_0(\Tot)\) factors through \(\CONT_0(\Tot_n^\circ)\) for some \(n\in\N\).  Thus \(\K_*(A)\) is a quotient of \(\K^*(\Tot_n^\circ)\).  Since the latter is finitely generated, so is \(\K_*(A)\).
\end{proof}

\begin{corollary}
  \label{cor:duals_compact}
  A compact space~\(\Tot\) has an abstract dual if and only if \(\K^*(\Tot)\) is finitely generated.
\end{corollary}

\begin{proof}
  A compact space~\(\Tot\) has an abstract dual if and only if \(\CONT(\Tot)\) has a Poincar\'e dual in \(\KKcat\).  Since \(\CONT(\Tot)\) automatically belongs to the UCT class, the assertion follows from Proposition~\ref{pro:duals_UCT}.
\end{proof}

\section{The second duality}
\label{sec:second_duality}

The notion of duality in \(\KK^\Grd\) is reflexive, that is, if~\(\hat{A}\) is dual to~\(A\), then~\(A\) is dual to~\(\hat{A}\).  This is because the tensor category \(\KK^\Grd\) is symmetric.  Therefore, if~\(p_\Tot^*\) is proper and \((\Dual,\Theta)\) is an abstract dual for~\(\Tot\), then we get another duality isomorphism of the form
\begin{equation}
  \label{eq:second_duality_compact}
  \KK^\Grd_*(\CONT_0(\Tot)\otimes C, D) \cong
  \KK^\Grd_{*+n}(C, \Dual\otimes D).
\end{equation}
Up to changing the order of the factors, the isomorphism~\eqref{eq:second_duality_compact} is constructed from the fundamental class and dual fundamental class of the original duality.  In the special case of a compact manifold with boundary acted upon by a compact group, \eqref{eq:second_duality_compact} is Kasparov's second Poincar\'e duality \cite{Kasparov:Novikov}*{Theorem 4.10}.  We are going to extend this isomorphism to the case where~\(\Tot\) is not proper over~\(\Base\).

Let \((\Dual,D,\Theta)\) be an \(n\)\nb-dimensional Kasparov dual for~\(\Tot\); we do not require \(p_\Tot\) to be proper.  Let~\(A\) be a \(\Grd\ltimes\Tot\)-\(\Cst\)\nb-algebra and let~\(B\) be a \(\Grd\)\nb-\(\Cst\)\nb-algebra.  Then \(\Dual\otimes B = \Dual\otimes_\Tot p_\Tot^*(B)\) is a \(\Grd\ltimes\Tot\)-\(\Cst\)\nb-algebra.  The natural map
\begin{equation}
  \label{eq:second_duality}
  \SPD\colon \KK^{\Grd\ltimes\Tot}_i(A,\Dual\otimes B)
  \xrightarrow{\forget_\Tot}
  \KK^\Grd_i(A,\Dual\otimes B)
  \xrightarrow{(-1)^{in}\blank \otimes_\Dual D}
  \KK^\Grd_{i-n}(A,B)
\end{equation}
is called the \emph{second duality map} associated to the Kasparov dual.

Consider the case \(A=p_\Tot^*(A_0)\) for some \(\Grd\)\nb-\(\Cst\)\nb-algebra~\(A_0\).  Then~\eqref{eq:second_duality} becomes a map
\[
\KK^{\Grd\ltimes\Tot}_i(p_\Tot^*(A_0),\Dual\otimes B) \to \KK^\Grd_{i-n}(\CONT_0(\Tot)\otimes A_0,B).
\]
By Proposition~\ref{pro:diagonal_restriction_compact}, its domain agrees with \(\KK^\Grd_i(A_0,\Dual\otimes B)\) if~\(p_\Tot^*\) is proper.  It can be checked that the map is indeed the inverse of the isomorphism in~\eqref{eq:second_duality_compact}.  Hence it is an isomorphism if~\(\Tot\) is proper over~\(\Base\).  But in general, the assumptions for a Kasparov dual do not imply~\eqref{eq:second_duality} to be an isomorphism, even for the case of trivial \(\Grd\ltimes \Tot\)-algebras like \(A=p_\Tot^*(A_0)\).  Theorem~\ref{the:second_duality_general} below, which is similar to Theorem~\ref{the:KK_first_non-trivial}, provides a necessary and sufficient condition.

\begin{notation}
  \label{note:Tot_structure_underline}
  In the following computations, we consider some tensor products of the form \(\Dual\otimes A\) where both \(\Dual\) and~\(A\) are \(\Grd\)\nb-\(\Cst\)\nb-algebras over~\(\Tot\).  Then \(\Dual\otimes A\) is a \(\Grd\)\nb-\(\Cst\)\nb-algebra over \(\Tot\times\Tot\), so that there are two ways to view it as a \(\Cst\)\nb-algebra over~\(\Tot\).  We underline the tensor factor whose \(\Tot\)\nb-structure we use.  Thus \(\CONT_0(\Tot)\) acts on \(\underline{\Dual}\otimes A\) by pointwise multiplication on the first tensor factor; we could also denote this by \(\Dual\otimes\forget_\Tot(A)\), but the latter notation is rather cumbersome.
\end{notation}

\begin{theorem}
  \label{the:second_duality_general}
  Let \(\Dual\) and~\(A\) be \(\Grd\ltimes\Tot\)\nb-\(\Cst\)-algebras.  Let \(\Theta\in\RKK^\Grd_n(\Tot;\UNIT,\Dual)\) and \(D\in\KK^\Grd_{-n}(\Dual,\UNIT)\) satisfy \eqref{eq:Kasparov_dual_1} and~\eqref{eq:other_third_condition} \textup(both are necessary conditions for Kasparov duals\textup).  The map \(\SPD\) in~\eqref{eq:second_duality} is an isomorphism for all \(\Grd\)\nb-\(\Cst\)\nb-algebras~\(B\) if and only if there is
  \[
  \widetilde{\Theta}_A \in \KK^{\Grd\ltimes\Tot}_n(A,\underline{\Dual}\otimes A)
  \]
  such that the diagram
  \begin{equation}
    \label{eq:second_duality_1}
    \begin{gathered}
      \xymatrix@C+1em{ A \ar[r]^{\widetilde{\Theta}_A} \ar[d]_{\Theta\otimes_\Tot A} &
        \Dual\otimes A \ar[dl]^{(-1)^n\flip} \\
        A\otimes\Dual }
    \end{gathered}
  \end{equation}
  commutes in \(\KKcat_\Grd\) --~that is, after forgetting the \(\Tot\)\nb-structure~-- and if for all \(\Grd\)\nb-\(\Cst\)-algebras~\(B\) and all \(\alpha\in\KK^{\Grd\ltimes\Tot}_i(A, \Dual\otimes B)\),
  \begin{equation}
    \label{eq:second_duality_2}
    \widetilde{\Theta}_A \otimes_A \forget_\Tot(\alpha)
    = \Theta\otimes_\Tot\alpha
    \qquad\text{in \(\KK^{\Grd\ltimes\Tot}_{i+n}(A, \underline{\Dual}\otimes\Dual\otimes B)\).}
  \end{equation}
  There is at most one element~\(\widetilde{\Theta}_A\) with these properties, and if it exists then the inverse isomorphism to~\eqref{eq:second_duality} is the map
  \[
  \SPD^* \colon \KK^\Grd_i(A,B) \to \KK^{\Grd\ltimes\Tot}_{i+n}(A,\Dual\otimes B),\qquad \alpha\mapsto \widetilde{\Theta}_A \otimes_A \alpha.
  \]
\end{theorem}

\begin{proof}
  Assume first that~\(\widetilde{\Theta}_A\) satisfies \eqref{eq:second_duality_1} and~\eqref{eq:second_duality_2}.  Define \(\SPD^*\) as above and let \(\beta \in \KK^\Grd_i(A,B)\).  We compute
  \begin{multline*}
    \SPD\circ\SPD^*(\beta)
    = (-1)^{(i+n)n} \forget_\Tot (\widetilde{\Theta}_A \otimes_A \beta) \otimes_\Dual D
    \\= (-1)^{n} \forget_\Tot (\widetilde{\Theta}_A) \otimes_\Dual D \otimes_A \beta
    = \forget_\Tot (\Theta \otimes_\Tot \ID_A) \otimes_\Dual D \otimes_A \beta
    = \beta,
  \end{multline*}
  using the graded commutativity of exterior products, \eqref{eq:second_duality_1}, and~\eqref{eq:Kasparov_dual_1}.  If \(\alpha\in\KK^{\Grd\ltimes\Tot}_i(A,\Dual\otimes B)\), then
  \begin{multline*}
    \SPD^*\circ\SPD(\alpha)
    = (-1)^{in} \widetilde{\Theta}_A \otimes_A \forget_\Tot(\alpha) \otimes_\Dual D
    = (-1)^{in}  \Theta \otimes_\Tot\alpha \otimes_\Dual D
    \\= \alpha \otimes_\Tot \Theta \otimes_\Dual D
    = \alpha
  \end{multline*}
  because of~\eqref{eq:second_duality_2}, graded commutativity of exterior products, and~\eqref{eq:Kasparov_dual_1}.  Thus \(\SPD\) and \(\SPD^*\) are inverse to each other as desired.

  Moreover, \eqref{eq:second_duality_1} and~\eqref{eq:Kasparov_dual_1} imply
  \begin{equation}
    \label{eq:define_tilde_Theta_A}
    \SPD(\widetilde{\Theta}_A)
    \defeq (-1)^n \forget_\Tot(\widetilde{\Theta}_A) \otimes_\Dual D
    = \forget_\Tot(\Theta \otimes_\Tot \ID_A) \otimes_\Dual D
    = \ID_A
  \end{equation}
  By~\eqref{eq:define_tilde_Theta_A}, if~\(\SPD\) is an isomorphism then~\(\widetilde{\Theta}_A\) is the unique \(\SPD\)-pre-image of the identity map on \(\forget_\Tot(A)\).  Naturality implies that \(\SPD^{-1}=\SPD^*\).  It remains to check that~\(\widetilde{\Theta}_A\) defined by~\eqref{eq:define_tilde_Theta_A} satisfies \eqref{eq:second_duality_1} and~\eqref{eq:second_duality_2}.

  To verify~\eqref{eq:second_duality_1}, we compute the \(\SPD\)-image of
  \[
  \widetilde{\Theta}_A \otimes_\Tot \Theta = (-1)^n \Theta \otimes_\Tot \widetilde{\Theta}_A \in \KK^\Grd_{2n}(A,\underline{\Dual} \otimes \Dual\otimes A)
  \]
  in two ways.  On the one hand,
  \begin{multline*}
    \SPD(\widetilde{\Theta}_A \otimes_\Tot \Theta) \defeq \forget_\Tot(\widetilde{\Theta}_A \otimes_\Tot \Theta) \otimes_\Dual D \\
    = \forget_\Tot(\widetilde{\Theta}_A) \otimes_\Dual T_\Dual(\Theta) \otimes_\Dual D = (-1)^n \forget_\Tot(\widetilde{\Theta}_A)
  \end{multline*}
  because \(\forget_\Tot\) is functorial and \(\widetilde{\Theta}_A\otimes_\Tot\Theta\) is the composition of \(\widetilde{\Theta}_A\colon A\to \underline{\Dual}\otimes A\) and \(\sigma_{\underline{\Dual}\otimes A}(\Theta)\colon \underline{\Dual}\otimes A \to \underline{\Dual} \otimes A\otimes\Dual\); we also use~\eqref{eq:other_third_condition}.  On the other hand,
  \begin{multline*}
    \SPD(\Theta \otimes_\Tot \widetilde{\Theta}_A) \defeq \forget_\Tot(\Theta\otimes_\Tot \widetilde{\Theta}_A) \otimes_\Dual D \\
    = \forget_\Tot(\Theta\otimes_\Tot \ID_A)\otimes_A \forget_\Tot(\widetilde{\Theta}_A) \otimes_\Dual D = (-1)^n\forget_\Tot(\Theta\otimes_\Tot \ID_A)
  \end{multline*}
  by~\eqref{eq:define_tilde_Theta_A}.  Since \(\SPD\) is injective, these two formulas imply~\eqref{eq:second_duality_1}.

  To check~\eqref{eq:second_duality_2}, we let \(\alpha\in\KK^{\Grd\ltimes\Tot}_i(A, \Dual\otimes B)\) and abbreviate \(\bar\alpha\defeq \forget_\Tot(\alpha)\).  We have \(\alpha\otimes_\Tot\Theta = (-1)^{in} \Theta\otimes_\Tot\alpha\) because the exterior product in Kasparov theory is graded commutative.  Equation~\eqref{eq:second_duality_2} is equivalent to \((-1)^{in} \SPD(\alpha\otimes_\Tot \Theta) = \bar\alpha\) because
  \[
  \SPD^{-1}(\bar\alpha) = \SPD^*(\bar\alpha) = \widetilde{\Theta}_A\otimes_A \bar\alpha,
  \]
  We can use \(\sigma_\Dual(\Theta)\in \KK^{\Grd\ltimes\Tot}_n(\Dual,\underline{\Dual}\otimes\Dual)\) to rewrite
  \[
  \alpha\otimes_\Tot \Theta = \alpha\otimes_\Dual \sigma_\Dual(\Theta).
  \]
  Since \(\forget_\Tot(\sigma_\Dual(\Theta))=T_\Dual(\Theta)\), the functoriality of \(\forget_\Tot\) and~\eqref{eq:other_third_condition} yield
  \begin{multline*}
    (-1)^{in}\SPD(\alpha\otimes_\Tot\Theta)
    = (-1)^{in}(-1)^{(i+n)n} \forget_\Tot(\alpha\otimes_\Dual \sigma_\Dual\Theta) \otimes_\Dual D
    \\= (-1)^n \forget_\Tot(\alpha) \otimes_\Dual T_\Dual(\Theta) \otimes_\Dual D
    = \bar{\alpha}.
  \end{multline*}
  Hence~\(\widetilde{\Theta}_A\) satisfies~\eqref{eq:second_duality_2}.
\end{proof}

\begin{remark}
  \label{rem:second_duality_general}
  The two conditions in Theorem~\ref{the:second_duality_general} are related: we claim that the first condition already implies
  \begin{equation}
    \label{eq:KK_second_1b}
    \forget_\Tot\bigl(\widetilde{\Theta}_A \otimes_A
    \forget_\Tot(\alpha)\bigr)
    = (-1)^{in} \forget_\Tot(\alpha \otimes_\Tot \Theta)
  \end{equation}
  in \(\KK^\Grd_{i+n}(A, \underline{\Dual}\otimes\Dual\otimes B)\) for all \(\Grd\)\nb-\(\Cst\)\nb-algebras~\(B\) and all \(\alpha\in\KK^{\Grd\ltimes\Tot}_i(A, \Dual\otimes B)\).  If the second Poincar\'e duality map in~\eqref{eq:second_duality} is an isomorphism, then the map \(\forget_\Tot\) in~\eqref{eq:KK_second_1b} is injective, so that~\eqref{eq:KK_second_1b} already implies \(\widetilde{\Theta}_A \otimes_A \forget_\Tot(\alpha) = (-1)^{in} \alpha \otimes_\Tot \Theta\).  Hence the second condition in Theorem~\ref{the:second_duality_general} is equivalent to the injectivity of \(\forget_\Tot\) on suitable groups.  All this is parallel to Remark~\ref{rem:KK_first_non-trivial} about the first Poincar\'e duality.

  We check~\eqref{eq:KK_second_1b}, abbreviating \(\forget_\Tot(\alpha)=\overline\alpha\).
  \begin{align*}
    \forget_\Tot\bigl(\widetilde{\Theta}_A \otimes_A \forget_\Tot(\alpha)\bigr) &= \overline{\widetilde{\Theta}_A} \otimes_A \overline{\alpha} = (-1)^n \overline{\Theta\otimes_\Tot \ID_A} \otimes_{A\otimes\Dual} \flip \otimes_A \overline{\alpha} \\
    &= (-1)^n \overline{\Theta\otimes_\Tot \ID_A} \otimes_A \overline{\alpha} \otimes_{\Dual\otimes\Dual} \flip_\Dual \\
    &= (-1)^n \overline{\Theta\otimes_\Tot\alpha} \otimes_{\Dual\otimes\Dual} \flip_\Dual \\
    &=(-1)^{n+in} \overline{\alpha\otimes_\Tot\Theta} \otimes_{\Dual\otimes\Dual} \flip_\Dual \\
    &=(-1)^{n+in} \overline{\alpha} \otimes_\Dual T_\Dual(\Theta) \otimes_{\Dual\otimes\Dual} \flip_\Dual \\
    &=(-1)^{in} \overline{\alpha} \otimes_\Dual T_\Dual(\Theta) = (-1)^{in}\overline{\alpha\otimes_\Tot\Theta}.
  \end{align*}
\end{remark}

If \(A=p_\Tot^*(A_0)\), then~\(\widetilde{\Theta}_A\) lies in the Kasparov group
\[
\KK^{\Grd\ltimes\Tot}_*(A,\underline{\Dual}\otimes A) = \KK^{\Grd\ltimes\Tot}_*(\CONT_0(\Tot)\otimes A_0, \underline{\Dual}\otimes\CONT_0(\Tot)\otimes A_0).
\]
The obvious Ansatz for~\(\widetilde{\Theta}_A\) is \(\widetilde{\Theta}_{p_\Tot^*(A_0)} \defeq \widetilde{\Theta}\otimes \ID_{A_0}\) for some
\[
\widetilde{\Theta}\in \KK^{\Grd\ltimes\Tot}_n\bigl(\CONT_0(\Tot), \underline{\Dual}\otimes\CONT_0(\Tot)\bigr).
\]
The latter group differs from \(\KK^{\Grd\ltimes\Tot}_n\bigl(\CONT_0(\Tot), \Dual\otimes \underline{\CONT_0(\Tot)}\bigr)\) that contains~\(\Theta\)~-- unless~\(\Tot\) is proper over~\(\Base\), in which case both groups agree with \(\KK^\Grd_n\bigl(\UNIT,\Dual\otimes\CONT_0(\Tot)\bigr)\) and \(\Theta\) and~\(\widetilde{\Theta}\) correspond to the same element~\(\dualfunclass\) in the latter group.  Thus our two duality isomorphisms use slightly different ingredients unless~\(p_\Tot\) is proper.

\begin{definition}
  \label{def:symmetric_Kasparov_dual}
  A \emph{symmetric Kasparov dual} for a \(\Grd\)\nb-space~\(\Tot\) is a quadruple \((\Dual,D,\Theta,\widetilde{\Theta})\) where~\(\Dual\) is a \(\Grd\ltimes\Tot\)-\(\Cst\)-algebra, \(D\in\KK^\Grd_{-n}(\Dual,\UNIT)\), and
  \[
  \Theta\in\KK^{\Grd\ltimes\Tot}_n(\UNIT_\Tot, \underline{\UNIT_\Tot}\otimes \Dual)
  = \RKK^\Grd_n(\UNIT,\Dual),\qquad
  \widetilde{\Theta}\in\KK^{\Grd\ltimes\Tot}_n(\UNIT_\Tot, \underline{\Dual}\otimes\UNIT_\Tot)
  \]
  satisfy the following conditions:
  \begin{itemize}
  \item \(\Theta \otimes_\Dual D = \ID_\UNIT\) in \(\RKK^\Grd_0(\Tot;\UNIT,\UNIT)\);
  \item \(T_\Dual(\Theta)\otimes_{\Dual\otimes\Dual} \flip_\Dual = (-1)^n T_\Dual(\Theta)\) in \(\KK^\Grd_n(\Dual,\Dual\otimes\Dual)\);
  \item \(\forget_\Tot(\Theta) = (-1)^n \forget_\Tot(\widetilde{\Theta}) \otimes_{\Dual\otimes\CONT_0(\Tot)} \flip\) in \(\KK^\Grd_n(\UNIT_\Tot,\UNIT_\Tot\otimes\Dual)\);
  \item \(\Theta \otimes_\Tot f = \Theta \otimes_\Dual T_\Dual(f)\) in \(\RKK^\Grd_{*+n}(\Tot;A,\Dual\otimes B)\) for all \(\Grd\)\nb-\(\Cst\)\nb-algebras \(A\) and~\(B\) and all \(f\in\RKK^\Grd_*(\Tot;A,B)\);
  \item \(\widetilde{\Theta} \otimes_{\CONT_0(\Tot)} \forget_\Tot(f) = \Theta \otimes_\Tot f\) in \(\KK^{\Grd\ltimes\Tot}_{*+n}(\Tot;p_\Tot^*A, \underline{\Dual}\otimes\Dual\otimes B)\) for all \(\Grd\)\nb-\(\Cst\)\nb-algebras \(A\) and~\(B\) and all \(f\in\KK^{\Grd\ltimes\Tot}_*(\Tot; p_\Tot^*A,\Dual\otimes B)\);
  \end{itemize}
\end{definition}

Equivalently, \((\Dual,D,\Theta)\) is a Kasparov dual for~\(\Tot\) and \(\widetilde{\Theta}\otimes \ID_A\) satisfies the conditions of Theorem~\ref{the:second_duality_general} with \(A= p_\Tot^*A_0\) for all \(\Grd\)\nb-\(\Cst\)-algebras~\(A_0\).  Here we identify \(\Dual\otimes_\Tot p_\Tot^*A_0\cong \Dual\otimes A_0\).  Thus a symmetric Kasparov dual provides both duality isomorphisms:
\begin{align*}
  \KK^{\Grd\ltimes \Tot}_i\bigl(p_\Tot^*(A),p_\Tot^*(B)\bigr)
  &\cong \KK^\Grd_{i+n}(\Dual\otimes A,B),\\
  \KK^{\Grd\ltimes \Tot}_i(p_\Tot^*(A), \underline{\Dual}\otimes B) &\cong \KK^\Grd_{i-n}(\CONT_0(\Tot)\otimes A,B)
\end{align*}
for all \(\Grd\)\nb-\(\Cst\)\nb-algebras \(A\) and~\(B\).

The second duality isomorphism for~\(\Tot\) yields, in particular, an isomorphism
\[
\KK^\Grd_*\bigl(\CONT_0(\Tot),\CONT_0(\Other) \bigr) \cong
\KK^{\Grd\ltimes\Tot}_*\bigl(\CONT_0(\Tot),\underline{\Dual}\otimes\CONT_0(\Other) \bigr)
\]
for another \(\Grd\)\nb-space~\(\Other\).  If \(\Dual= \CONT_0(\VB)\) for a space~\(\VB\) over~\(\Tot\), then the right hand side is the \(\Grd\)\nb-equivariant \(\K\)\nb-theory of \(\VB\times\Other\) with \(\Tot\)\nb-compact support (see~\cite{Emerson-Meyer:Equivariant_K}).  Thus bivariant \(\K\)\nb-theory with commutative coefficients reduces to ordinary \(\K\)\nb-theory with support conditions under a duality assumption.

Applications to the Baum--Connes conjecture require a variant of the second duality isomorphism with different support conditions, which we get as in Section~\ref{sec:modify_support_first}.  Let~\(I_\Tot\) denote the directed set of \(\Grd\)\nb-compact subsets of~\(\Tot\), let~\(A|_\Other\) for a \(\Grd\ltimes\Tot\)\nb-\(\Cst\)-algebra~\(A\) denote the restriction to~\(\Other\), and let \(R^\Other_{\Other'}\colon A|_\Other\to A|_{\Other'}\) for \(\Other\subseteq\Other'\) denote the restriction map.

\begin{theorem}
  \label{the:second_duality_general_local}
  Let \(\Dual\) and~\(A\) be \(\Grd\ltimes\Tot\)\nb-\(\Cst\)-algebras.  Let \(\Theta\in\RKK^\Grd_n(\Tot;\UNIT,\Dual)\) and \(D\in\KK^\Grd_{-n}(\Dual,\UNIT)\) satisfy \eqref{eq:Kasparov_dual_1} and~\eqref{eq:other_third_condition}.  Then the map~\(\SPD\) induces an isomorphism
  \[
  \varinjlim_{\Other\in I_\Tot} \KK^{\Grd\ltimes\Tot}_i(A|_\Other,\Dual\otimes B)
  \to \varinjlim_{\Other\in I_\Tot} \KK^\Grd_i(A|_\Other,B)
  \]
  if and only if for each \(\Other\in I_\Tot\) there is \(\Other'\in I_\Tot\) and
  \[
  \widetilde{\Theta}_A|_\Other \in \KK^{\Grd\ltimes\Tot}_n(A|_{\Other'},\underline{\Dual}\otimes A|_\Other)
  \]
  such that the diagram
  \[
  \begin{gathered}
    \xymatrix@C+1em{ A|_{\Other'} \ar[r]^{\widetilde{\Theta}_A|_\Other} \ar[d]_{\Theta\otimes_\Tot A|_{\Other'}} &
      \Dual\otimes A|_{\Other'} \ar[d]^{(-1)^n\flip}\\
      A|_{\Other'} \otimes\Dual \ar[r]_{R_{\Other'}^\Other} &
      A|_\Other \otimes\Dual
    }
  \end{gathered}
  \]
  commutes in \(\KKcat_\Grd\) --~that is, after forgetting the \(\Tot\)\nb-structure~-- and for all \(\Grd\)\nb-\(\Cst\)-algebras~\(B\) and all \(\alpha\in\KK^{\Grd\ltimes\Tot}_i(A|_\Other, \Dual\otimes B)\),
  \[
  \widetilde{\Theta}_A|_\Other \otimes_{A|_\Other} \forget_\Tot(\alpha)
  = (R_{\Other'}^\Other)^*(\Theta\otimes_\Tot\alpha)
  \qquad\text{in \(\KK^{\Grd\ltimes\Tot}_{i+n}(A|_{\Other'}, \underline{\Dual}\otimes\Dual\otimes B)\).}
  \]
\end{theorem}

The proof of Theorem~\ref{the:second_duality_general_local} is literally the same as for Theorem~\ref{the:second_duality_general}.

\begin{theorem}
  \label{the:BC-domain}
  Under the assumptions of Theorem~\textup{\ref{the:second_duality_general_local}}, there are natural isomorphisms
  \[
  \varinjlim_{\Other\in I_\Tot} \KK^\Grd_*(\CONT_0(\Other),B) \cong \K_{*+n}\bigl(\Grd\ltimes (\Dual\otimes B)\bigr)
  \]
  for all \(\Grd\)\nb-\(\Cst\)\nb-algebras~\(B\).  If~\(\Tot\) is a proper \(\Grd\)\nb-space, then the following diagram commutes:
  \[
  \xymatrix{
    \varinjlim \KK^\Grd_*(\CONT_0(\Other),B)  \ar[r]^-{\cong} \ar[dr]_{\textup{Index}_\Grd} &
    \K_{*+n}\bigl(\Grd\ltimes (\Dual\otimes B)\bigr)\ar[d]^{D}\\
    &\K_*(\Grd\ltimes B)
  }
  \]
  Here~\(\textup{Index}_\Grd\) denotes the equivariant index map that appears in the Baum--Connes assembly map.
\end{theorem}

\begin{proof}
  The isomorphism follows by combining Theorems \ref{the:second_duality_general_local} and~\ref{the:RKK_via_crossed}.  It is routine but tedious to check that the diagram commutes.
\end{proof}

Theorem~\ref{the:BC-domain} corrects \cite{Tu:Novikov}*{Proposition 5.13}, which ignores the distinction between \(\KK^{\Grd\ltimes\Tot}_*(\CONT_0(\Tot), \underline{\CONT_0(\Tot)}\otimes\Dual)\) and \(\KK^{\Grd\ltimes\Tot}_*(\CONT_0(\Tot), \CONT_0(\Tot)\otimes\underline{\Dual})\).

In practice, we get~\(\widetilde{\Theta}_A|_\Other\) from~\(\widetilde{\Theta}_A\) by restricting to~\(\Other\) in the target variable.  This only yields \(\widetilde{\Theta}_A|_\Other \in \KK^{\Grd\ltimes\Tot}_n(A,\underline{\Dual}\otimes A|_\Other)\).  We need that~\(\widetilde{\Theta}_A\) is sufficiently local for this restriction to factor through the restriction map \(A\to A|_{\Other'}\) for some \(\Grd\)\nb-compact subset~\(\Other'\), and we need that the conditions for duality can be checked locally as well, so that they hold for these factorisations.  The classes~\(\Theta_A|_\Other\) needed for Theorem~\ref{the:KK_first_non-trivial_local} can be constructed in a similar way by factoring~\(\Theta_A\).

\begin{definition}
  \label{def:local_sym_dual}
  A symmetric dual is called \emph{local} if such factorisations exist for~\(\Theta\) and~\(\widetilde{\Theta}\) and if these satisfy the appropriate conditions for the duality isomorphisms in Theorems \ref{the:KK_first_non-trivial_local} and~\ref{the:second_duality_general_local}.
\end{definition}

Thus in the situation of a local symmetric dual, we also get natural isomorphisms
\begin{align*}
  \varinjlim_{\Other\in I_\Tot} \KK^{\Grd\ltimes\Tot}_*\bigl( \CONT_0(\Other) \otimes A, p_\Tot^*(B)\bigr)
  &\cong \varinjlim_{\Other\in I_\Tot} \KK^\Grd_*( \Dual|_\Other \otimes A, B),\\
  \varinjlim_{\Other\in I_\Tot} \KK^{\Grd\ltimes\Tot}_*( \CONT_0(\Other) \otimes A, \Dual \otimes B)
  &\cong \varinjlim_{\Other\in I_\Tot} \KK^\Grd_*( \CONT_0(\Other) \otimes A, B).
\end{align*}

Assume now that~\(\Tot\) is a \emph{universal} proper \(\Grd\)\nb-space.  Then Corollary~\ref{cor:complementary_CC_CI} interprets the functor \(A\mapsto \Dual\otimes A\) as the localisation functor for the subcategory \(\mathcal{CC}\) of all objects~\(A\) with \(p_{\EG\Grd}^*(A)=0\).  Hence \(A\mapsto \K_{*+n}\bigl(\Grd\ltimes (\Dual\otimes A)\bigr)\) is the localisation of \(\K_*(\Grd\ltimes A)\) at the subcategory~\(\mathcal{CC}\).  And Theorem~\ref{the:BC-domain} identifies the Baum--Connes assembly map with the canonical map from this localisation to \(\K_*(\Grd\ltimes A)\).  Moreover, Theorem~\ref{the:dual_EG2} shows that~\(D\) induces an isomorphism \(\K_{*+n}\bigl(\Grd\ltimes (\Dual\otimes A)\bigr) \cong \K_*(\Grd\ltimes A)\) if~\(A\) is a proper \(\Grd\)\nb-\(\Cst\)-algebra.  As a result:

\begin{theorem}
  \label{the:BC_iso_proper}
  The Baum--Connes assembly map is an isomorphism for proper coefficient algebras provided~\(\EG\Grd\) has a local symmetric Kasparov dual.
\end{theorem}

Finally, let us discuss the difference between the groups
\[
\KK^{\Grd\ltimes\Tot}_n(\CONT_0(\Tot),\underline{\CONT_0(\Tot)}\otimes \Dual)
\qquad\text{and}\qquad
\KK^{\Grd\ltimes\Tot}_n(\CONT_0(\Tot),\CONT_0(\Tot)\otimes \underline{\Dual})
\]
that contain \(\Theta\) and~\(\widetilde{\Theta}\).  We use the same notation as in the end of Section~\ref{sec:modify_support_first}.  Cycles for both groups consist of a Hilbert module~\(\mathcal{E}\) over \(\CONT_0(\Tot)\otimes\Dual\) and a norm-continuous \(\Grd\)\nb-equivariant family of Fredholm operators~\(F_{(\tot_1,\tot_2)}\).  In addition, for a cycle for \(\KK^{\Grd\ltimes\Tot}_n(\CONT_0(\Tot),\underline{\CONT_0(\Tot)}\otimes \Dual)\), the first coordinate projection must be proper on~\(S_\varepsilon\) for all \(\varepsilon>0\); for a cycle for \(\KK^{\Grd\ltimes\Tot}_n(\CONT_0(\Tot),\CONT_0(\Tot)\otimes \underline{\Dual})\), the second coordinate projection must be proper on~\(S_\varepsilon\) for all \(\varepsilon>0\).

In concrete constructions, it often happens that both coordinate projections on the support of~\(\Theta\) are proper, and then we may use the same pair \((\mathcal{E},F)\) (up to a sign~\((-1)^n\)) to define~\(\widetilde{\Theta}\).  Moreover, the factorisations needed for a local symmetric Kasparov dual exist in this case.

\section{Duals for bundles of smooth manifolds}
\label{sec:tangent_dual}

We construct local symmetric Kasparov duals for a bundle of smooth manifolds with boundary and establish first and second Poincar\'e duality isomorphisms with coefficients in bundles of \(\Cst\)\nb-algebras that are locally trivial in a sufficiently strong sense.  This generalises results in \cites{Kasparov:Novikov, Echterhoff-Emerson-Kim:Duality, Tu:Twisted_Poincare}.

There are several equivalent duals that are related by Thom isomorphisms.  When dealing with a single smooth manifold, Gennadi Kasparov~\cite{Kasparov:Novikov} formulates duality using Cliford algebra bundles.  The Clifford algebra bundle may be replaced by the tangent bundle.   If the manifold is \(\K\)\nb-oriented, then it is self-dual.  In~\cite{Emerson-Meyer:Correspondences}, another dual involving a stable normal bundle appears naturally.  The latter has the advantage of working for all cohomology theories.  Already in real \(\K\)\nb-theory, we meet a problem with the tangent space duality because we have to equip the tangent space with a non-trivial ``real'' structure for things to work out.  The three duals indicated above are clearly equivalent because of Thom isomorphisms, so that it suffices to establish duality for one of them.  But it is useful for applications to describe the classes \(D\), \(\Theta\) and~\(\widetilde{\Theta}\) explicitly in each case.

Before we come to that, we first explain the notion of a bundle of smooth manifolds that we use and the assumptions we impose.  Then we describe different duals for a bundle of smooth manifolds with boundary and state the duality results.  Finally, we prove the duality isomorphisms for strongly locally trivial coefficients.

\subsection{Bundles of smooth manifolds}
\label{sec:bundles_smooth}

We do not require the base spaces of our bundles to be manifolds and only require smoothness of the action along the fibres, but we require the action to be proper.  First we explain why we choose this particular setup.  Then we define bundles of smooth manifolds and smooth groupoid actions on them and construct nice fibrewise Riemannian metrics.  The latter will be needed to construct the duality.

The properness assumption avoids certain rather severe analytical difficulties.  The most obvious of these is the absence of \(\Grd\)\nb-invariant metrics on various vector bundles like the tangent bundle.  This has the effect that adjoints of natural \(\Grd\)\nb-invariant differential operators such as the vertical de Rham differential fail to be \(\Grd\)\nb-invariant.  If we wanted to prove the Baum--Connes Conjecture for some groupoid, we would have to overcome exactly such difficulties.  Here we \emph{use} the Baum--Connes Conjecture or, more precisely, the Dirac dual Dirac method, in order to avoid these difficulties and replace non-proper actions by proper ones.

The trick is as follows.  Let~\(\Grd\) be a groupoid and let~\(\Tot\) be a \(\Grd\)\nb-space.  Let~\(\EG\Grd\) be a universal proper \(\Grd\)\nb-space.  Then we study the \emph{proper} \(\Grd\ltimes\EG\Grd\)-space \(\Tot\times\EG\Grd\) instead of the \(\Grd\)\nb-space~\(\Tot\).  Note that the groupoid \(\Grd\ltimes\EG\Grd\) is itself proper, whence all its actions are automatically proper.  For topological computations, replacing~\(\Grd\) by~\(\Grd\ltimes\EG\Grd\) often yields the same results by Theorem~\ref{the:dual_EG2}, which asserts that
\begin{equation}
  \label{eq:p_EG}
  p_{\EG\Grd}^*\colon \KK^\Grd_*(A,B) \to \RKK^\Grd_*(\EG\Grd;A,B)  
\end{equation}
is invertible once~\(A\) is \(\KK^\Grd\)-equivalent to a proper \(\Grd\)\nb-\(\Cst\)\nb-algebra.  The map~\eqref{eq:p_EG} is the analogue of the Baum--Connes assembly map for \(\KK^\Grd_*(A,B)\) by Corollary~\ref{cor:complementary_CC_CI}.

The invertibility of~\eqref{eq:p_EG} is closely related to the dual Dirac method.  It is automatic if~\(\Grd\) acts properly on~\(A\).  The existence of a dual Dirac morphism is equivalent to~\(p_{\EG\Grd}^*\) being invertible for all proper coefficient algebras~\(B\).  The map~\(p_{\EG\Grd}^*\) is invertible for arbitrary \(A\) and~\(B\) if and only if~\(\Grd\) has a dual Dirac morphism with \(\gamma=1\).  For instance, this happens if~\(\Grd\) is an amenable groupoid, or just acts amenably on~\(A\).  If there is a dual Dirac morphism, then the map~\(p_{\EG\Grd}^*\) is split surjective, and its kernel is the kernel of the \(\gamma\)\nb-element \(\gamma\in\KK^\Grd_0(\UNIT,\UNIT)\), which is notoriously difficult to compute.  Thus we may view \(\RKK^\Grd_*(\EG\Grd;A,B)\) as the topologically accessible part of \(\KK^\Grd_*(A,B)\).

Even if~\(\Gamma\) is the fundamental group of a smooth manifold~\(M\), the universal proper \(\Gamma\)\nb-space~\(\EG\Gamma\) need not have the homotopy type of a smooth manifold.  Hence the trick above requires that we replace a smooth manifold~\(M\) by a bundle \(\EG\Gamma\times M\) of smooth manifolds over an arbitrary base space~\(\EG\Gamma\).  Fortunately, we only need smoothness along the fibres, and the properness of the action on \(\EG\Gamma\times M\) is worth giving up smoothness along the base.

Let~\(\Base\) be a locally compact space and let \(p\colon \Tot\to\Base\) be a space over~\(\Base\).  We want to define what it means for~\(\Tot\) to be a bundle of smooth manifolds over~\(\Base\).  We require an open covering of~\(\Tot\) by chart neighbourhoods that are homeomorphic to \(U\times\R^n\) with~\(U\) open in~\(\Base\), such that~\(p\) becomes the projection to the first coordinate on \(U\times\R^n\).  We also require that the change of coordinate maps on intersections of chart neighbourhoods are smooth in the \(\R^n\)\nb-direction.

\begin{example}
  \label{exa:submersion_as_bundle}
  Let \(\Tot\) and~\(\Base\) be smooth manifolds and let \(\pi\colon \Tot\to\Base\) be a submersion.  Then~\(\Tot\) is a bundle of smooth manifolds over~\(\Base\).
\end{example}

More generally, we consider bundles of smooth manifolds \emph{with boundary}.  These are defined similarly, allowing \(U\times\R^{n-1}\times [0,\infty)\) instead of \(U\times\R^n\) in the local charts.

Given two bundles \(\Tot\) and~\(\Other\) of smooth manifolds with boundary over~\(\Base\) and a continuous map \(f\colon \Tot\to\Other\) over~\(\Base\), we call~\(f\) \emph{fibrewise smooth} or \(\CONT^{0,\infty}\) if derivatives of arbitrary order in the \(\R^{n-1}\times[0,\infty)\)\nb-direction of the maps \(U_1\times\R^{n-1}\times [0,\infty)\to U_2\times\R^{n-1}\times [0,\infty)\) that we get from~\(f\) by restriction to chart neighbourhoods are continuous functions.

\begin{remark}
  \label{rem:local_triviality}
  If~\(p\) is proper, that is, the fibres of~\(p\) are compact, then any such bundle of smooth manifolds is locally trivial (via local homeomorphisms that restrict to diffeomorphisms on the fibres).  We sketch the proof.  Fix \(\base\in\Base\), let \(M\defeq p^{-1}(\base)\) be the fibre.  This is a smooth manifold by assumption.  For any \(\tot\in p^{-1}(\base)\), there is a chart neighbourhood \(U_\tot\subseteq\Tot\) of~\(\tot\) that is identified with \(p(U_\tot)\times\R^n\).  Since~\(p\) is proper, finitely many such chart neighbourhoods \((U_i)\) cover a neighbourhood of the fibre~\(M\).  Shrinking them, if necessary, we may assume that they all involve the same open subset \(V = p(U_\tot)\subseteq\base\), so that our chart neighbourhoods cover \(p^{-1}(V)\).

  The charts provide local retractions \(r_i\colon U_i\to U_i\cap M\).  In order to patch these local retractions together, we choose a fibrewise smooth partition of unity \((\tau_i)\) subordinate to our covering and embed~\(M\) into~\(\R^N\) for some \(N\in\N\).  We get a fibrewise smooth map
  \[
  h \defeq \sum_i \tau_i\cdot r_i\colon V \to \R^N,
  \]
  which maps the fibre \(M=p^{-1}(\base)\) identically to \(M\subseteq\R^N\).  Shrinking~\(V\), if necessary, we can achieve that \(h(V)\) is contained in a tubular neighbourhood~\(E\) of \(M\) in~\(\R^N\), so that we can compose~\(h\) with a smooth retraction \(E\to M\).  This yields a fibrewise smooth retraction \(r\colon p^{-1}(V)\to M\).  Since smooth maps that are close to diffeomorphisms are still diffeomorphisms, \(r\) restricts to a diffeomorphism on the fibres \(p^{-1}(\base')\) for~\(\base'\) in some neighbourhood of~\(\base\).  On this smaller neighbourhood, \(r\times p\colon p^{-1}(V)\to M\times V\) trivialises our bundle.
\end{remark}

There is a well-defined vector bundle \(\Tvert\Tot\) on~\(\Tot\) --~called \emph{vertical tangent bundle}~-- that consists of the tangent spaces in the fibre directions.  This bundle and the bundles of fibrewise differential forms derived from it are bundles of smooth manifolds over~\(\Base\), so that we may speak of \(\CONT^{0,\infty}\)-sections.  A \emph{fibrewise Riemannian metric} on~\(\Tot\) is a \(\CONT^{0,\infty}\)-section of the bundle of positive definite bilinear forms on~\(\Tvert\Tot\).

If \(p\colon \Tot\to\Base\) has a structure of smooth manifold over~\(\Base\) and \(f\colon \Base'\to\Base\) is a continuous map, then
\[
f^*(p) \defeq f\times_\Base p\colon f^*(\Tot) \defeq \Base'\times_\Base\Tot\to\Base'
\]
inherits a structure of smooth manifold over~\(\Base'\).

\begin{definition}
  \label{def:smooth_action}
  If~\(\Grd\) is a groupoid with base space~\(\Base\), then a continuous action of~\(\Grd\) on~\(\Tot\) is called (fibrewise) \emph{smooth} if the map \(s^*(\Tot)\to r^*(\Tot)\) that describes the action is fibrewise smooth.
\end{definition}

\begin{example}
  \label{exa:example_of_smooth_G_manifold}
  Let~\(\Grd\) be a Lie groupoid.  The range map \(r\colon \Grd\to \Base\) gives \(\Tot \defeq \Grd\) the structure of a smooth bundle of manifolds over~\(\Base\) with a smooth action of~\(\Grd\) by translations.
\end{example}

Let~\(\Tot\) be a bundle of smooth manifolds with boundary over~\(\Base\) and let~\(\Grd\) act on it smoothly and properly.  We are going to construct a Kasparov dual for~\(\Tot\).

First we construct a collar neighbourhood near the boundary~\(\partial\Tot\) of~\(\Tot\).  The boundary~\(\partial\Tot\) is a bundle of smooth manifolds over~\(\Base\) with a smooth action of~\(\Grd\).

\begin{lemma}
  \label{lem:collar}
  The embedding \(\partial\Tot\to\Tot\) extends to a \(\Grd\)\nb-equivariant \(\CONT^{0,\infty}\)-diffeo\-morphism from \(\partial\Tot\times[0,1)\) onto an open neighbourhood of~\(\partial\Tot\) in~\(\Tot\).
\end{lemma}

\begin{proof}
  Each \(\tot\in \partial\Tot\) has a neighbourhood in~\(\Tot\) that is diffeomorphic to \(U\times\R^{n-1}\times[0,1)\) with \(U\subseteq\Base\) open, such that~\(\tot\) corresponds to a point in \(U\times\R^{n-1}\times\{0\}\).  We transport the inward pointing normal vector field \(\partial/\partial t_n\) on \(U\times\R^{n-1}\times[0,1)\) to a locally defined vector field along the fibres of~\(\Tot\).  Patching them together via a \(\CONT^{0,\infty}\)\nb-partition of unity, we get a \(\CONT^{0,\infty}\)-vector field \(\xi\colon \Tot\to\Tvert\Tot\) such that \(\xi(\tot)\in \Tvert_\tot\Tot\) points inward for all \(\tot\in\partial\Tot\).  Averaging over the \(\Grd\)\nb-action, we can arrange for this vector field to be \(\CONT^{0,\infty}\) and \(\Grd\)\nb-equivariant as well because~\(\Grd\) acts properly and has a Haar system.

  Let \(\Psi\colon \Tot\times[0,\infty)\to\Tot\) be the flow associated to this vector field.  Then
  \[
  \partial\Tot\times[0,1)\to \Tot,
  \qquad (\tot,t)\mapsto \Psi(\tot,\varrho(\tot)\cdot t)
  \]
  for a suitable \(\Grd\)\nb-invariant \(\CONT^{0,\infty}\)-function \(\varrho\colon \Tot\to(0,\infty)\) will be a \(\Grd\)\nb-equivariant diffeomorphism onto a neighbourhood of~\(\partial\Tot\) because~\(\Psi\) is a diffeomorphism near \(\partial\Tot\times\{0\}\).
\end{proof}

Using this equivariant collar neighbourhood, we embed~\(\Tot\) in a bundle of smooth manifolds without boundary
\[
\Tot^\circ \defeq \Tot \sqcup_{\partial\Tot\times [0,1)}
\partial\Tot\times(-\infty,1) =
\Tot \cup \partial\Tot\times(-\infty,0).
\]
Of course, \(\Tot^\circ\) is diffeomorphic to the interior \(\Tot\setminus\partial\Tot\) of~\(\Tot\), but we prefer to view it as an enlargement of~\(\Tot\) by the \emph{collar} \(\partial\Tot\times(-\infty,0)\).  There is a continuous \(\Grd\)\nb-equivariant \emph{retraction}
\begin{equation}
  \label{eq:retract_collar}
  r\colon \Tot^\circ\to\Tot
\end{equation}
that maps points in \(\partial\Tot\times (-\infty,0]\) to their first coordinate.  Clearly, this is even a deformation retraction via \(r_t(\tot,s)\defeq (\tot,ts)\) for \(t\in [0,1]\), \(\tot\in\partial\Tot\), \(s\in (0,\infty)\) and \(r_t(\tot)\defeq \tot\) for \(\tot\in\Tot\).

\begin{remark}
  \label{rem:boundary_no_problem}
  Since~\(r\) is a \(\Grd\)\nb-homotopy equivalence, \(r^*\colon \RKKcat_\Grd(\Tot) \to \RKKcat_\Grd(\Tot^\circ)\) is an equivalence of categories.  Therefore, an abstract dual for~\(\Tot\) is the same as an abstract dual for~\(\Tot^\circ\).  This explains why the presence of a boundary creates no problems for the first Poincar\'e duality isomorphism.  We must, however, take the boundary into account for the second Poincar\'e duality isomorphism because the forgetful functors on \(\RKKcat_\Grd(\Tot)\) and \(\RKKcat_\Grd(\Tot^\circ)\) are not equivalent: they involve \(\CONT_0(\Tot)\) and \(\CONT_0(\Tot^\circ)\), and these are \emph{not} homotopy equivalent because~\(r\) is not proper.
\end{remark}

\begin{lemma}
  \label{lem:nice_Riemannian_metric}
  There is a \(\Grd\)\nb-invariant Riemannian metric on~\(\Tot^\circ\) that is of product type in a neighbourhood of the collar and that is complete in the following sense.  Equip each fibre of \(p\colon \Tot^\circ\to\Base\) with the distance function associated to the Riemannian metric.  For each \(R\in\R_{\ge0}\) and each compact subset \(K\subseteq\Tot^\circ\), the set of \(\tot\in\Tot^\circ\) that have distance at most~\(R\) from a point in~\(K\) is compact.
\end{lemma}

\begin{proof}
  Let \(U_1\defeq \partial \Tot\times (-\infty,1)\) and \(U_2\defeq \Tot\setminus\partial \Tot\times [0,\nicefrac{1}{2})\).  There is a fibrewise smooth, \(\Grd\)\nb-invariant partition of unity \((\varphi_1,\varphi_2)\) subordinate to this covering.

  Now choose any Riemannian metric on \(\partial\Tot\).  Since~\(\Grd\) acts properly, we can make this metric \(\Grd\)\nb-invariant by averaging with respect to the Haar system of~\(\Grd\).  We equip \(U_1=\partial\Tot\times(-\infty,1)\) with the product metric.  Similarly, we get a \(\Grd\)\nb-invariant Riemannian metric on~\(U_2\).  We patch these metrics together with the partition of unity \((\varphi_1,\varphi_2)\).  This produces a \(\Grd\)\nb-invariant Riemannian metric on~\(\Tot^\circ\) that is of product type on the collar neighbourhood \(\partial\Tot\times(-\infty,\nicefrac{1}{2})\).  But it need not yet be complete.

  To achieve a complete metric, we use a fibrewise smooth \(\Grd\)\nb-invariant function \(f_0\colon \Tot\to\R_{\ge0}\) that induces a proper map \(\Grd\backslash\Tot\to \R_{\ge0}\).  Let \(\pi_1\colon U_1\to\partial\Tot\) and \(\pi_2\colon U_1\to (-\infty,1)\) be the coordinate projections.  We define a fibrewise smooth \(\Grd\)\nb-invariant function \(f=(f_1,f_2,f_3)\colon \Tot^\circ\to\R^2_{\ge0}\times (-\infty,1)\) by
  \[
  f_1(\tot)\defeq \varphi_1(\tot) f_0\bigl(\pi_1(\tot)\bigr),\qquad
  f_2(\tot)\defeq \varphi_2(\tot) f_0(\tot),\qquad
  f_3(\tot)\defeq \varphi_1(\tot)\pi_2(\tot).
  \]
  Clearly, \(f\) induces a proper function on \(\Grd\backslash \Tot^\circ\).  Embed~\(\Tot^\circ\) into \(\Tot^\circ\times\R^2_{\ge0}\) via \((\ID,f_1,f_2)\) and replace our metric by the subspace metric from \(\Tot^\circ\times\R^2_{\ge0}\).  This is still a \(\Grd\)\nb-invariant Riemannian metric of product type on \(\partial\Tot\times(-\infty,\nicefrac{1}{2})\) because \(f_1\) and~\(f_2\) are constant there.  We claim that this new metric is complete.

  Our construction of the new metric ensures that
  \[
  d(\tot_1,\tot_2) \ge \abs{f_1(\tot_1)-f_1(\tot_2)},\qquad
  d(\tot_1,\tot_2) \ge \abs{f_2(\tot_1)-f_2(\tot_2)}
  \]
  for all \(\tot_1,\tot_2\) in the same fibre of~\(\Tot^\circ\).  We also get
  \[
  d(\tot_1,\tot_2) \ge \abs{f_3(\tot_1)-f_3(\tot_2)}
  \]
  because this already holds for the old metric~-- recall that it is of product type on the collar.  Since~\(f\) becomes proper on \(\Grd\backslash\Tot^\circ\), this estimate shows that the closed fibrewise \(R\)\nb-neighbourhood \(B_R(K)\) of a compact subset \(K\subseteq\Tot^\circ\) is \(\Grd\)\nb-compact, where \(R\in\R_{\ge0}\) is arbitrary.

  We must show that \(B_R(K)\) is compact, not just \(\Grd\)\nb-compact.  Let
  \[
  \Other\defeq \bigl\{\tot\in\Tot^\circ \bigm| d\bigl(f(\tot),f(K)\bigr)\le R\bigr\}.
  \]
  Any point of \(B_R(K)\) is connected to one in~\(K\) by a path of length~\(R\); this path must lie in~\(\Other\) by the above estimates.

  The subset~\(\Other\) is \(\Grd\)\nb-invariant and \(\Grd\)\nb-compact because~\(f\) is \(\Grd\)\nb-invariant and becomes proper on \(\Grd\backslash\Tot^\circ\).  Hence there is a compact subset \(L\subseteq\Other\) with \(\Grd\cdot L=\Other\).  We may assume \(K\subseteq L\).  Let~\(M\) be a compact neighbourhood of~\(L\) in~\(\Other\).  Thus there is \(\varepsilon>0\) with \(B_\varepsilon(K)\cap\Other\subseteq M\).  Let~\(H\) be the set of all \(\grd\in\Grd^{(1)}\) for which there exists \(\tot\in L\) with \(\grd\cdot\tot\in M\).  This subset is compact because~\(\Grd\) acts properly on~\(\Tot\).  Since \(\Grd\cdot L=\Other\supseteq M\), we get \(M= H\cdot L\).

  Choose \(N\in\N\) with \(N\varepsilon> R\).  For any \(\tot\in B_R(K)\), there is a path of length \(N\varepsilon\) in the fibre of~\(\tot\) that connects~\(\tot\) to a point in~\(K\).  This path must be contained in~\(\Other\).  Thus we get a chain of points \(\tot_0,\dotsc,\tot_N\in\Other\) that belong to the fibre of~\(\tot\) and satisfy \(\tot_0\in K\), \(\tot_N=\tot\), and \(d(\tot_j,\tot_{j+1})\le\varepsilon\).  We want to prove by induction that \(\tot_j\in H^j\cdot L\).  This is clear for \(\tot_0\in H^0\cdot K=K\subseteq L\).  Suppose \(\tot_j\in H^j\cdot L\) has been established, write \(\tot_j=\grd\cdot\tot'\) with \(\grd\in H^j\), \(\tot'\in L\).  Since all points~\(\tot_i\) lie in the same fibre and the metric is \(\Grd\)\nb-invariant, we have \(d(\grd^{-1}\tot_{j+1},\grd^{-1}\tot_j)\le\varepsilon\) as well, so that \(\grd^{-1}\tot_{j+1}\in B_\varepsilon(L)\subseteq M=H\cdot L\).  As a consequence, \(\grd^{-1}\tot_{j+1}= \grd'\cdot\tot''\) for some \(\grd'\in H\), \(\tot''\in L\).  Thus \(\tot_{j+1}=\grd\cdot\grd'\cdot\tot''\in H^j\cdot H\cdot L= H^{j+1}\cdot L\) as claimed.

  Our inductive argument yields \(B_R(K)\subseteq H^N\cdot L\).  Since the right hand side is compact, so is \(B_R(K)\).  Thus the metric on~\(\Other\) is complete as asserted.
\end{proof}

\subsection{Construction of the duality}
\label{sec:duality_smooth}

Let~\(\Tot\) be a bundle of smooth manifolds with a fibrewise smooth proper \(\Grd\)\nb-action as above.  We construct duals for~\(\Tot\) using the following additional data: a \(\Grd\)\nb-equivariant vector bundle~\(\VB\) over~\(\Tot^\circ\) and a \(\Grd\)\nb-equivariant spinor bundle~\(\Spinor\) for the vector bundle \(\VB'\defeq \VB\oplus \Tvert\Tot^\circ\).  Since~\(\Tot\) is a deformation retract of~\(\Tot^\circ\), such vector bundles are determined uniquely up to isomorphism by their restrictions to~\(\Tot\).

We have three main examples in mind.  First, if~\(\Tot^\circ\) is \(\K\)\nb-oriented, then we are given a spinor bundle~\(\Spinor\) for~\(\Tvert\Tot^\circ\) itself, so that we may take~\(\VB\) to be the trivial \(0\)\nb-dimensional vector bundle with total space~\(\Tot^\circ\).  Secondly, we may take \(\VB\defeq\Tvert\Tot^\circ\) and use the canonical complex spinor bundle \(\Spinor \defeq \Lambda_\C(\Tvert\Tot^\circ)\) associated to the complex structure on \(\Tvert\Tot^\circ\oplus\Tvert\Tot^\circ\).  Thirdly, we may let~\(\VB\) be the normal bundle of an embedding of~\(\Tot^\circ\) into the total space of a \(\K\)\nb-oriented vector bundle over~\(\Base\) as in~\cite{Emerson-Meyer:Correspondences}.

The underlying \(\Cst\)\nb-algebra of our dual is simply \(\Dual\defeq \CONT_0(\VB)\), where we also write~\(\VB\) for the total space of the vector bundle~\(\VB\).  Thus we get \(\CONT_0(\Tot^\circ)\) and \(\CONT_0(\Tvert\Tot^\circ)\) in the first two examples above.  We view~\(\VB\) as a space over~\(\Tot\) by combining the bundle projection \(\VB\to\Tot^\circ\) and the retraction \(r\colon \Tot^\circ\to\Tot\); this turns~\(\Dual\) into a \(\Cst\)\nb-algebra over~\(\Tot\).  The given action of~\(\Grd\) on~\(\VB\) turns~\(\Dual\) into a \(\Grd\)\nb-\(\Cst\)\nb-algebra.

\begin{remark}
  \label{rem:real_case}
  In the real case, we cannot use the tangent dual because the complex structure on \(\Tvert\Tot^\circ\oplus\Tvert\Tot^\circ\) only produces a complex spinor bundle, which is not enough for a Thom isomorphism in \(\KO\)\nb-theory.  In ``real'' \(\K\)\nb-theory, we may let \(\VB\defeq \Tvert\Tot^\circ\) with the ``real'' structure \((\tot,\xi)\mapsto (\tot,-\xi)\) for all \(\tot\in\Tot^\circ\), \(\xi\in\Tvert_\tot\Tot^\circ\).  The complex spinor bundle \(\Spinor \defeq \Lambda_\C(\Tvert\Tot^\circ)\) has a canonical ``real'' structure, and provides a spinor bundle in the appropriate sense.  Thus the dual for~\(\Tot\) in ``real'' \(\KK\)-theory is the ``real'' \(\Cst\)\nb-algebra \(\CONT_0(\Tvert\Tot^\circ,\C)\) with the involution
  \[
  \conj{f}(\tot,\xi) \defeq \conj{f(\tot,-\xi)} \qquad\text{for \(\tot\in\Tot^\circ\), \(\xi\in \Tvert_\tot\Tot^\circ\), \(f\colon \Tvert\Tot^\circ\to\C\).}
  \]
  With this ``real'' \(\Cst\)\nb-algebra, everything works exactly as in the complex case.  In the real case, we may let~\(\Dual\) be the real subalgebra \(\{f\in\CONT_0(\Tvert\Tot^\circ,\C)\mid \conj{f}=f\}\).
\end{remark}

The ingredients \(D\) and~\(\Theta\) of the Kasparov dual are easy to describe as \emph{wrong-way maps}.  These are constructed in~\cite{Connes-Skandalis:Longitudinal}, but only in the non-equivariant case and for maps between smooth manifolds.  The generalisation to \(\Grd\)\nb-equivariant \(\CONT^{0,\infty}\)-maps with appropriate \(\K\)\nb-orientation is straightforward.  We give a few more details about this because we need them, anyway, to verify the conditions for the duality isomorphisms.

The total space of~\(\VB\) is a bundle of smooth manifolds over~\(\Base\).  Its vertical tangent bundle is isomorphic to the pull-back of the \(\Grd\)\nb-equivariant vector bundle \(\VB' = \VB\oplus\Tvert\Tot^\circ\) on~\(\Tot^\circ\).  Hence the projection map \(p_\VB\colon \VB\to\Base\) is (\(\Grd\)\nb-equivariantly) \(\K\)\nb-oriented by~\(\Spinor\).  We let
\[
D\defeq (p_\VB)_! \in \KK^\Grd_*(\CONT_0(\VB),\CONT_0(\Base)\bigr).
\]
Here we treat~\(p_\VB\) as if it were a \(\K\)\nb-oriented submersion, that is, \((p_\VB)_!\) is the \(\KK\)\nb-class of the family of Dirac operators along the fibres of~\(p_\VB\) with coefficients in the spinor bundle~\(\Spinor\).  The completeness of the Riemannian metric established in Lemma~\ref{lem:nice_Riemannian_metric} ensures that this family of elliptic differential operators is essentially self-adjoint and thus defines a Kasparov cycle (compare~\cite{Kasparov:Novikov}*{Lemma 4.2}).  Since the family of Dirac operators is \(\Grd\)\nb-equivariant, we get a class in \(\KK^\Grd_*(\CONT_0(\VB),\UNIT)\) as needed.  Alternatively, we may use symbols as in~\cite{Connes-Skandalis:Longitudinal} to avoid unbounded operators.  Either way, neither the lack of smoothness of~\(\Base\) nor the additional \(\Grd\)\nb-equivariance pose problems for the construction of~\(D\).

If~\(\Tot^\circ\) is \(\K\)\nb-oriented and \(\VB=\Tot\), then~\(D\) is just the family of Dirac operators along the fibres of \(\Tot^\circ\to\Base\).  If \(\VB=\Tvert\Tot^\circ\) and~\(\Spinor\) is attached to the canonical complex structure on \(\VB'\cong \Tvert_\C\Tot^\circ\), then~\(D\) is the family of Dolbeault operators along the fibres of~\(\Tvert_\C\Tot^\circ\) --~recall that the Dolbeault operator on an almost complex manifold is equal to the Dirac operator for the associated \(\K\)\nb-orientation.

Let~\(\delta\) be the map
\[
\delta\colon \Tot \to \Tot\times_\Base \VB, \qquad \tot\mapsto \bigl(\tot,(\tot,0)\bigr).
\]
That is, we combine the diagonal embedding \(\Tot\to \Tot\times_\Base\Tot \subseteq \Tot\times_\Base \Tot^\circ\) and the zero section of~\(\VB\).  We are going to construct a corresponding \(\Grd\ltimes\Tot\)-equivariant Kasparov cycle \(\Theta\defeq \delta_!\) in \(\KK^{\Grd\ltimes\Tot}_*\bigl(\CONT_0(\Tot), \CONT_0(\Tot\times_\Base \VB)\bigr)\).  Here we treat~\(\delta\) like a smooth immersion, so that our main task is to describe a tubular neighbourhood for~\(\delta\).  Such constructions are also carried out in~\cite{Emerson-Meyer:Correspondences}.  Later proofs will use the following detailed description of~\(\delta\).

There is a fibrewise smooth \(\Grd\)\nb-invariant function \(\varrho\colon \Tot^\circ\to (0,1)\) such that the (fibrewise) exponential function \(\exp_\tot\colon \Tvert_\tot \Tot^\circ \to \Tot^\circ\) restricts to a diffeomorphism from the ball of radius \(\varrho(\tot)\) in~\(\Tvert_\tot\Tot^\circ\) onto a neighbourhood of~\(\tot\) inside its fibre.  Let~\(\Tvert\Tot\) be the restriction of~\(\Tvert\Tot^\circ\) to a vector bundle on \(\Tot\subseteq\Tot^\circ\).  Then the map
\[
\Tubtot\colon \Tvert\Tot \to \Tot\times_\Base\Tot^\circ,\qquad
(\tot,\xi)\mapsto \bigl(\tot,\exp_\tot(\xi')\bigr)
\]
with \(\xi'\defeq \xi\cdot \varrho(\tot)/\sqrt{\norm{\xi}^2 +1}\) is a \(\CONT^{0,\infty}\)-diffeomorphism from the total space of~\(\Tvert\Tot\) onto an open neighbourhood of~\(\Tot\) in \(\Tot\times_\Base\Tot^\circ\).

We want to construct a corresponding \(\Grd\ltimes\Tot\)-equivariant \(\CONT^{0,\infty}\)-diffeomorphism
\[
\bar\delta\colon \VB'|_\Tot \to\Tot\times_\Base\VB,
\]
where \(\VB'\defeq \Tvert\Tot^\circ\oplus\VB\) and where we view the right hand side as a space over~\(\Tot\) by the first coordinate projection.  To construct~\(\bar\delta\), we fix a \(\Grd\)\nb-equivariant fibrewise connection on the vector bundle~\(\VB\).  This provides a parallel transport for vectors in~\(\VB\), that is, a \(\Grd\)\nb-equivariant map
\[
\Tvert\Tot^\circ\times_{\Tot^\circ}\VB \to \VB,\qquad
(\tot,\xi,\eta) \mapsto \tau_{\tot,\xi}(\eta)
\]
with \(\tau_{\tot,\xi}(\eta)\in \VB_{\exp_\tot(\xi)}\) for \(\xi\in\Tvert_\tot\Tot^\circ\), \(\eta\in\VB_\tot\).  In the examples \(\VB=\Tot^\circ\) or \(\VB=\Tvert\Tot^\circ\), this parallel transport is easy to describe: it is constant for the trivial vector bundle~\(\Tot^\circ\) over~\(\Tot^\circ\), and the differential of the exponential map on~\(\Tvert\Tot^\circ\).  Finally, the formula for~\(\bar\delta\) is
\[
\bar\delta(\tot,\xi,\eta) \defeq \bigl(\tot,\exp_\tot(\xi'), \tau_{\tot,\xi'}(\eta)\bigr)
\qquad\text{for all \(\tot\in\Tot\), \(\xi\in\Tvert_\tot\Tot^\circ\), \(\eta\in\VB_\tot\)}
\]
with \(\xi'\defeq \xi\cdot \varrho(\tot)/\sqrt{\norm{\xi}^2 +1}\) as above.  It is easy to check that this is a \(\Grd\ltimes\Tot\)-equivariant \(\CONT^{0,\infty}\)-diffeomorphism onto a \(\Grd\)\nb-invariant open subset~\(U\) of \(\Tot\times_\Base \VB\).

The resulting class \(\Theta\defeq \delta_!\in \KK^{\Grd\ltimes\Tot}_*\bigl(\CONT_0(\Tot), \CONT_0(\Tot\times_\Base\VB)\bigr)\) is obtained by composing the Thom isomorphism for the \(\K\)\nb-orientation~\(\Spinor|_\Tot\) on~\(\VB'|_\Tot\) with the \(\Grd\ltimes\Tot\)-equivariant ideal inclusion \(\CONT_0(\VB'|_\Tot) \to \CONT_0(\Tot\times_\Base \VB)\) associated to the open embedding~\(\bar\delta\) (extend functions by~\(0\) outside~\(U\)).  More explicitly, we pull the spinor bundle~\(S\) back to a Hermitian vector bundle~\(S_U\) on~\(U\).  The underlying Hilbert module of~\(\Theta\) is the space of all \(\CONT_0\)\nb-sections of~\(S_U\), with the pointwise multiplication by functions in \(\CONT_0(\Tot\times_\Base\VB)\) and the pointwise inner product, the canonical action of~\(\Grd\), and the action of \(\CONT_0(\Tot)\) by pointwise multiplication via the first coordinate projection:
\[
f_1\cdot f_2\bigl(\tot_1,(\tot_2,\eta)\bigr) \defeq f_1(\tot_1)\cdot f_2\bigl(\tot_1,(\tot_2,\eta)\bigr)
\]
for \(f_1\in\CONT_0(\Tot)\), \(f_2\in\CONT_0(U,S_U)\subseteq \CONT_0(\Tot\times_\Base\Tvert\Tot^\circ,S_U)\), \(\tot_1\in\Tot\), \(\tot_2\in\Tot^\circ\), and \(\xi\in \VB_{\tot_2}\).  The essentially unitary operator for our Kasparov cycle is given by Cliford multiplication with \(\xi/\sqrt{1+\norm{\xi}^2}\) at \(\bar\delta(\tot,\xi)\in U\), where \(\tot\in\Tot\) and \(\xi\in\VB'_\tot\).

To get a symmetric Kasparov dual, we also need
\[
\widetilde{\Theta}\in \KK^{\Grd\ltimes \Tot}_*(\CONT_0(\Tot), \CONT_0(\Tot) \otimes \underline{\Dual}),
\]
where the underlined factor is the one whose \(\Tot\)\nb-structure we use.  We get~\(\widetilde{\Theta}\) from~\(\Theta\) by changing the action of \(\CONT_0(\Tot)\) to
\[
f_1\cdot f_2\bigl(\tot_1,(\tot_2,\xi)\bigr) \defeq f_1\bigl(r(\tot_2)\bigr) \cdot f_2\bigl(\tot_1,(\tot_2,\xi)\bigr),
\]
where \(r\colon \Tot^\circ\to\Tot\) is the retraction described above, and leaving everything else as before; the new representation of \(\CONT_0(\Tot)\) is \(\Tot\)\nb-linear for the \(\Tot\)\nb-structure on the second tensor factor \(\CONT_0(\Tvert\Tot^\circ)\).

\begin{theorem}
  \label{the:tangent_dual_symmetric}
  Let~\(\Tot\) be a bundle of smooth \(n\)\nb-dimensional manifolds with boundary over~\(\Base\) with a fibrewise smooth \(\Grd\)\nb-action, let~\(\VB\) be a \(\Grd\)\nb-equivariant vector bundle over~\(\Tot^\circ\) of dimension~\(k\), and let~\(\Spinor\) be a \(\Grd\)\nb-equivariant \(\K\)\nb-orientation \textup(complex spinor bundle\textup) for the vector bundle \(\Tvert\Tot^\circ\oplus\VB\) over~\(\Tot^\circ\).  Let \(\Dual\defeq \CONT_0(\VB)\) and let \(D=(p_\VB)_!\), \(\Theta=\delta_!\), and \(\widetilde{\Theta}\) as described above.  Then \((\Dual,D,\Theta,\widetilde{\Theta})\) is a \(-(k+n)\)\nb-dimensional local symmetric Kasparov dual for~\(\Tot\).  Hence there are natural isomorphisms
  \begin{align*}
    \KK^{\Grd\ltimes \Tot}_{*}\bigl(p_\Tot^*(A),p_\Tot^*(B)\bigr)
    &\cong \KK^\Grd_{*+k+n}(\CONT_0(\VB)\otimes A,B),\\
    \KK^{\Grd\ltimes \Tot}_{*+k+n}(p_\Tot^*(A), \CONT_0(\VB)\otimes B) &\cong \KK^\Grd_{*}(\CONT_0(\Tot)\otimes A,B)
  \end{align*}
  for all \(\Grd\)\nb-\(\Cst\)\nb-algebras \(A\) and~\(B\) and
  \[
  \K_{*+k+n}\bigl(\Grd\ltimes \CONT_0(\VB,B)\bigr) \cong \varinjlim_{\Other\in I_\Tot} \KK^\Grd_{*}(\CONT_0(\Other),B)
  \]
  for all \(\Grd\)\nb-\(\Cst\)\nb-algebras~\(B\).
\end{theorem}

In particular, we may take here \(\VB=\Tot^\circ\) if~\(\Tot\) is \(\K\)\nb-oriented in the sense that the vector bundle~\(\Tvert\Tot^\circ\) over~\(\Tot^\circ\) is \(\K\)\nb-oriented, and we may always take \(\VB=\Tvert\Tot^\circ\) (with an appropriate ``real'' structure in the ``real'' case, see Remark~\ref{rem:real_case}).

\begin{corollary}
  \label{cor:KK_as_RK}
  Let~\(\Grd\) be a locally compact groupoid with object space~\(\Base\), let~\(\Tot\) be a proper \(\Grd\)\nb-space, and let~\(\Other\) be any \(\Grd\)\nb-space.  Suppose that~\(\Tot\) is a bundle of smooth manifolds with~\(\Grd\) acting fibrewise smoothly.  Then there is a natural isomorphism
  \[
  \KK^\Grd_*\bigl(\CONT_0(\Tot),\CONT_0(\Other)\bigr) \cong \KK^{\Grd\ltimes\Tot}_*\bigl(\CONT_0(\Tot), \CONT_0(\VB\times_\Base \Other)\bigr) \eqdef \RK_{\Grd,\Tot}^*\bigl(\VB\times_\Base \Other),
  \]
  where the last group is the \(\Grd\)\nb-equivariant \(\K\)\nb-theory of \(\VB\times_\Base\Other\) with \(\Tot\)\nb-compact support.
\end{corollary}

\begin{proof}
  Put \(A=\UNIT\) and \(B=\CONT_0(\Other)\) in Theorem~\ref{the:tangent_dual_symmetric} and use the second Poincar\'e duality to get the first isomorphism.  The right hand side is exactly the definition of the \(\Grd\)\nb-equivariant \(\K\)\nb-theory with \(\Tot\)\nb-compact support in~\cite{Emerson-Meyer:Equivariant_K}.
\end{proof}

The equivariant \(\K\)\nb-theory groups that appear in Corollary~\ref{cor:KK_as_RK} are discussed in detail in~\cite{Emerson-Meyer:Equivariant_K}.  Corollary~\ref{cor:KK_as_RK} is used in~\cite{Emerson-Meyer:Correspondences} to describe suitable Kasparov groups by geometric cycles as in~\cite{Baum-Block:Bicycles}.  Theorem~\ref{the:tangent_dual_symmetric} will be proved in Section~\ref{sec:proof_theorem}, together with a generalisation to non-trivial bundles~\(A\) over~\(\Tot\).  Some situations involving foliation groupoids are discussed in Section~\ref{sec:exa_foliations}.

\subsection{Duality for strongly locally trivial bundles}
\label{sec:tangent_locally_trivial}

We define a class of \(G\ltimes\Tot\)-\(\Cst\)\nb-algebras for which we establish first and second Poincar\'e duality isomorphisms, extending results in~\cite{Echterhoff-Emerson-Kim:Duality} in two aspects: we allow \emph{bundles} of smooth manifolds \emph{with boundary} instead of smooth manifolds.

For the purposes of the following definition, we replace \(\delta\colon \Tot\to \Tot\times \VB\) by the diagonal embedding \(\delta'\colon \Tot\to \Tot\times_\Base\Tot\).  Let \(U'\subseteq \Tot\times_\Base\Tot\) be the image of~\(U\) under the projection \(\Tot\times_\Base \VB\to\Tot\times_\Base\Tot\).  Let \(\pi'_1\colon U'\to\Tot\) and \(\pi'_2\colon U'\to\Tot\) be the coordinate projections.  Let \(p_U^{U'}\colon U\to U'\) be the canonical projection, then \(\pi'_j\circ p_U^{U'}=\pi_j\) for \(j=1,2\).

\begin{definition}
  \label{def:strongly_locally_trivial}
  A \(\Grd\ltimes\Tot\)-\(\Cst\)\nb-algebra~\(A\) is called \emph{strongly locally trivial} if \((\pi'_1)^*(A)\) and \((\pi'_2)^*(A)\) are isomorphic as \(\Grd\ltimes U'\)-\(\Cst\)\nb-algebras via some isomorphism
  \[
  \alpha'\colon (\pi'_1)^*(A) \xrightarrow{\cong} (\pi'_2)^*(A) \qquad\text{in \(\Cstarcat_{\Grd\ltimes U'}\),}
  \]
  whose restriction to the diagonal \(\Tot\subseteq U'\) is the identity map on~\(A\).
\end{definition}

What this definition provides is a \(\Grd\)\nb-equivariant local parallel transport on the bundle~\(A\).  Not surprisingly, this exists provided~\(A\) is a smooth bundle with a suitable connection (see~\cite{Echterhoff-Emerson-Kim:Duality}, strongly locally trivial bundles are called \emph{feasible} there).  On finite-dimensional vector bundles, the connection between local parallel transport and connections is discussed in~\cite{Kubarski-Teleman:Direct_connections}.

It can be shown that~\(\alpha'\) is unique up to homotopy if it exists, using that the coordinate projections on~\(U'\) are homotopy equivalences \(U'\to\Tot\).  By the way, the following constructions still work if the isomorphism~\(\alpha'\) only exists in \(\KKcat_{\Grd\ltimes U'}\).

\begin{example}
  \label{exa:trivial_slt}
  If \(A=\CONT_0(\Tot,A_0) = p_\Tot^*(A_0)\) for some \(\Grd\)\nb-\(\Cst\)\nb-algebra~\(A_0\), that is, \(A\) is trivial along the fibres of \(\Tot\to\Base\), then~\(A\) is strongly locally trivial because
  \[
  (\pi'_1)^*(A) \cong (p_\Tot\pi'_1)^*(A_0) = (p_{U'})^*(A_0) = (p_\Tot\pi'_2)^*(A_0) \cong (\pi'_2)^*(A).
  \]
  Here \(p_{U'}=p_\Tot\pi'_1=p_\Tot\pi'_2\colon U'\to\Base\) is the canonical projection.
\end{example}

Let~\(A\) be a strongly locally trivial \(\Grd\ltimes\Tot\)-\(\Cst\)\nb-algebra.  We have decorated everything in Definition~\ref{def:strongly_locally_trivial} with primes because we will mainly use the corresponding isomorphisms on~\(U\) henceforth: we can pull back the isomorphism~\(\alpha'\) over~\(U'\) to an isomorphism
\[
\alpha\defeq (p_U^{U'})^*(\alpha')\colon \pi_1^*(A) \to \pi_2^*(A) \qquad \text{in \(\Cstarcat_{\Grd\ltimes U}\) via \(p_U^{U'}\).}
\]
Conversely, since~\(p_U^{U'}\) is a retraction, the isomorphism~\(\alpha\) forces~\(\alpha'\) to exist.

\begin{notation}
  \label{note:Tot-structure}
  In the following computations, it is important to remember whether we view~\(U\) as a space over~\(\Tot\) via \(\pi_1\) or~\(\pi_2\).  We write \(U_{\pi_1}\) and~\(U_{\pi_2}\) for the corresponding spaces over~\(\Tot\).  Similarly for~\(U'\).
\end{notation}

\begin{definition}
  \label{def:small_theta}
  Let \(\vartheta\in\KK^{\Grd\ltimes\Tot}_{n+k}\bigl(\CONT_0(\Tot), \CONT_0(U_{\pi_1})\bigr)\) be the composite of the class of the Thom isomorphism in \(\KK^\Grd_{n+k}\bigl(\CONT_0(\Tot), \CONT_0(\VB')\bigr)\) with the isomorphism \(\CONT_0(\VB')\cong\CONT_0(U_{\pi_1})\) from the tubular neighbourhood; we choose the tubular neighbourhood as in the construction of~\(\Theta\), so that this isomorphism is \(\Tot\)\nb-linear if~\(U\) is viewed as a space over~\(\Tot\) via \(\pi_1\colon U\to \Tot\).  Let \(\tilde\vartheta\in \KK^{\Grd\ltimes\Tot}_{n+}\bigl(\CONT_0(\Tot), \CONT_0(U_{\pi_2})\bigr)\) be the variant where we change the action of \(\CONT_0(\Tot)\) so as to get a cycle that is \(\Tot\)\nb-linear if~\(U\) is viewed as a space over~\(\Tot\) via~\(\pi_2\), as in the construction of~\(\widetilde{\Theta}\).
\end{definition}

By construction, we get \(\Theta\) and~\(\widetilde{\Theta}\) out of \(\vartheta\) and~\(\tilde\vartheta\) by composing with the class of the embedding \(\CONT_0(U) \to \CONT_0(\Tot\times_\Base \VB)\).  To get the \(\K\)-oriented classes \(\Theta_A\) and \(\widetilde{\Theta}_A\) we also bring in the isomorphism \(\alpha\colon \pi_1^*(A)\to \pi_2^*(A)\) over~\(U\) as follows.

\begin{definition}
  \label{def:theta_A}
  Let \(\Theta_A\in \KK^{\Grd\ltimes\Tot}_{n+k}\bigl(A, p_\Tot^*(\Dual\otimes_\Tot A)\bigr)\) be the composition
  \[
  A \cong \CONT_0(\Tot) \otimes_\Tot A \xrightarrow{\vartheta\otimes_\Tot\ID_A} \CONT_0(U_{\pi_1}) \otimes_\Tot A = \pi_1^*(A) \xrightarrow[\cong]{\alpha} \pi_2^*(A) \xrightarrow{\subset} p_\Tot^*(\CONT_0(\VB)\otimes_\Tot A).
  \]
  Let \(\widetilde{\Theta}_A\in \KK^{\Grd\ltimes\Tot}_{n+k}(A, A\otimes\underline{\Dual}) \cong \KK^{\Grd\ltimes\Tot}_{n+k}(A, \underline{\Dual}\otimes A)\) be the composition
  \[
  A \cong \CONT_0(\Tot) \otimes_\Tot A \xrightarrow{\tilde\vartheta\otimes_\Tot\ID_A} \CONT_0(U_{\pi_2}) \otimes_\Tot A = \pi_2^*(A) \xrightarrow[\cong]{\alpha^{-1}} \pi_1^*(A) \xrightarrow{\subset} A \otimes \underline{\Dual}.
  \]
\end{definition}

\begin{theorem}
  \label{the:PD_tangent}
  Let \(\Tot\to\Base\) be a bundle of smooth manifolds with boundary and let~\(A\) be a strongly locally trivial \(\Grd\ltimes\Tot\)-\(\Cst\)\nb-algebra.  Then we get Poincar\'e duality isomorphisms of the first and second kind:
  \begin{align*}
    \KK^{\Grd\ltimes \Tot}_*(A,\CONT_0(\Tot)\otimes B) &\cong
    \KK^\Grd_*(\CONT_0(\VB)\otimes_\Tot A,B),\\
    \KK^{\Grd\ltimes \Tot}_*(A,\CONT_0(\VB)\otimes B) &\cong \KK^\Grd_*(A,B)
  \end{align*}
  for all \(\Grd\)\nb-\(\Cst\)\nb-algebras~\(B\).  The maps are as described in Theorems \textup{\ref{the:KK_first_non-trivial}} and \textup{\ref{the:second_duality_general}}.  If~\(I_\Tot\) denotes the directed set of \(\Grd\)\nb-compact subsets of~\(\Tot\), then we get isomorphisms
  \begin{align*}
    \varinjlim_{\Other\in I_\Tot} \KK^{\Grd\ltimes \Tot}_*(A|_{\Other}, \CONT_0(\Tot)\otimes B) &\cong
    \varinjlim_{\Other\in I_\Tot} \KK^\Grd_*(\CONT_0(\VB)\otimes_\Tot A|_{\Other}, B),\\
    \varinjlim_{\Other\in I_\Tot} \KK^{\Grd\ltimes \Tot}_*(A|_{\Other}, \CONT_0(\VB)\otimes B) &\cong
    \varinjlim_{\Other\in I_\Tot} \KK^\Grd_*(A|_{\Other},B).
  \end{align*}
\end{theorem}  

In particular, we may take here \(\VB=\Tvert\Tot^\circ\) with the canonical \(\K\)\nb-orientation from the almost complex structure on \(\Tvert\Tot^\circ\oplus\Tvert\Tot^\circ\), or we may take \(\VB=\Tot^\circ\) if~\(\Tot\) is \(\K\)\nb-oriented.  In the latter case, we get duality isomorphisms
\begin{align*}
  \KK^{\Grd\ltimes \Tot}_*(A,\CONT_0(\Tot)\otimes B) &\cong
  \KK^\Grd_*(\CONT_0(\Tot^\circ)\otimes_\Tot A,B),\\
  \KK^{\Grd\ltimes \Tot}_*(A,\CONT_0(\Tot^\circ)\otimes B) &\cong \KK^\Grd_*(A,B).
\end{align*}

To make this more concrete, consider the special case where \(A=\UNIT_\Tot=\CONT_0(\Tot)\) and \(B=\UNIT=\CONT_0(\Base)\) and use the definitions in~\cite{Emerson-Meyer:Equivariant_K}.  The first duality isomorphisms identify the equivariant representable \(\K\)\nb-theory of~\(\Tot\),
\[
\RK_\Grd^*(\Tot) \defeq \KK^{\Grd\ltimes\Tot}_*\bigl(\CONT_0(\Tot),\CONT_0(\Tot)\bigr),
\]
with the equivariant locally finite \(\K\)\nb-homology of~\(\VB\),
\[
\K^{\Grd,\lf}_*(\VB) \defeq \KK^\Grd_*\bigl(\CONT_0(\VB), \CONT_0(\Base)\bigr),
\]
and the \(\K\)\nb-theory
\[
\K_\Grd^*(\Tot) \cong \varinjlim_{\Other\in I_\Tot}
\KK^{\Grd\ltimes\Tot}_* \bigl(\CONT_0(\Other),\CONT_0(\Tot)\bigr)
\]
with \(\varinjlim_{\Other\in I_\Tot} \KK^\Grd_*\bigl(\CONT_0(\VB|_\Other),\CONT_0(\Base)\bigr)\).

The second duality isomorphisms identify the equivariant representable \(\K\)\nb-theory of~\(\VB\) with \(\Tot\)\nb-compact support,
\[
\RK_{\Grd,\Tot}^*(\VB) \defeq \KK^{\Grd\ltimes\Tot}_*\bigl(\CONT_0(\Tot), \CONT_0(\VB)\bigr),
\]
with the equivariant locally finite \(\K\)\nb-homology of~\(\Tot\),
\[
\K^{\Grd,\lf}_*(\Tot) \defeq \KK^\Grd_*\bigl(\CONT_0(\Tot), \CONT_0(\Base)\bigr),
\]
and the equivariant \(\K\)\nb-theory \(\K^*_\Grd(\VB)\) with the equivariant \(\K\)\nb-homology of~\(\Tot\),
\[
\K^\Grd_*(\Tot) \defeq \varinjlim_{\Other\in I_\Tot} \KK^\Grd_*\bigl(\CONT_0(\Other), \CONT_0(\Base)\bigr).
\]
If we also drop the groupoid actions, we recover well-known classical constructions.

An important special case of \(\Cst\)\nb-algebra bundles over~\(\Tot\) are continuous trace \(\Cst\)\nb-algebras.  The question when they are strongly locally trivial is already discussed in~\cite{Echterhoff-Emerson-Kim:Duality}.  This is automatic in the non-equivariant case, but requires a mild condition about the group actions in general.  Assuming strong local triviality, we may decorate the statements above by such twists.

The strongly locally trivial \(\Cst\)\nb-algebras over~\(\Tot\) with fibre~\(\Comp(\Hils)\) form a group with respect to tensor product over~\(\Tot\).  If~\(A^*\) is the inverse of~\(A\), that is, \(A\otimes_\Tot A^*\) is Morita equivalent to \(\CONT_0(\Tot)\), then
\[
\KK^{\Grd\ltimes\Tot}_*(A\otimes_\Tot D,B) \cong \KK^{\Grd\ltimes\Tot}_*(D,A^*\otimes_\Tot B)
\]
by exterior tensor product over~\(\Tot\) with \(A^*\) and~\(A\).  In particular, we may identify
\begin{align*}
  \KK^{\Grd\ltimes\Tot}_*(A,\CONT_0(\Tot)\otimes B) &\cong \KK^{\Grd\ltimes\Tot}_*(\CONT_0(\Tot),A^*\otimes B),\\
  \KK^{\Grd\ltimes\Tot}_*(A,\CONT_0(\VB)\otimes B) &\cong \KK^{\Grd\ltimes\Tot}_*(\CONT_0(\Tot),A^*\otimes_\Tot \CONT_0(\VB)\otimes B)
\end{align*}
If we specialise to \(B=\UNIT\), then \(\KK^{\Grd\ltimes\Tot}_*(\CONT_0(\Tot),A^*\otimes B) = \KK^{\Grd\ltimes\Tot}_*(\CONT_0(\Tot),A^*)\) is the twisted representable \(\K\)\nb-theory of~\(\Tot\) with twist~\(A^*\).

When twisted equivariant \(\K\)\nb-theory is relevant, then we should replace \(\CONT_0(\VB)\) by the Cliford algebra dual as in~\cite{Kasparov:Novikov}.  Equivariant Bott periodicity shows that this Clifford algebra bundle is \(\KK^{\Grd\ltimes\Tot}\)-equivalent to \(\CONT_0(\VB)\), so that we also get such duality statements.  The Clifford algebra bundle is a strongly locally trivial continuous trace \(\Cst\)\nb-algebra over~\(\Tot\).  We refer to~\cite{Echterhoff-Emerson-Kim:Duality} for more details.

\subsection{Verifying the conditions for a duality}
\label{sec:proof_theorem}
We will now give the proof of Theorem \ref{the:PD_tangent}, showing that the quadruple \((\CONT_0(\VB), D, \Theta,\tilde\Theta)\) described above satisfies the conditions to give a duality.

First we verify~\eqref{eq:Kasparov_dual_1}, that is,
\[
\Theta \otimes_\Dual D = \ID_{\CONT_0(\Tot)}
\qquad\text{in \(\KK^{\Grd\ltimes\Tot}_0\bigl(\CONT_0(\Tot),
  \CONT_0(\Tot)\bigr)\).}
\]
In the framework of wrong-way maps, \eqref{eq:Kasparov_dual_1} amounts to the functoriality statement \(\pi_1! \circ \delta! = (\pi_1\circ\delta)!\).  Recall that
\[
\vartheta\in \KK^{\Grd\ltimes\Tot}_{n+k}\bigl(\CONT_0(\Tot), \CONT_0(U)\bigr) \cong \KK^{\Grd\ltimes\Tot}_{n+k}\bigl(\CONT_0(\Tot), \CONT_0(\VB'|_\Tot)\bigr)
\]
generates the Thom isomorphism for the vector bundle \(\VB'|_\Tot\cong U\) over~\(\Tot\) with respect to the \(\Grd\)\nb-equivariant \(\K\)\nb-orientation~\(\Spinor\).  Let~\(i\) be the embedding \(\CONT_0(U) \to \CONT_0(\Tot\times_\Base \VB)\), where we extend functions by~\(0\) outside~\(U\).  We factor \(\Theta=\vartheta\otimes_{\CONT_0(U)} [i]\).

We have \(\Theta\otimes_\Dual D = \Theta \otimes_{\Tot,\Dual} p_\Tot^*(D)\).  Recall that~\(D\) is the class in Kasparov theory associated to the family of Dirac operators on the fibres of \(\VB\to\Base\).  Hence \(p_\Tot^*(D)\) is the class in Kasparov theory associated to the family of Dirac operators on the fibres of \(\pi_1\colon \Tot\times_\Base\VB\to\Tot\).  A routine computation with symbols shows that composing \(p_\Tot^*(D)\) with~\(i\) simply restricts everything to~\(U\), so that we get the class in \(\KK^{\Grd\ltimes\Tot}_{-n-k}\bigl(\CONT_0(\VB'|_\Tot), \CONT_0(\Tot)\bigr)\) of the family of Dirac operators on the fibres of \(U\subseteq\Tot\times_\Base\VB\).

But \(U\cong\VB'|_\Tot\) is the total space of a \(\K\)\nb-oriented vector bundle over~\(\Tot\), and the family of Dirac operators and the class~\(\vartheta\) are inverse to each other, implementing the Thom isomorphism \(\Tot \sim \VB'|_\Tot\) in \(\KK^{\Grd\ltimes\Tot}\).  This goes back to Gennadi Kasparov~\cite{Kasparov:Operator_K}, and a simple proof in the groupoid setting can be found in \cite{LeGall:KK_groupoid}*{\S7.3.2}.  This finishes the proof of~\eqref{eq:Kasparov_dual_1}.

Next we check~\eqref{eq:Kasparov_dual_3}, that is, \(\nabla \otimes_{\Dual\otimes\Dual} \flip = (-1)^{n+k}\nabla\) for \(\nabla\defeq T_\Dual(\Theta)\) in \(\KK^\Grd_{n+k}(\Dual,\Dual\otimes\Dual) = \KK^\Grd_{n+k}(\CONT_0(\VB),\CONT_0(\VB\times_\Base\VB)\).  By construction, \(\nabla\) is the wrong-way element associated to the map
\[
\VB
\cong \VB\times_\Tot \Tot
\xrightarrow{\ID\times_\Tot\delta} \VB\times_\Tot\Tot\times_\Tot \VB
\cong \VB \times_\Base \VB,\qquad
V\ni (\tot,\xi)\mapsto \bigl((\tot,\xi),(r(\tot),0)\bigr),
\]
where~\(r\) is the collar retraction from~\eqref{eq:retract_collar}.  This is homotopic to the diagonal embedding of~\(\VB\) via the homotopy
\[
\VB \times [0,1] \to \VB \times_\Base \VB \times [0,1],\qquad
(\tot,\xi,t)\mapsto \bigl((\tot,\xi),\bigl(r_t(\tot),Dr_t(t\xi)\bigr),t\bigr).
\]
This whole map behaves like an immersion and has a tubular neighbourhood; hence we get a homotopy of wrong-way elements, which connects \(T_\Dual(\Theta)\) to the wrong-way element for the diagonal embedding \(\VB \to \VB \times_\Base \VB\).  Thus~\eqref{eq:Kasparov_dual_3} amounts to the statement \(\flip_!\circ \Delta_!= (-1)^{n+k}\Delta_!\), where \(\Delta\colon \VB\to \VB\times_\Base\VB\) is the diagonal embedding and \(\flip\colon \VB\times\VB\to\VB\times\VB\) maps \((\xi,\eta)\mapsto (\eta,\xi)\).  The maps \(\Delta\) and~\(\flip\) are \(\K\)\nb-oriented because~\(\VB\) is \(\K\)\nb-oriented by~\(\Spinor\).  The sign \((-1)^{k+n}\) appears because the fibres of~\(\VB\) have dimension~\(k+n\) and the flip map changes orientations and \(\K\)\nb-orientations by such a sign.  To actually complete the argument, we must replace the maps above by open embeddings on suitable vector bundles (tubular neighbourhoods) and connect these open embeddings by homotopies.  This is possible (up to the sign \((-1)^{n+k}\)) because these open embeddings are determined uniquely up to isotopy by their differential on the zero-section.  We leave further details to the reader.

The remaining condition~\eqref{eq:Kasparov_dual_2} for a Kasparov dual (Definition~\ref{def:Kasparov_dual}) is equivalent to~\eqref{eq:KK_first_non-trivial_4} for \(p_\Tot^*(A)\) for all \(\Grd\)\nb-\(\Cst\)\nb-algebras~\(A\) because \(\Theta_{p_\Tot^*(A)} = \Theta\otimes\ID_A\).  We consider the more general case of strongly locally trivial bundles right away.

To begin with, we notice that the classes \(\Theta_A\) and~\(\widetilde{\Theta}_A\) are local in the sense required by Theorems \ref{the:KK_first_non-trivial_local} and~\ref{the:second_duality_general_local}.  That is, for each \(\Grd\)\nb-compact subset~\(\Other\) there is a \(\Grd\)\nb-compact subset~\(\Other'\) such that \(\Theta_A\) and~\(\widetilde{\Theta}_A\) restrict to cycles
\[
\Theta_A|_\Other\in \KK^{\Grd\ltimes\Tot}_{n+k} \bigl(A|_{\Other'}, p_\Tot^*(\Dual \otimes_\Tot A|_\Other)\bigr)\quad\text{and}\quad
\widetilde{\Theta}_A|_\Other \in \KK^{\Grd\ltimes\Tot}_{n+k}(A|_{\Other'},\underline{\Dual}\otimes A|_\Other).
\]
Here we let~\(\Other\) be the closure of the set of all \(\tot\in\Tot\) for which there is \(\other\in\Other\) with \((\tot,\other)\in U'\) or \((\other,\tot)\in U'\).  It is clear from the definition that the restrictions of \(\Theta_A\) and~\(\widetilde{\Theta}_A\) to~\(\Other\) factor through~\(\Other'\) as needed.

Let~\(A\) be a strongly locally trivial bundle with isomorphism \(\alpha\colon \pi_1^*(A) \to \pi_2^*(A)\), and let \(f\in\KK^{\Grd\ltimes X}_i(A,p_\Tot^*B)\) for some \(\Grd\)\nb-\(\Cst\)\nb-algebra~\(B\).  Equation~\eqref{eq:KK_first_non-trivial_4} asserts \(\Theta_A\otimes_{\Dual\otimes_\Tot A} T_\Dual(f) = \Theta \otimes_\Tot f\).  By the definition of exterior products, the right hand side is the composition
\[
A \xrightarrow{\Theta\otimes_\Tot \ID_A} \underline{A}\otimes\Dual \xrightarrow{f\otimes\ID_\Dual} \underline{p_\Tot^*(B)}\otimes\Dual.
\]
Both \(\Theta_A\) and \(\Theta\otimes_\Tot\ID_A\) factor through
\[
\vartheta\otimes_\Tot\ID_A\colon A \cong \CONT_0(\Tot)\otimes_\Tot A \to \CONT_0(U_{\pi_1})\otimes_\Tot A = \pi_1^*(A).
\]
Thus it suffices to compare the compositions
\begin{multline*}
  \pi_1^*(A) \xrightarrow[\cong]{\alpha} \pi_2^*(A) \xrightarrow{\subset} \underline{\CONT_0(\Tot)} \otimes (\CONT_0(\VB) \otimes_\Tot A) \\
  \xrightarrow{p_\Tot^* T_\Dual(f)} \underline{\CONT_0(\Tot)} \otimes (\CONT_0(\VB) \otimes_\Tot p_\Tot^*B) \cong p_\Tot^*(\CONT_0(\VB) \otimes B)
\end{multline*}
and
\[
\pi_1^*(A) \xrightarrow{\subset} \underline{A}\otimes\Dual \xrightarrow{f\otimes\ID_\Dual} p_\Tot^*(B)\otimes\CONT_0(\VB).
\]
To see the difference, we view~\(f\) as a family \((f_\tot)_{\tot\in\Tot}\) of Kasparov cycles for \(A_\tot\) and \(p_\Tot^*(B)_\tot = B\).  Then the above compositions are given by families of Kasparov cycles parametrised by~\(U\).  The first composite that describes \(\Theta_A \otimes_{\Dual\otimes_\Tot A} T_\Dual(f)\) yields
\[
(\pi_1^*A)_u \cong A_{\pi_1(u)} \xrightarrow[\cong]{\alpha} A_{\pi_2(u)} \xrightarrow{f_{\pi_2(u)}} B.
\]
The second composite that describes \(\Theta \otimes_\Tot f\) yields
\[
(\pi_1^*A)_u \cong A_{\pi_1(u)} \xrightarrow{f_{\pi_1(u)}} B.
\]
Recall that~\(U\) is the total space of a vector bundle over~\(\Tot\), with bundle map~\(\pi_1\).  Hence the first map is homotopic to the second one via
\[
(\pi_1^*A)_u = A_{\pi_1(u)} \xrightarrow[\cong]{\alpha} A_{\pi_2(t\cdot u)} \xrightarrow{f_{\pi_2(t\cdot u)}} B
\]
for \(t\in [0,1]\) because \(\pi_2(0\cdot u)= \pi_1(u)\) for all \(u\in U\).  This pointwise formula describes a homotopy between the corresponding cycles for \(\KK^{\Grd\ltimes (\Tot\times_\Base\Tot)}_*\bigl(\pi_1^*(A), p_\Tot^*(\Dual\otimes B)\bigr)\).  This establishes~\eqref{eq:KK_first_non-trivial_4} and hence~\eqref{eq:Kasparov_dual_2}, so that \((\Dual,D,\Theta)\) is a Kasparov dual for~\(\Tot\).  Furthermore, the homotopy above is sufficiently local to apply to yield the corresponding condition in Theorem~\ref{the:KK_first_non-trivial_local}.

To establish the first Poincar\'e duality isomorphism for a strongly locally trivial bundle~\(A\), it remains to verify~\eqref{eq:KK_first_non-trivial_3}, that is,
\[
T_\Dual(\Theta_A) = \flip \circ T_\Dual(\Theta\otimes_\Tot \ID_A),
\]
where \(\flip\) exchanges the tensor factors \(\Dual\) and \(\Dual\otimes_\Tot A\) in the target object.  This is closely related to~\eqref{eq:Kasparov_dual_3}.  Roughly speaking, the homotopy that is used to prove~\eqref{eq:Kasparov_dual_3} can be performed over a sufficiently small neighbourhood of the diagonal in \(\Tot\times_\Base\Tot\), so that the coefficients~\(A\) create no further problems because of the isomorphism \(\pi_1^*(A)\cong \pi_2^*(A)\) over such a small neighbourhood.  We leave the details to the reader; once again, the necessary homotopies are local enough to yield the corresponding condition in Theorem~\ref{the:KK_first_non-trivial_local} as well.  Thus Theorems \ref{the:KK_first_non-trivial} and~\ref{the:KK_first_non-trivial_local} apply if~\(A\) is strongly locally trivial.

Now we turn to the second Poincar\'e duality, verifying the two conditions in Theorem~\ref{the:second_duality_general}.  First, \eqref{eq:second_duality_1} requires \(\widetilde{\Theta}_A\) and \(\Theta\otimes_\Tot\ID_A\) to agree up to the flip automorphism \(A\otimes\Dual \cong \Dual\otimes A\)~-- after forgetting the \(\Tot\)\nb-structure.

By definition, \(\widetilde{\Theta}_A\) involves
\[
A \xrightarrow{\tilde\vartheta\otimes_\Tot \ID_A} \CONT_0(U_{\pi_2}) \otimes_\Tot A = \pi_2^*(A) \xrightarrow{\alpha^{-1}} \pi_1^*(A).
\]
The difference between \(\forget_\Tot(\tilde\vartheta)\) and \(\forget_\Tot(\vartheta)\) is only the way \(\CONT_0(\Tot)\) acts on the cycle, and the two actions are clearly homotopic.  This homotopy is not \(\Tot\)\nb-linear, of course, but we can control what happens: we get a Kasparov cycle for
\[
\KK^{\Grd\ltimes\Tot}_*\bigl(A, \CONT_0(U\times[0,1])\bigr),
\]
where we view \(U\times[0,1]\) as a space over~\(\Tot\) via \((u,t)\mapsto \pi_2(t\cdot u)\), using the vector bundle structure on \(\pi_1\colon U\to \Tot\).  The isomorphism~\(\alpha\) pulls back to an isomorphism
\[
\CONT_0(U\times[0,1])\otimes_\Tot A \cong \pi_1^*(A)\otimes \CONT([0,1]) \qquad\text{in \(\Cstarcat_{\Grd\ltimes U\times[0,1]}\).}
\]
Thus \(\forget_\Tot\bigl(\alpha^{-1} \circ (\tilde\vartheta \otimes_\Tot \ID_A)\bigr) = \forget_\Tot(\vartheta\otimes_\Tot\ID_A)\), and this implies~\eqref{eq:second_duality_1}.  Once again, the homotopy is local and also yields the corresponding condition in Theorem~\ref{the:second_duality_general_local}.

It remains to check \eqref{eq:second_duality_2} and the analogous condition in Theorem~\ref{the:second_duality_general_local}; this will also establish that \((\Dual,D,\Theta,\widetilde{\Theta})\) is a local symmetric Kasparov dual.  Instead of giving a detailed computation, we use a less explicit but more coneptual argument using Remark~\ref{rem:second_duality_general}.

Let~\(B\) be a \(\Grd\)\nb-\(\Cst\)\nb-algebra and let \(f\in \KK^{\Grd\ltimes\Tot}_*(A,\Dual\otimes B)\).  We must check
\begin{equation}
  \label{eq:proof_Kasparov_2b}
  \widetilde{\Theta}_A \otimes_A f = \Theta \otimes_\Tot f\qquad
  \text{in \(\KK^\Grd_*(A,\underline{\Dual}\otimes\Dual\otimes B)\).}
\end{equation}
Remark~\ref{rem:second_duality_general} yields
\begin{equation}
  \label{eq:proof_Kasparov_2c}
  \forget_\Tot\bigl(\widetilde{\Theta}_A \otimes_A f\bigr)
  = \forget_\Tot(\Theta \otimes_\Tot f).
\end{equation}
The proof of~\eqref{eq:proof_Kasparov_2c} in Remark~\ref{rem:second_duality_general} constructs a homotopy between both sides using the first condition in Theorem~\ref{the:second_duality_general}, commutativity of exterior products, and the cocommutativity of~\(\nabla\).  Since we have already checked these conditions, we get~\eqref{eq:proof_Kasparov_2c} for free.  But actually, our arguments show a bit more: the homotopy that we get by following through the argument is supported in a small neighbourhood of the diagonal (we make this more precise below).  A straightforward extension of \cite{Connes-Skandalis:Longitudinal}*{Lemma 2.2} now shows that the space of cycles that are supported \emph{sufficiently close} to the diagonal deformation-retracts to the space of cycles supported \emph{on} the diagonal.  Since our homotopy is supported near the diagonal, we conclude that we can modify it so that its support lies on the diagonal, so that we get a sufficiently local and \(\Grd\ltimes\Tot\)-equivariant homotopy.

The support is defined as before Lemma 2.2 in~\cite{Connes-Skandalis:Longitudinal}, but such that the support of a Kasparov cycle for \(\KK^\Grd_*(A,\Dual\otimes\Dual\otimes B)\) is a subset of \(\Tot\times_\Base\Tot\).  Namely, the Hilbert module that appears in such a cycle is a bimodule over \(\CONT_0(\Tot)\), and its support is contained in \(\Tot\times_\Base\Tot\subseteq \Tot\times\Tot\) because the representation of~\(A\) is \(\Base\)\nb-equivariant.

\begin{lemma}
  \label{lem:almost_equivariance_enough}
  Let~\(A\) be strongly locally trivial, and let~\(U'\) be as in Definition~\textup{\ref{def:strongly_locally_trivial}}.  Any Kasparov cycle for \(\KK^\Grd_*(A,\Dual\otimes\Dual\otimes B)\) supported in~\(U'\) is homotopic to a \(\Grd\ltimes\Tot\)-equivariant cycle in a canonical way, so that the space of cycles with support~\(U'\) deformation retracts onto the space of \(\Grd\ltimes\Tot\)-equivariant cycles.
\end{lemma}

\begin{proof}
  Let \((\varphi,F,\mathcal{E})\) be a cycle supported in~\(U'\).  We leave \(F\) and~\(\mathcal{E}\) fixed and only modify the representation~\(\varphi\).  The representation~\(\varphi\) of~\(A\) and the action of \(\CONT_0(\Tot)\) by right multiplication define a \(\Grd\ltimes\Tot\)-equivariant representation of \(p_\Tot^*(A) = \CONT_0(\Tot)\otimes A\) on~\(\mathcal{E}\).  By definition of the support, this factors through
  \[
  \CONT_0(\Tot)\otimes A|_{\supp\mathcal{E}} = p_{\supp\mathcal{E}}^*(A).
  \]
  Since \(\supp \mathcal{E}\subseteq U'\), the coordinate projections \(\supp\mathcal{E}\to \Tot\) are proper and the isomorphism \((\pi_1')^*(A)\cong (\pi_2')^*(A)\) provides a \(\Grd\ltimes\Tot\)\nb-equivariant \(^*\)\nb-homomorphism
  \[
  A\to p_{\supp\mathcal{E}}^*(A).
  \]
  Composition with this \(^*\)\nb-homomorphism retracts the space of cycles supported in~\(U'\) to the space of \(\Grd\ltimes\Tot\)-equivariant cycles.  Since the tubular neighbourhood~\(U\) of \(\delta(\Tot)\) in \(\Tot\times_\Base\VB\) deformation retracts to \(\delta(\Tot)\), the projection \(U'\to\Tot\) is a deformation retraction as well.  Hence the corresponding map on Kasparov cycles supported on~\(U'\) is a deformation retraction.
\end{proof}

Lemma~\ref{lem:almost_equivariance_enough} shows that~\eqref{eq:proof_Kasparov_2c} can be lifted to~\eqref{eq:proof_Kasparov_2b}.  This verifies all the conditions in Theorems \ref{the:second_duality_general} and~\ref{the:second_duality_general_local}, so that we get the second Poincar\'e duality isomorphisms.  This finishes the proof of Theorems \ref{the:tangent_dual_symmetric} and~\ref{the:PD_tangent}.

\subsection{An example: foliation groupoids}
\label{sec:exa_foliations}

We briefly review the construction of the holonomy groupoid~\(\Grd\) of a foliated manifold~\((\Base, \Fol)\) and introduce some actions of~\(\Grd\).  These are automatically bundles of smooth manifolds, so that our general theory applies.  We formulate the duality theorems in this case and sketch how the Euler characteristics defined here are related to the index of the leafwise de Rham operator and the \(L^2\)\nb-Euler characteristic defined in~\cite{Connes:Survey_foliations}, referring to \cites{Emerson-Meyer:Euler, Emerson-Meyer:Equi_Lefschetz} for proofs.  The geometric framework of~\cite{Emerson-Meyer:Correspondences} is much more suitable for actual computations of both Euler characteristics and Lefschetz invariants, so that we do not give any further examples here.

Let \(\base\in\Base\).  The universal cover~\(\tilde{L}\) of the leaf~\(L\) through~\(\base\) is the quotient of the set of paths \(\gamma\colon [0,1]\to L\) with \(\gamma(0) = \base\) by the relation of homotopy with fixed endpoints.  Any such path extends to a holonomy map \(\Sigma_{\gamma(0)}\to\Sigma_{\gamma(1)}\) using the local triviality of the foliation, where \(\Sigma_{\gamma(0)}\) and \(\Sigma_{\gamma(1)}\) are local transversals through \(\gamma(0)\) and~\(\gamma(1)\).  Write \(\gamma\sim\gamma'\) if the paths \(\gamma\) and~\(\gamma'\) generate the same holonomy maps (in particular, they have the same endpoints).  The holonomy covering~\(\hat{L}\) of~\(L\) is the set of holonomy classes of paths~\(\gamma\) with \(\gamma(0)=\base\).  Since homotopic loops generate the same holonomy map, this is a quotient of the universal covering~\(\tilde{L}\).

Let~\(\Grd\) be the set of holonomy classes of paths in leaves with arbitrary endpoints, and let \(r,s\colon \Grd\rightrightarrows\Base\) be the maps that send a path to its endpoints.  Thus the fibre of~\(s\) at~\(\base\) is the holonomy cover~\(\hat{L}\) described above.  Concatenation of paths with matching endpoints defines a multiplication on~\(\Grd\) that turns this into a groupoid.  Finally, there are rather obvious local charts on~\(\Grd\) that turn it into a smooth (non-Hausdorff) manifold.  This manifold is Hausdorff in many situations, for example if the foliation is described by analytic foliation charts.  \emph{We assume from now on that~\(\Grd\) is Hausdorff.}

The following recipe yields some free and proper \(\Grd\)\nb-spaces.  Let \(\pi\colon \Base'\to\Base\) be a space over~\(\Base\).  Let
\[
\Tot \defeq \Grd\times_{s,\pi} \Base' = \{(\gamma,\base')\in \Grd\times\Base' \mid s(\gamma) = \pi(\base')\}
\]
and view this as a space over~\(\Base\) via \(p\colon \Tot\to\Base\), \(p(\gamma,\base')\defeq r(\gamma)\).  Then~\(\Grd\) acts on~\(\Tot\) via multiplication on the left.  It is easy to see that~\(\Grd\) acts freely and properly on~\(\Tot\) because the action of~\(\Grd\) on itself by left multiplication is free and proper.  The second coordinate projection identifies the orbit space \(\Grd\backslash\Tot\) with~\(\Base'\).  For instance, if~\(\pi\) is the identity map, then we get the action of~\(\Grd\) on its morphism space by left multiplication.

If \(\Base'\subseteq\Base\) is a transversal to~\(\Fol\) in the sense that it is a submanifold with \(T_\base\Base' \oplus \Fol_\base = T_\base \Base\) for all \(\base\in\Base'\), then the fibres of \(r\colon \Tot \to \Base = \Base\) are
\[
\Tot_\base = \{\gamma\in\Grd \mid \text{\(s(\gamma) \in\Base'\) and \(r(\gamma) = \base\)}\},
\]
and these subsets of~\(\Grd\) are countable and topologically discrete and thus (zero-dimensional) manifolds.

Returning to the general case of a smooth manifold \(\pi\colon \Base' \to \Base\), note that the space~\(\Tot\) is a smooth manifold because~\(s\) is a submersion.  The space~\(\Tot\) constructed above is a bundle of smooth manifolds over~\(\Base\) if \(p\colon \Tot\to\Base\) is a submersion.  This is equivalent to \(\pi\colon \Base'\to\Base\) being transverse to~\(\Fol\) in the sense that \(\pi(T_{\base'}\Base') + \Fol_{\pi(\base')} = T_\base\Base\) for all \(\base'\in\Base'\).  We assume this transversality from now on.  Then \(\Fol'\defeq \pi^{-1}(\Fol)\) is a foliation on~\(\Base'\) whose leaves are the connected components of the \(\pi\)\nb-pre-images of the leaves of~\(\Fol\).  Transversality of~\(\pi\) implies that a leaf-path in~\(\Fol'\) has non-trivial holonomy if and only if its image in~\(\Fol\) has non-trivial holonomy.

The fibres of \(\Tot\to\Base\) are of the form \(\hat{L}_\base \times_{L_\base} \pi^{-1}(L_\base)\).  The vertical tangent bundle~\(\Tvert\Tot\) of \(p\colon \Tot\to\Base\) is the pull-back of~\(\Tvert \Fol'\) along the coordinate projection \(\Tot\to\Base'\), \((\gamma,\base')\mapsto \base'\), that is, \(\Tvert\Tot \cong \Tot\times_{\Base'}\Tvert \Fol'\).

Theorem~\ref{the:tangent_dual_symmetric} provides a symmetric Kasparov dual for~\(\Tot\), involving \(\Dual = \CONT_0(\Tvert\Tot)\).  We are going to examine the special case where the coefficient \(\Cst\)\nb-algebras \(A\) and~\(B\) in Theorem~\ref{the:tangent_dual_symmetric} are trivial.  The first duality isomorphism yields
\[
\RK^*_\Grd(\Tot) \defeq \RKK^\Grd_*(\Tot;\UNIT,\UNIT) \cong \KK^\Grd_*(\CONT_0(\Tvert\Tot),\UNIT) \eqdef \K_*^{\Grd,\lf}(\Tvert\Tot),
\]
that is, the \(\Grd\)\nb-equivariant representable \(\K\)\nb-theory of~\(\Tot\) agrees with the \(\Grd\)\nb-equivariant locally finite \(\K\)\nb-homology of~\(\Tvert\Tot\).  Since~\(\Grd\) acts freely on~\(\Tot\) with orbit space~\(\Base'\), \(\Grd\ltimes\Tot\) is Morita equivalent to~\(\Base'\).  Hence \(\RK^*_\Grd(\Tot)\cong \RK^*(\Base')\) and the first duality isomorphism yields \(\K_*^{\Grd,\lf}(\Tvert\Tot) \cong \RK^*(\Base')\).

The second duality isomorphism yields
\[
\K_*^{\Grd,\lf}(\Tot) \defeq \KK_*^\Grd(\CONT_0(\Tot),\UNIT) \cong \KK_*^{\Grd\ltimes\Tot}\bigl(\CONT_0(\Tot), \CONT_0(\Tvert\Tot)\bigr) \eqdef \RK^*_{\Grd,\Tot}(\Tvert\Tot),
\]
that is, the \(\Grd\)\nb-equivariant locally finite \(\K\)\nb-homology of~\(\Tot\) agrees with the \(\Grd\)\nb-equivariant \(\K\)\nb-theory of~\(\Tvert\Tot\) with \(\Tot\)\nb-compact support.  Using \(\Tvert\Tot= \Tot\times_{\Base'} \Tvert\Fol'\) and the Morita equivalence \(\Grd\ltimes\Tot\sim\Base'\), we may identify
\[
\KK_*^{\Grd\ltimes\Tot}\bigl(\CONT_0(\Tot), \CONT_0(\Tvert\Tot)\bigr) \cong \KK_*^{\Base'}\bigl(\CONT_0(\Base'), \CONT_0(\Tvert \Fol')\bigr) \eqdef \K^*_{\Base'}(\Tvert \Fol')
\]
that is, we get the \(\K\)\nb-theory with \(\Base'\)\nb-compact support of the underlying space of the distribution~\(\Tvert \Fol'\) on \(\Base'\).

Assume now that~\(\Tot\) is a universal proper \(\Grd\)\nb-space.  Then the results in Section~\ref{sec:dual_EG} show that \(D\in\KK_*^\Grd\bigl(\CONT_0(\Tvert\Tot),\CONT_0(\Base)\bigr)\) is a Dirac morphism for~\(\Grd\) and that the Baum--Connes assembly map for~\(\Grd\) with coefficients~\(B\) is equivalent to the map
\begin{equation}
  \label{eq:foliation_BC_localisation}
  \K_*\bigl(\Grd\ltimes \CONT_0(\Tvert\Tot,B)\bigr) \to
  \K_*\bigl(\Grd\ltimes B\bigr)
\end{equation}
induced by~\(D\).  The second duality isomorphism combined with another natural isomorphism identifies
\[
\varinjlim_{\Other\in I_\Tot} \KK_*^\Grd\bigl(\CONT_0(\Other),B\bigr) \cong \K_*\bigl(\Grd\ltimes \CONT_0(\Tvert\Tot,B)\bigr)
\]
and identifies the map in~\eqref{eq:foliation_BC_localisation} with the index map
\[
\varinjlim_{\Other\in I_\Tot} \KK^\Grd_*(\CONT_0(\Other), B) \to \K_*(\Grd\ltimes B).
\]
By Theorem~\ref{the:dual_EG2}, these assembly maps are isomorphisms if~\(\Grd\) acts properly on~\(B\).

This leads to the question when~\(\Tot\) is universal.  Since~\(\Grd\) acts freely on~\(\Tot\), this forces the holonomy groupoid~\(\Grd\) to be torsion-free; that is, two parallel leaf paths must have the same holonomy once some finite powers of them have the same holonomy.  Another necessary condition for~\(\Tot\) to be universal is that its fibres~\(\Tot_\base\) be contractible for all \(\base\in\Base\); in fact, this condition is sufficient as well because~\(\Grd\) is Morita equivalent to an \'etale groupoid (we omit further details).  Thus~\(\Tot\) is a universal proper \(\Grd\)\nb-space if the holonomy of~\(\Fol\) is torsion-free and the holonomy coverings of the leaf pre-images \(p^{-1}(L_\base)\) are contractible.  Of course, this implies that these holonomy coverings are universal coverings.  For instance, \(\Grd\) itself is a universal proper \(\Grd\)\nb-space if and only if the foliation has torsion-free holonomy and the holonomy covers of the leaves are contractible.

\begin{example}
  Let~\(M\) be a compact, smooth, aspherical manifold and let~\(\tilde{M}\) be its universal cover.  Let \(G \defeq \pi_1(M)\) act by deck transformations on~\(\tilde{M}\) and let~\(V\) be a smooth compact manifold with a free \(G\)\nb-action.  Foliate \(\Base \defeq \tilde{M}\times_G V\) by the images in~\(\Base\) of the slices \(\tilde{M}\times \{v\}\), for \(v\in V\).  This is a foliation with contractible leaves because~\(M\) is assumed aspherical.  In this case, the morphism space~\(\Grd\) is a universal proper \(\Grd\)\nb-space.

  More generally, if the \(G\)\nb-action on~\(V\) is not free, then~\(\Grd\) will still be universal if the holonomy representation of the fundamental group of each leaf is faithful.
\end{example}

Now we allow~\(\Tot\) to be a proper \(\Grd\)\nb-equivariant bundle of smooth manifolds over~\(\Base\).  We want to show that the equivariant Euler characteristic of~\(\Tot\) is the class in \(\KK^\Grd_0(\CONT_0(\Tot),\UNIT)\) given by the family of de Rham operators along the fibres of~\(\Tot\).  It is possible to do this computation using the tangent space dual described above.  But it simplifies if we use the Clifford algebra dual used in \cites{Emerson-Meyer:Euler, Emerson-Meyer:Equi_Lefschetz}.  This does not change the result because the Euler characteristic is independent of the chosen Kasparov dual.  We just sketch the situation here in order to give an example of Euler characteristics in the context of smooth \(\Grd\)\nb-manifolds.
 
Fix an invariant metric on the vertical tangent bundle~\(\Tvert\Tot\) and form the associated bundle of vertical Clifford algebras \(\ctau(\Tot)\); this is a locally trivial bundle of finite-dimensional \(\Cst\)\nb-algebras over~\(\Tot\).  The Thom isomorphism provides an invertible element in \(\KK^{\Grd\ltimes\Tot}_*\bigl(\CONT_0(\Tvert\Tot), \ctau(\Tot)\bigr)\) (the idea in \cite{LeGall:KK_groupoid}*{Th\'eor\`eme 7.4} shows how to get this result equivariantly for groupoids).  Since we have Kasparov duality with \(\CONT_0(\Tvert\Tot)\), we also have it with the Clifford bundle instead of \(\CONT_0(\Tvert\Tot)\).  Even more, since these two duals are \(\KK^{\Grd\ltimes\Tot}\)-equivalent, the first and second duality isomorphisms translate from one to the other, even if the coefficient algebras are non-trivial bundles over~\(\Tot\).  One can check that the classes \(D\) and~\(\Theta\) for this new dual are exactly the same ones as in \cites{Emerson-Meyer:Euler, Emerson-Meyer:Equi_Lefschetz}.  The same easy computation as in the group case in~\cite{Emerson-Meyer:Euler} then shows that
\begin{equation}
  \label{eq:euler:foliation_deRham}
  \Eul_\Grd = [D_\textup{dR}] \in
  \KK^\Grd_*(\CONT_0(\Tot),\CONT_0(\Base)\bigr).
\end{equation}
Here \(D_\textup{dR}\) denotes the de Rham operator along the fibres of the anchor map \(\Tot\to\Base\).  If \(\Tot=\Grd\), then these fibres are just the holonomy covers of the leaves of~\(\Fol\), so that we get the family of de Rham operators along the leaves of the foliation.

Let
\[
\mu\colon \Ktop(\Grd) \to \K_*(\Cstr\Grd)
\]
be the Baum--Connes assembly map and let \(\varphi\colon \Grd \to \mathcal{E\Grd}\) be the classifying map of the proper \(\Grd\)\nb-space \(\Tot\defeq\Grd\).  This induces a map \(\varphi_*\colon \KK^{\Grd}_*\bigl( \CONT_0(\Tot), \CONT_0(\Base)\bigr) \to \Ktop_*(\Grd)\).  The map
\[
\mu_\Tot\defeq \mu\circ \varphi_*\colon \KK^\Grd_*\bigl( \CONT_0(\Tot),\CONT_0(\Base)\bigr)\to \K_*(\Cstr\Grd)
\]
is the \(\Grd\)\nb-equivariant index map for~\(\Grd\).  It maps \(\Eul_\Grd\in \KK^\Grd_0\bigl(\CONT_0(\Tot), \CONT_0(\Base)\bigr)\) to the equivariant index in \(\K_0(\Cstr\Grd)\) of the family of de Rham operators on the holonomy covers of the leaves of the foliation by~\eqref{eq:euler:foliation_deRham}.

If~\(\Lambda\) is a \(\Grd\)-invariant transverse measure on the foliation, then we may pair it with classes in \(\K_0(\Cstr\Grd)\) to extract numerical invariants.  For the equivariant Euler characteristic of the foliation, this yields the alternating sum of its \(L^2\)\nb-Betti numbers,
\begin{equation}
  \label{eq:euler:foliation_deRham_ltwo_version}
  \Eul_\Tot^{(2)} = (\Lambda\circ\mu_\Tot)
  \bigl([D_\textup{dR}]\bigr) = \sum_i (-1)^{i}\beta_{L^2}^{i}.
\end{equation}

\section{Conclusion and outlook}
\label{sec:conclusion}

We have constructed analogues of the first and second Poincar\'e duality isomorphisms in~\cite{Kasparov:Novikov} for proper groupoid actions on \(\Cst\)\nb-algebra bundles over possibly non-compact spaces.  In the simplest case of a smooth manifold without any groupoid action, this generalises familiar isomorphisms \(\K_*^\lf(\Tvert M) \cong \RK^*(M)\) between the locally finite \(\K\)\nb-homology of the tangent bundle and the representable \(\K\)\nb-theory of~\(M\), and \(\K^*_M(\Tvert M) \cong \K_*^\lf(M)\) between the \(\K\)\nb-theory of the tangent bundle with \(M\)\nb-compact support and the locally finite \(\K\)\nb-homology of~\(M\).

These duality isomorphisms follow from the existence of a symmetric Kasparov dual.  We have constructed such a dual for bundles of smooth manifolds, equivariantly with respect to a smooth proper groupoid action.  Furthermore, we have extended the two duality isomorphisms by allowing strongly locally trivial bundles.

A different construction in~\cite{Emerson-Meyer:Euler} provides a dual for a finite-dimensional simplicial complex.  This is, in fact, a symmetric Kasparov dual.  We plan to discuss this elsewhere, together with a discussion of the new aspects that appear for bundles, namely, the singularities that necessarily appear when triangulating bundles of smooth manifolds.  More generally, we may replace simplicial complexes by stratified pseudomanifolds.  A duality isomorphism in this setting was recently established by Claire Debord and Jean-Marie Lescure (see~\cite{Debord-Lescure:K-duality_stratified}).  It remains to establish that this is another instance of a symmetric dual.

The two duality isomorphisms discussed in this article are related to the dual Dirac method and the Baum--Connes Conjecture.  Roughly speaking, the duality shows that the approach to the Baum--Connes Conjecture by Baum, Connes and Higson in~\cite{Baum-Connes-Higson:BC} via the equivariant \(\K\)\nb-homology of the universal proper \(\Grd\)\nb-space and Kasparov's approach using Dirac, dual Dirac, and the \(\gamma\)\nb-element are equivalent whenever the universal proper \(\Grd\)\nb-space has a symmetric Kasparov dual.

The second duality isomorphism reduces \(\KK\)\nb-groups to \(\K\)\nb-groups with support conditions.  This is used in~\cite{Emerson-Meyer:Correspondences} to describe equivariant bivariant \(\K\)\nb-theory groups by geometric cycles (under some assumptions).

Furthermore, we have used the duality to define equivariant Euler characteristics and Lefschetz invariants.  The construction of these invariants only uses formal properties of Kasparov theory and therefore works equally well in purely geometric bivariant theories defined using correspondences.  This seems the appropriate setting for explicit computations of such Lefschetz invariants.

\end{document}